\documentclass{article}

\usepackage[utf8]{inputenc}
\usepackage[margin=1.2in]{geometry}
\usepackage[title]{appendix}
\usepackage{amsthm}
\usepackage{url,color}
\usepackage{hyperref}
\usepackage{amsmath, amsfonts}
\usepackage{mathtools}
\usepackage{mathalfa}
\usepackage{amssymb}
\usepackage{upgreek}
\usepackage{mathrsfs}
\usepackage{relsize}
\usepackage{bm}

\usepackage[makeroom]{cancel}
\usepackage[mathscr]{eucal}
\usepackage{tikz}
\usetikzlibrary{positioning}
\numberwithin{equation}{section}
\newcommand{\blank}{\mathrel{\;\cdot\;}}

\newcommand{\dell}{\partial}

\newcommand{\ass}{\quad\hbox{as }}

\newcommand{\gradperp}{\nabla^{\perp}}

\newcommand{\R}{\mathbb R}
\newcommand{\nn}{\nabla} 
\newcommand{\ve}{\varepsilon} 
  
\newcommand{\be}{\begin{equation}} 
\newcommand{\ee}{\end{equation}}

\newcommand{\bigO}{\mathcal{O}}
\newcommand{\twopartdef}[4]
   {
    \left\{
      \begin{array}{ll}
        #1 & \text{if } #2 \\
        #3 & \text{if } #4
      \end{array}
    \right.
   }

\DeclareMathOperator{\supp}{supp}

\DeclareMathOperator{\grad}{\nabla}

\newtheorem{definition}{Definition}[section]

\newtheorem{theorem}[definition]{Theorem}

\newtheorem{corollary}[definition]{Corollary}

\newtheorem{lemma}[definition]{Lemma}

\newtheorem{remark}[definition]{Remark}

\usepackage[colorinlistoftodos]{todonotes}

\def\bcr{\begin{color}{red}}
\def\bcb{\begin{color}{blue}}
\def\ec{\end{color}}

\title{An Expanding Self-Similar Vortex Configuration for the $2$D Euler Equations}

\author{Juan D\'avila\footnote{Department of Mathematical Sciences, University of Bath, Bath, BA2 7AY, UK. \emph{Email Address}:\ {\tt jddb22@bath.ac.uk}}\ \ Manuel del Pino\footnote{Department of Mathematical Sciences, University of Bath, Bath, BA2 7AY, UK. \emph{Email Address}:\ {\tt mdp59@bath.ac.uk}}\ \ Monica Musso\footnote{Department of Mathematical Sciences, University of Bath, Bath, BA2 7AY, UK. \emph{Email Address}:\ {\tt mm2683@bath.ac.uk}}\ \ Shrish Parmeshwar\footnote{
CY Cergy Paris Université, Laboratoire Analyse, Géométrie, Modélisation, F-95302
Cergy-Pontoise, France. \emph{Email Address}:\ {\tt shrish.parmeshwar@cyu.fr}}}

\date{}

\begin{document}

\maketitle

\begin{abstract}
This paper addresses the long-time dynamics of solutions to the $2D$ incompressible Euler equations. We construct solutions with continuous vorticity $\omega_\ve(x,t)$ concentrated around points $\xi_j(t)$ that converge to a sum of Dirac delta masses as $\ve \to 0$. These solutions are associated with the Kirchhoff-Routh point-vortex system, and the points $\xi_{j}(t)$ follow an expanding self similar trajectory of spirals, with the support of the vorticities contained in balls of radius $3\varepsilon$ around each $\xi_{j}$.
\end{abstract}

\section{Introduction}\label{introduction-section}

We consider the Euler equations for an incompressible, inviscid fluid in $\mathbb{R}^2$. Written in terms of the vorticity $\omega(x,t)$ and stream function $\Psi(x,t)$, this system of equations is given by

\be \label{2d-euler-vorticity-stream}  \left\{   
 \begin{aligned}
\omega_{t}+\gradperp_{x}\Psi\cdot\nabla_{x}\omega&=0\ \ \ \ &&\text{in}\ \mathbb{R}^{2}\times (0,T),\\ \Psi= (-\Delta_{x})^{-1}&\omega\ \  &&\text{in}\ \mathbb{R}^{2}\times (0,T), \\
\omega(\cdot,0)= &\mathring{\omega}  &&\text{in}\ \mathbb{R}^{2}.
\end{aligned} \right.  \ee
The Poisson equation in \eqref{2d-euler-vorticity-stream} above implies that for suitable $\omega$, $\Psi$ is given by
$$
 \Psi(x,t)= (-\Delta_{x})^{-1}\omega (x,t)   = \frac 1{2\pi} \int_{\mathbb R^2}\log \frac 1{|x-y|} \omega (y,t)\, dy. \ \  
$$
The fluid velocity field  $v =  \nn^\perp \Psi$  is then recovered by the Biot-Savart Law:
$$
v(x,t) =  \frac 1{2\pi} \int_{\mathbb R^2} \frac {(y-x)^\perp} {|y-x|^2} \omega (y,t)\, dy \ ,
$$
where $z^{\perp}=(z_{2},-z_{1})$.

\medskip
The Cauchy problem for \eqref{2d-euler-vorticity-stream} admits a global well-posedness theory for weak solutions arising from initial data in $L^1(\R^2)\cap L^\infty (\R^2)$; see classical work by Wolibner \cite{Wolibner1933} and Yudovich \cite{Yudovich1963}. In addition, any extra regularity for the initial data $\mathring{\omega}$ is propagated and also present for the solution $\omega$. See for instance Chapter 8 in Majda-Bertozzi's book \cite{majdabertozzi}. 

\medskip
Carrying on the work started in \cite{DDMP2VP2023}, this paper is concerned with constructing solutions to \eqref{2d-euler-vorticity-stream} that exhibit nontrivial dynamics as $t\to \infty$. We are interested in solutions with a finite number of regions with highly {\em concentrated} vorticity, and in particular, of the form 
\be \label{ome}
\omega_\ve (x,t) =  \sum_{j=1}^N   \frac{m_j} {\ve^2}  \Big [  W\left ( \frac {x-\xi_j (t) } \ve   \right ) + o(1) \Big ] ,
\ee
around points $\xi_1(t), \ldots, \xi_N(t)$, where $W(y)$ is a certain, chosen radially symmetric profile with unit mass, $o(1) \to 0 $ as $\ve \to 0$ and $m_j \in \mathbb{R}$ are the circulations (or masses) of the vortices. Consequently these solutions will satisfy
\begin{equation}\label{vortices-to-deltas}
\omega_\ve(x,t)  \rightharpoonup  \omega_S(x,t) = \sum_{j=1}^N   {m_j} \delta_0  (  {x-\xi_j (t) } ) \ass \ve\to 0 ,
\end{equation}
where $\delta_0(y)$ is a Dirac mass at $0$.
%
Formally the points $\xi_j(t)$ should solve the Kirchhoff-Routh point-vortex ODE system
\begin{equation}\label{KR}
\dot \xi_j (t) = -\sum_{\ell \not= j} \frac{m_\ell}{2\pi} \frac{(\xi_j(t) -\xi_\ell (t) )^\perp}{|\xi_j(t) -\xi_\ell (t)|^2}, \quad j=1, \ldots , N.
\end{equation}
Using complex variables, $z_{j}=\xi_{j}^{1}+i\xi_{j}^{2}$, we can rewrite \eqref{KR} as
\begin{equation}\label{KRComplex}
    \overline{\frac{dz_{j}}{dt}}=\sum_{\ell\neq j}\frac{m_{\ell}}{2\pi i}\frac{1}{z_{j}-z_{l}}, \quad j=1,\dots,N.
\end{equation}
In \cite{DDMW2020} the following was proven:  {\em Given a finite time $T>0$ and a solution $\left( \xi_1 , \ldots , \xi_N  \right)$ to  System \eqref{KR} with no collisions  in $[0,T]$,  namely with
$$ \inf_{t\in [0,T],\, i\ne j} |\xi_i(t) -\xi_j(t)| >0 , $$ there exists a solution of \eqref{2d-euler-vorticity-stream} with the precise form \eqref{ome} and the specific choice  
$W(y) =  \frac 1{\pi} \frac 1{(1+|y|^2)^2} $.} 

\medskip
It was also shown in \cite{MP1993} previously that initial conditions supported on $\ve$-disks concentrated around the points $\xi_j (0)$ as bumps with masses $m_j$ produced solutions satisfied the convergence to a sum of delta masses, \eqref{vortices-to-deltas}. However, this work was without the asymptotic description \eqref{ome} near the vortices.   

\medskip
Both of these works rely crucially on the finiteness in time $T>0$  in order to get uniform estimates in $[0,T]$ as $\ve \to 0$. It is not clear that, in general, the solutions in \cite{MP1993,DDMW2020} keep the form \eqref{ome} as time evolves. Such a long time result would be desirable, as trajectories of system \eqref{KR} generically do not collide, see Chapter 4 in \cite{MarchioroPulvirentiBook}.

\medskip
This paper and similar work carried out in \cite{DDMP2VP2023} deals with the construction of solutions to 
the Euler equations \eqref{2d-euler-vorticity-stream} that persist \eqref{ome} and satisfy the convergence \eqref{vortices-to-deltas} {\em at all times}.

\medskip
The problem of understanding long-time behaviours and patterns for solutions of \eqref{2d-euler-vorticity-stream} still
has many open problems, and has been the subject of many conjectures. See for instance \cite{elgindimurraysaid} and its references.

\medskip
Clearly, in the analysis of long-time dynamics, steady solutions such as stationary solutions, travelling waves, and rotating solutions are of fundamental. See for instance \cite{ao,HmidiMateu,smets-vanschaftingen,caowei} and references therein. The problem of finding such solutions is typically reduced to solving semilinear elliptic equations.

\medskip
Part of why steady solutions are so fundamental, is that they offer hope in finding time-dependent configurations that satisfy \eqref{ome} for all time, in the form of solutions that closely resemble superpositions of steady ones. Indeed, this is the broad strategy used by the authors in \cite{DDMP2VP2023}, and in this work. However, in general, such a strategy has its own challenges, arising due to the emergence of coupling error terms at main order, and our limited understanding of the long-term behavior of perturbations to steady solutions. 

\medskip
In this paper, we will consider a configuration of three point vortices $\xi_{*j}$, $j=1,2,3$, whose trajectories are three similar spirals meaning that the whole configuration separates as a rotating and expanding triangle. In complex form the trajectories are given, for $j=1,2,3$ by
\begin{align}
    \xi_{*j}^{1}(t)+i\xi_{*j}^{2}(t)=z_{*j}(t)=z_{*j}(0)\left(1+\frac{t}{\tau}\right)^{\frac{1}{2}+i\Lambda\tau}, \quad j=1,2,3,\label{self-similar-spiral-ode-solution-complex}
\end{align}
where $\tau$ and $\Lambda$ are positive constants related to the initial data. Defining, for $i,j=1,2,3$,
\begin{align}
    L_{ij}(t)\coloneqq\left|z_{*i}(t)-z_{*j}(t)\right|=\left|\xi_{*i}(t)-\xi_{*j}(t)\right|=L_{ij}(0)\left(1+\frac{t}{\tau}\right)^{\frac{1}{2}},\label{point-vortex-pairwise-distance-function}
\end{align}
we also require
\begin{align}
    m_{1}+m_{2}+m_{3}>0, \quad m_{1},m_{2}>0, \quad m_{3}<0,\label{mass-constraint}\\
    m_{1}m_{2}+m_{2}m_{3}+m_{1}m_{3}=0,\label{harmonic-mean-condition}\\
    m_{1}m_{2}L_{12}(t)^{2}+m_{2}m_{3}L_{23}(t)^{2}+m_{1}m_{3}L_{13}(t)^{2}=0,\label{angular-momentum-condition}\\
    L_{23}(0)\neq L_{13}(0),\label{no-equilateral-triangle-condition}
\end{align}
where the $m_{j}$ are the masses of the vortices $\xi_{*j}$. For a full description of the trajectories of the corresponding point vortex solution to Kirchhoff-Routh system \eqref{KR}, see \cite{aref1979,novikovsedov,aref}. We also provide more detail in Section \ref{modified-vortices}.

\medskip
We now describe the steady state that we will use as a model for a single vortex in this paper. Let $\Gamma(x)$ be the radial classical solution of the problem
\begin{align}
\Delta_{x} \Gamma+\Gamma^{\gamma}_{+}=0 \ \text{on}\ \mathbb{R}^{2} , \quad \{ \Gamma>0\} = B_1
\label{power-semilinear-problem-R2}
\end{align}
Let $M$ be the mass, or circulation, of this vortex, that is
\begin{align}
M\coloneqq\int_{\mathbb{R}^{2}}\Gamma^{\gamma}_{+}dx.\label{vortex-mass-definition}
\end{align}
The main result of this paper is the construction of an initial vorticity of three compactly supported, $\varepsilon$-concentrated functions that generate a  time evolved vorticity that remains localised around three points $\xi_{*j} (t)$, $j=1,2, 3$, at all times $t$. Suppose $\left(m_{1},m_{2},m_{3}\right)$ are three masses for a solution of the form \eqref{self-similar-spiral-ode-solution-complex} to \eqref{KR} satisfying \eqref{mass-constraint}--\eqref{no-equilateral-triangle-condition}. We state our result in the following form

\begin{theorem}\label{teo1} Let $\gamma\geq19$. There exists a constant $C>0$ such that for all $\varepsilon>0$ small enough, and all $T_0>0 $ large enough, there exists a solution to \eqref{2d-euler-vorticity-stream}, $(\omega_\ve , \Psi_\ve)$, on the whole interval $[T_0,\infty)$, with
\begin{equation}\label{vorticityexp}
\omega_\ve (x,t) =  \varepsilon^{-2}\sum_{j=1}^{3}\frac{m_{j}}{M}\left(\Gamma\left(\frac{x-\xi_{j}(t)}{\varepsilon}\right)\right)^{\gamma}_{+} + \varepsilon^{-2} \phi_\ve (x,t),
\end{equation}
where we have the decomposition of the trajectories $\xi_{j}=\xi_{*j}+\tilde{\xi}_{j}$ with $\xi_{*j}$ defined in complex variables in \eqref{self-similar-spiral-ode-solution-complex}, and $\tilde{\xi}_{j}$ functions $[T_{0},\infty):\rightarrow\mathbb{R}^{2}$, that satisfy, for all $t\in\left[T_{0},\infty\right)$ and $j=1,2,3$, 
\begin{align*}
|t^{\frac{3}{2}}\tilde{\xi}_{j}\left(t\right)|+|t^{\frac{5}{2}}\dot{\tilde{\xi}}_{j}\left(t\right)|&\leq C\ve^{4-\sigma},
\end{align*}
for some small $\sigma>0$.

\medskip
The function $\phi_\ve $ in \eqref{vorticityexp} is continuous, and satisfies
$$
|\phi_\ve (x,t) | \leq \frac{C\ve^{1-\sigma }}{t^{\frac{1}{2}-\sigma}} , \quad \|\phi_\ve (\cdot,t)\|_{L^{2}\left(\mathbb{R}^{2}\right)}\leq \frac{C\ve^{2}}{t},  \quad \text{for all } t >T_0,
$$
for any $\sigma >0$ small. The support of $\phi_{\varepsilon}$ satisfies
$$
{\mbox {supp}} \, \phi_\ve (x,t) \subseteq \bigcup_{j=1}^{3}B_{3\ve } \left(\xi_{*j}\right),  \quad \text{for all } t >T_0.
$$
\end{theorem}

The functions $ \Psi_\ve (x, t+T_0)$,  $ \omega_\ve (x, t+T_0)$,  $t>0$ form a solution of \eqref{2d-euler-vorticity-stream} in $[0,\infty)$ and has the expected properties in the whole interval $[0,\infty)$. As in \cite{DDMP2VP2023}, the condition $\gamma\geq19$ is a technical hypothesis required in the proof of some a priori estimates. 

\medskip
By $T_{0}>0$ large enough, in particular, we require $L_{ij}\left(T_{0}\right)\geq 2\max_{k,l=1,2,3}{L_{kl}\left(0\right)}$, for all $i,j=1,2,3$, $i\neq j$ with $L_{ij}\left(t\right)$, defined in \eqref{point-vortex-pairwise-distance-function}, and we also have the freedom to enlarge the size of $T_{0}$ in order to satisfy an estimate in the proof of Lemma \ref{L2-interior-a-priori-estimate-lemma}.

\medskip
We remark here that in \cite{novikovsedov}, the existence of expanding self similar spirals analogous to \eqref{self-similar-spiral-ode-solution-complex}, but with $4$ or $5$ points are constructed. Nothing in our proof relies on only having $3$ vortices in the configuration, so our method of construction goes through for these $4$ and $5$ point configurations without any extra mathematical difficulty, as long the collection of masses satisfy analogous constraints to \eqref{mass-constraint}--\eqref{no-equilateral-triangle-condition}.
 
\medskip
Our result is an example of a solution to \eqref{2d-euler-vorticity-stream} where one has precise information on the long-time behaviour, a type of result that, as mentioned above, remains rare for the $2D$ Euler equations. Whilst we construct initial conditions that allow us to exhibit the long time behavior claimed in Theorem \ref{teo1}, we do not claim that perturbing our initial data will preserve this structure. Physically, experiments and simulations suggest that most solutions to \eqref{2d-euler-vorticity-stream} evolve over infinite time towards simpler dynamics. The following informal conjecture has been formulated for the $2D$ Euler equations in \cite{sve}, \cite{sch}: as $t\to \pm \infty$ generic solutions experience loss of compactness. 
Significant  confirmations of this conjecture have emerged   near specific steady states. Important stability results for shear flows have been found in \cite{Bedrossian,Ionescu, ionescu2,masmoudizhao}. Stability results for fast-decay, radial steady states have been found \cite{Bedrossian,Ionescu,choi-lim}. See also \cite{elgindimurraysaid} for discrete symmetry configurations, and \cite{abe-choi,bnl} for travelling waves.

\medskip
Well-known non-smooth standard flows are vortex patches. These are solutions with discontinuous but bounded and integrable vorticities. They also enjoy additional properties when compared to the more regular desingularized solutions we construct in this work. For vortex patches, there are several long-time results. In \cite{Zbarsky}, a global solution to \eqref{2d-euler-vorticity-stream} is found, consisting of vortex patches approximating the expanding self similar solution to \eqref{KR} with $N=3$ given by \eqref{self-similar-spiral-ode-solution-complex} that we also consider in this paper. Thus, one may naturally compare the patch solutions constructed in \cite{Zbarsky} and the solutions constructed in the present work. Apart from a difference in regularity, it is important to note that in \cite{Zbarsky}, each patch may grow in size at a rate $t^{\frac{1}{4}^+}$, compensated by the fact that they separate from each other at a rate $t^{\frac{1}{2}}$. This seems crucial in the analysis. See also \cite{serfati}. Our solutions also separate at a rate of $t^{\frac{1}{2}}$, however as noted in Theorem \ref{teo1}, we find a uniform-in-time bound on the supports around the centre of each region of concentrated vorticity, a fundamentally qualitative difference to \cite{Zbarsky}.

\medskip
In addition to the above discussion, existence of co-rotating and counter-rotating pair vortex patches was shown in \cite{HmidiMateu}. Logarithmic spiraling solutions were also recently constructed in \cite{Jeongsaid}. For other related results, we refer to \cite{ej1} and to the results cited within. We finally mention the interesting recent paper 
\cite{HmidiMasmoudi} where solutions in the form of two vortex pairs leapfrogging each other are constructed with the use of KAM theory.  

\medskip
We discuss the scheme of the proof in the next section.

\section{Scheme of the proof}\label{scheme}

The strategy we adopt to construct the  solution predicted by Theorem \ref{teo1}  is to first  produce a solution to the Euler equations \eqref{2d-euler-vorticity-stream} on a finite and arbitrarily large time interval $[T_0, T]$, for  $T \gg T_0$. This solution is built as a perturbation of the sum of three $\varepsilon$-concentrated vortices traveling along the similar spirals discussed in the previous section. The perturbation will be small and decay in time.  To conclude our argument we apply the Ascoli-Arzel\'a Theorem for a sequence of times $(T_n)_n$, with $T_n\to \infty$ as $n \to \infty$ and show that the limit function gives the desired solution in the whole interval  $[T_0, \infty)$. We do this in Section \ref{conclusion}.

\medskip
The main step in our proof is the construction of the solution to the Euler equations \eqref{2d-euler-vorticity-stream} on $[T_0,T]$ as described above. 

\medskip
To do this we begin as mentioned above, by perturbing an ansatz of three $\varepsilon$-concentrated power nonlinearity vortices defined in \eqref{power-semilinear-problem-R2} and Section \ref{power-nonlinearity-section}. To construct a good first approximation therefore requires a good understanding of the linear operator obtained by linearising around the power nonlinearity vortices. This information is the content of Appendix \ref{appendix-section}. The explanation of this procedure is the content of Section \ref{appro}, and is carried out in detail in \ref{first-approximation-section}. We note that at the level of the first approximation we must deal with the fact that the point-vortex solutions to \eqref{KR} given by \eqref{self-similar-spiral-ode-solution-complex} only separate at $t^{\frac{1}{2}}$ as opposed to the rate of $t$ which we saw in \cite{DDMP2VP2023}. This leads to less decay in time when constructing the first approximation, and requires a precise tracking of the modal structure of the approximation in order to guarantee enough time decay to perform the bootstrapping argument that lets us construct solutions on $\left[T_{0},T\right]$. For example, see Remark \ref{first-approximation-zero-mass}.

\medskip
In addition, due to this weaker separation rate, it is crucial that we do not lose any time decay when inverting the linearised operator around the solutions \eqref{self-similar-spiral-ode-solution-complex}, like we do for the analogous linearised operator in \cite{DDMP2VP2023}, and thus it is necessary to carry out a precise analysis into the longtime asymptotics for the linearisation around the solutions \eqref{self-similar-spiral-ode-solution-complex}, performed in Section \ref{point-vortex-trajectory-section}.

\medskip
We then construct solutions to \eqref{2d-euler-vorticity-stream} on $\left[T_{0},T\right]$ in a manner closely related to the fixed point argument used in \cite{DDMW2020}. The explanation of this is the content of Section \ref{cs} and is carried out in Section \ref{constructing-full-solution-section}.

\subsection{Power Nonlinearity Vortex}\label{power-nonlinearity-section}
In Section \ref{introduction-section} we introduced the function  $\Gamma$ solving \eqref{power-semilinear-problem-R2}, which we can write as
\begin{align}
    \Gamma(x)=\twopartdef{\nu(|x|)}{|x|\leq1}{\nu'(1)\log{|x|}}{|x|>1},\label{Gamma-piecewise-definition}
\end{align}
    where $\nu (|x|)$ is the unique positive, radial, ground state solving
$$
    \Delta_{x}\nu+\nu^{\gamma}_{+}=0 \ \text{on}\ B_{1}.
    $$
Recalling that 
$M=\int_{\mathbb{R}^{2}}\Gamma^{\gamma}_{+}dx$
(see \eqref{vortex-mass-definition}),
a direct computation gives
\begin{align}
    M=2\pi\left|\nu'(1)\right|,\label{mass-in-terms-of-nu}
\end{align}
where we note that $\nu'(1)<0$.

\medskip

For any $\varepsilon >0$, rescaling $\Gamma$ by $\varepsilon$ leaves it $C^1$ across the boundary of the ball radius $\varepsilon$. Thus, the rescaled $\Gamma$ solve the following family of semilinear elliptic equations
 \begin{equation}\label{Gamma-epsilon-elliptic-equation}
\begin{aligned}
    \Delta_{x}\Gamma\left(\frac{x}{\varepsilon}\right)+\varepsilon^{-2}\left(\Gamma\left(\frac{x}{\varepsilon}\right)\right)^{\gamma}_{+}=0,
\end{aligned}
\end{equation}
which are solutions with concentrated vorticity in the sense of \eqref{ome}--\eqref{vortices-to-deltas} for each $\varepsilon>0$ small enough. Observe that for all $\varepsilon>0$
\begin{align}
  \varepsilon^{-2}\int_{\mathbb{R}^{2}}\left(\Gamma\left(\frac{x}{\varepsilon}\right)\right)^{\gamma}_{+}dx=\int_{B_{1}}\Gamma^{\gamma}_{+}\left(y\right)\  dy=M\label{epsilon-vortex-mass}.
\end{align}

\subsection{Vortex trajectories}
\label{modified-vortices}
From \cite{aref1979,aref}, and as mentioned in the discussion \eqref{self-similar-spiral-ode-solution-complex}--\eqref{no-equilateral-triangle-condition}, we know that there is a solution $\left(\xi_{*1}(t),\xi_{*2}(t),\xi_{*3}(t)\right)$ to the system \eqref{KR} for $N=3$, given by, in complex form, 
\begin{align*}
    z_{*j}(t)=z_{*j}(0)\left(1+\frac{t}{\tau}\right)^{\frac{1}{2}+i\Lambda\tau},
\end{align*}
so that
\begin{align}
    z_{*j}(t)=\xi_{*j}^{1}(t)+i\xi_{*j}^{2}(t),\label{xi-star-j-definition}
\end{align}
where the $\xi_{*j}^{i}(t)$ are the vector components of $\xi_{*j}(t)$. 

\medskip

Following the construction in \cite{aref1979}, the mass constraints \eqref{mass-constraint}, the harmonic mean constraint \eqref{harmonic-mean-condition}, and the angular momentum constraint \eqref{angular-momentum-condition} can be used to show that the lengths $L_{12}(t),L_{23}(t),L_{13}(t)$ defined in \eqref{point-vortex-pairwise-distance-function} are all proportional to one another. Let $A(t)$ be the area of the triangle formed at time $t$ by the vortices. If we suppose that $L_{23}(t)>L_{12}(t)>L_{13}(t)$, and that the triangle the point vortices form has positive orientation, we can conclude the right hand side of the ODEs
\begin{align*}
    \frac{d}{dt}\left(L_{12}^{2}\right)&=\frac{2}{\pi}m_{3}A(t)\left(\frac{1}{L_{23}^{2}}-\frac{1}{L_{13}^{2}}\right)\\
    \frac{d}{dt}\left(L_{23}^{2}\right)&=\frac{2}{\pi}m_{1}A(t)\left(\frac{1}{L_{13}^{2}}-\frac{1}{L_{12}^{2}}\right)\\
    \frac{d}{dt}\left(L_{13}^{2}\right)&=\frac{2}{\pi}m_{2}A(t)\left(\frac{1}{L_{12}^{2}}-\frac{1}{L_{23}^{2}}\right)
\end{align*}
are all positive constants, once again using the mass constraints \eqref{mass-constraint}, and the fact that $A(t)$ is proportional to the square of all the side lengths. The constant $\tau$ can then be calculated by integrating any three of the above ODEs:
\begin{align*}
    \frac{1}{\tau}=\frac{2}{\pi}m_{3}\frac{A(0)}{L_{12}^{2}(0)}\left(\frac{1}{L_{23}^{2}(0)}-\frac{1}{L_{13}^{2}(0)}\right).
\end{align*}
\medskip
From there, we can show that 
\begin{align}
    \Lambda=\frac{\left(m_{1}+m_{2}\right)^{3}+m_{1}^{3}\left(1+\Pi^{-1}\right)+m_{2}^{3}\left(1+\Pi\right)}{4\pi\left(m_{1}+m_{2}\right)^{2}L_{12}(0)^{2}},\label{omega-definition}
\end{align}
where
\begin{align*}
    \Pi=\frac{m_{1}}{m_{2}}\frac{L_{13}(0)^{2}}{L_{23}(0)^{2}},
\end{align*}
and finally we use the length constraint \eqref{no-equilateral-triangle-condition} to ensure that the initial triangle formed by the point vortices is not an equilateral, which then automatically excludes the case of rigid motion in the form of rotation with no expansion.
\subsection{Approximate solution on $[T_0,T]$}\label{appro}

Our initial step in developing a two-vortex pair solution to the Euler equations \eqref{2d-euler-vorticity-stream} is to create a reasonably accurate first approximation within a finite time interval $\left[T_{0}, T\right]$. We treat $T$ as an arbitrarily large parameter, with the intention of eventually taking $T\to\infty$ to construct our complete solution over the interval $\left[T_{0},\infty\right)$. 
\medskip
 
We  fix $T>0$, a large time satisfying
\begin{align}
    T\gg T_{0}.\label{T-definition}
\end{align}
To describe the approximate solution  we first introduce, for $j=1,2,3$, 
\begin{align}
     \xi_{j}(t)=(\xi^{1}_{j}(t),\xi^{2}_{j}(t)),\label{vortex-trajectories-definition-new}
\end{align}
time dependent vector fields, moving in $\mathbb{R}^{2}$. We assume they have the form  
\begin{align}
    \xi_{j}(t)=\xi_{*j}(t) +\tilde{\xi}_{j} (t).\label{point-vortex-trajectory-function-decomposition-new}
\end{align}
The first terms $\xi_{*j}(t)$ are explicit and given in complex form by \eqref{self-similar-spiral-ode-solution-complex}. We can see that they satisfy
\begin{equation}\label{xi0-bound}
\begin{aligned}
    \|(1+t)^{-\frac{1}{2}}\xi_{*j}\|_{\infty}+\|(1+t)^{\frac{1}{2}}\dot{\xi}_{*j}\|_{\infty}&\leq C_{1},
\end{aligned} 
\end{equation}
where the norm is the $L^{\infty}$ norm on functions $[0,\infty)\rightarrow\mathbb{R}^{2}$. The points 
$\tilde{\xi}_{j}(t)=\left(\tilde{\xi}^{1}_{j}(t),\tilde{\xi}^{2}_{j}(t)\right)$ are free parameters to adjust at the end of our construction to obtain a true solution on $[T_0,T]$. For the
moment, we ask that they are continuous functions on $[T_0, T ]$ for which $\dot{\tilde {\xi}}_{j}$ exists and such that
\begin{equation}\label{tildexi-bound}
\begin{aligned}
    \|t^{\frac{3}{2}}\tilde{\xi}_{j}\|_{\left[T_{0},T\right]}+\|t^{\frac{5}{2}}\dot{\tilde{\xi}}_{j}\|_{\left[T_{0},T\right]}\leq\varepsilon^{4-\sigma},\ \ \tilde{\xi}_{j}\left(T\right)=0,
\end{aligned}
\end{equation}
for some $\sigma >0$ arbitrarily small, and where the norm is the $L^{\infty}$ norm on functions $\left[T_{0},T\right]\rightarrow\mathbb{R}^{2}$.

\medskip
Let
\begin{align}
    \xi(t)=\left(\xi_{1}(t),\xi_{2}(t),\xi_{3}(t)\right),\label{point-vortex-decomposition-vector-form}
\end{align}
with an analogous definition for $\xi_{*}(t)$ and $\tilde{\xi}(t)$. We start our construction with a stream function $\Psi_0 [\xi] (x,t)$ defined as the sum of the stream function of three vortices travelling according to the point vortex solution described in Section \ref{modified-vortices}, denoted by $\Psi_{j}$, $j=1,2,3$. We proceed in a similar way with the vorticity. If the solution to the point vortex system $\xi_{*j}(t)$, $j=1,2,3$, have masses $m_{j}$, define
\begin{align}
    \kappa_{j}=\frac{m_{j}}{M},\label{kappa-j-definition}
\end{align}
where $M$ is defined in \eqref{vortex-mass-definition}.

\medskip
Then we write
\begin{equation} \label{psi-0-definition}
\begin{aligned}
    \Psi_{0}[\xi](x,t)&=\sum_{j=1}^{3}\kappa_{j}\Gamma\left(\frac{x-\xi_{j}(t)}{\varepsilon}\right),\\
    \omega_{0}[\xi](x,t)&=\varepsilon^{-2}\sum_{j=1}^{3}\kappa_{j}U\left(\frac{x-\xi_{j}(t)}{\varepsilon}\right),
\end{aligned}
\end{equation}
with $\Gamma$ defined in \eqref{power-semilinear-problem-R2}, and
\begin{align}
    U(x)=\left(\Gamma(x)\right)^{\gamma}_{+}.\label{power-nonlinearity-vorticity-definition}
\end{align}
For future convenience, we define
\begin{align}
    \Gamma_{j}(x,t)=\Gamma\left(\frac{x-\xi_{j}(t)}{\varepsilon}\right)\label{gamma-j-u-j-definition}
\end{align}
and similarly for $U_{j}(x,t)$.

\medskip
We introduce the Euler operators
\begin{equation}\label{euler-operators}
\begin{aligned}
    E_{1}\left[\Psi , \omega \right]&=\dell_{t}\omega(x,t)+\gradperp_{x}\Psi(x,t)\cdot\nabla_{x}\omega(x,t),\\
    E_{2}\left[\Psi, \omega\right]&=\Delta_{x}\Psi(x,t)+\omega(x,t).
\end{aligned}
\end{equation}
The error of approximation for $\left(\Psi_0 [\xi] (x,t), \omega_0 [\xi] (x,t) \right)$ is such that, for some constant $C>0$, it satisfies
$$
|E_1 \left[\Psi_0 [\xi], \omega_0 [\xi]\right] | =\bigO{\left(\varepsilon^{-2}\right)}, \quad {\mbox {supp}} \, \left ( E_1 \left[\Psi_0[\xi], \omega_0 [\xi]\right] \right) \subseteq \bigcup_{j=1}^{3}B_{\varepsilon}\left(\xi_{j}(t)\right).
$$
By \eqref{Gamma-epsilon-elliptic-equation}, we have
$$
E_2  \left[\Psi_0 [\xi], \omega_0 [\xi] \right] =0.$$
In Section \ref{first-approximation-section} we prove there exists an improved approximation $(\Psi_* [\xi] , \omega_*[\xi]  )$ 
such that the new error $E_1 \left[\Psi_* , \omega_* \right]$ satisfies, for some constant $C>0$,
$$
|E_1 \left[\Psi_* [\xi], \omega_* [\xi]\right] | \leq C \frac{\ve}{(1+t)^{\frac{5}{2}}}, \quad {\mbox {supp}} \, \left ( E_1 \left[\Psi_* [\xi], \omega_* [\xi] \right] \right) \subseteq \bigcup_{j=1}^{3}B_{\ve} (\xi_{j}(t)).
$$
Also,
$$
|E_2 \left[\Psi_* [\xi], \omega_* [\xi]\right] | \leq C \frac{\ve^{6}}{(1+t)^4}.
$$

\medskip
A crucial step in the construction of the improved approximation is to recognize the very special structure of the principal parts of the linear terms in $E_1 \left[\Psi_0 [\xi] + \Psi , \omega_0 [\xi]+ \omega \right]$. This structure manifests on the right hand side of $\mathbb{R}^{2}$ in the form of a linear elliptic problem, in coordinates $y=\varepsilon^{-1}\left(x-\xi_{j}(t)\right)$, as
\begin{align*}
    \gradperp_{y}\left(\Gamma(y)\right)\cdot\nabla_{y}\left(\mathscr{L}\left(\psi\right)\right)=G,
\end{align*}
where $\mathscr{L}$ is the linearized operator around the power nonlinearity vortex, for some right hand side $G$. This operator can then be rewritten in polar coordinates $y=re^{i\theta}$ as
\begin{align*}
    \frac{\Gamma'(r)}{r}\frac{\dell}{\dell \theta}\left(\mathscr{L}\left(\psi\right)\right)=G.
\end{align*}
Thus, if we can ensure the right hand side has no mode $0$ terms in its Fourier expansion, we can solve this problem via a set of ODEs for each Fourier mode.

\medskip
This allows us to improve the initial approximation $(\Psi_0 [\xi], \omega_0 [\xi])$ by solving  linear elliptic problems. This can be done under the assumption  that $\tilde \xi$ satisfies \eqref{tildexi-bound}. The approximate vorticity $\omega_* [\xi]$ is found of the form
\begin{equation}\label{omega*}
\omega_* = \omega_0 + \ve^{-2} \sum_{j=1}^{3}\kappa_{j}\phi_{j}, 
\end{equation}
so that $\phi_{j}$ has support contained in $B_{\varepsilon}(\xi_{j}(t))$, and satisfies the bound
$$
\| \phi_j \|_{L^\infty (\R^2 )} \leq  \frac{C\ve^2}{t}.$$

\subsection{Construction of a solution}\label{cs}

Next, let $K>0$ be a constant. Define cutoffs by
\begin{align}
    \eta_{j}(x,t)\coloneqq\eta_{0}\left(\frac{\left|x-\xi_{j}(t)\right|}{K\left(1+\frac{t}{\tau}\right)^{\frac{1}{2}}}\right),\label{cutoffs-definition-1}
\end{align}
where $\eta_{0}(r)$ is a cutoff that is $1$ on $r\leq1$ and $0$ on $r\geq2$.  We recall that
\begin{align}
    \supp U_{j}\subset B_{\varepsilon}\left(\xi_{j}(t)\right).\label{first-approximation-main-order-vorticity-support}
\end{align}
We can choose $K$ as
\begin{align}
    K=L_{13}(0)<L_{12}(0)<L_{23}(0),\label{K-definition-new}
\end{align}
as per the discussion in Section \ref{modified-vortices}. The choice of $K$ in \eqref{K-definition-new} ensures $\eta_{j}\equiv1$ on $\supp{U_{j}}$ for $j=1,2,3$, and the choice of $T_{0}$, given in Theorem \ref{teo1}, so that $L_{kl}\left(T_{0}\right)\geq2L_{23}(0)$ for $k,l=1,2,3$, $k<l$, ensures  $\eta_{j}\equiv0$ on $\supp{U_{k}}$ where $k\neq j$. 

\medskip
We look for a solution to the Euler equations
\begin{equation}\label{eu-finiteinterval}
E_1 [\Psi , \omega] (x,t) = E_2 [\Psi , \omega ] (x,t) = 0 \quad (x,t) \in \R^2 \times [T_0 , T]
\end{equation}
(see \eqref{euler-operators} for the definition of $E_1$ and $E_2$) of the form
$$
\Psi (x,t) = \Psi_* (x,t) + \psi_* (x,t), \quad \omega (x,t) = \omega_* (x,t) + \phi_* (x,t),
$$
where $\psi_*$ and $\phi_*$ are small corrections of the previously found first approximations. We decompose these remainders as follows
\begin{align*}
    \psi_{*}&=\sum_{j=1}^{3}\kappa_{j}\eta_{j}\psi_{*j}\left(\frac{x-\xi_{j}(t)}{\varepsilon},t\right)+\psi^{out}_*, \\
    \phi_{*}&=\varepsilon^{-2}\sum_{j=1}^{3}\kappa_{j}\phi_{*j}\left(\frac{x-\xi_{j}(t)}{\varepsilon},t\right).
\end{align*}
In terms of 
$$
\psi_*^{in} = (\psi_{*1},\psi_{*2},\psi_{*3}), \quad \psi^{out}_*, \quad \phi_*^{in} = (\phi_{*1},\phi_{*2},\phi_{*3})
$$
the Euler equations become
\begin{align*}
    E_{1}\left[\Psi , \omega \right]&=\varepsilon^{-4}\sum_{j=1}^{3}\kappa_{j}E_{*j1}\left(\phi_{*j},\psi_{*j},\psi_*^{out};\xi\right)=0, \\\
    E_{2}\left[\Psi , \omega \right]&=E_{*2}^{out}\left(\psi_*^{in},\psi_*^{out};\xi\right)=0,
    \end{align*}
where, recalling \eqref{omega*},
\begin{align*}
    E_{*j1}\left(\phi_{*j}, \psi_{*j}, \psi_*^{out}\right) &=\varepsilon^{2}\dell_{t}\left(U_j + \phi_j +\phi_{*j}\right)\\
    &+\varepsilon^{2}\gradperp_{x}\left(\Psi_{*}+ \kappa_{j}\psi_{*j}+ \psi_*^{out}\right)\cdot\nabla_{x}\left(U_j + \phi_j +\phi_{*j}\right),
\end{align*}
\begin{align*}
    E_{*2}^{out} \left( \psi_*^{out},\phi_{*}^{in}, \psi_*^{in} \right)&=\Delta_{x}\psi_*^{out}+\sum_{j=1}^{3}\kappa_{j}\left(\psi_{*j}\Delta_{x}\eta_{j}+2\nabla_{x}\eta_{j}\cdot\nabla_{x}\psi_{*j}\right)+ E_2 [\Psi_*, \omega_*].
\end{align*}
The
inner-outer gluing method consists of finding $\psi_*^{in}$, $\psi_*^{out}$, $\phi_*^{in}$ to satisfy the following system of coupled equations
\begin{equation}\label{inj}
\begin{aligned}
    E_{*j1} \left(\phi_{*j}, \psi_{*j}, \psi_*^{out}\right)&= 0 \quad {\mbox {in}} \quad \R^2 \times [T_0, T], \quad 
    \phi_{*j} (\cdot , T) = 0  \quad  {\mbox {in}} \quad \R^2 \\
    -\Delta_{x}\psi_{*j} & = \phi_{*j} \quad {\mbox {in}} \quad \R^2 \times [T_0, T]
\end{aligned}
\end{equation}
for $j=1,2,3,$ and
\begin{equation}\label{out}
E_{*2}^{out} \left(\psi_*^{out} , \phi_*^{in} , \psi_*^{in} \right) = 0 \quad {\mbox {in}} \quad \R^2 \times [T_0, T].
\end{equation}

A solution of the system \eqref{inj}-\eqref{out}, denoted as $(\phi_*^{in}, \psi_*^{in}, \psi_*^{out})$, can be used to construct a solution for equation \eqref{eu-finiteinterval} by a straightforward addition. If we ensure that the remainders $\psi_*$ and $\phi_*$ are suitably small, which can be achieved by adjusting $\tilde \xi$ in \eqref{point-vortex-trajectory-function-decomposition-new} suitably, we can obtain the desired solution. To do this, we will formulate the system as a fixed-point problem for a compact operator within a ball of an appropriate Banach space. We will then find a solution using a degree-theoretical argument, which entails establishing a priori estimates for a homotopical deformation of the problem into a linear one. 

\medskip
It should be noted that in order to construct a solution on $\left[T_{0},T\right]$ we require good a priori estimates on a certain linearized operator. For this we need a lower bound on a quadratic form that formally satisfies the structure
\begin{align}
    \int_{\mathbb{R}^{2}}\left(\frac{\Phi^{2}}{\mathfrak{K}}-\Phi\left(-\Delta\right)^{-1}\left(\Phi\right)\right),\label{quadratic-form-discussion}
\end{align}
for some fixed profile $\mathfrak{K}$ and some $\Phi$. Bounds of this nature are also important in \cite{DDMW2020, gallay1, gallay2}. Crucially, in those works, $\mathfrak{K}$ is positive on all of $\mathbb{R}^{2}$, so the integral makes sense with some decay assumptions on $\Phi$. However, the natural choice for $\mathfrak{K}$ in our case would correspond to $\gamma \Gamma^{\gamma-1}_{+}$, a function with compact support. This was, in essence, the same difficulty the authors encountered when constructing the solutions presented in \cite{DDMP2VP2023}. Therefore, a large part of Section \ref{constructing-full-solution-section} is devoted to obtaining lower bounds on a quantity that has similar structure to \eqref{quadratic-form-discussion} using ideas heavily inspired by the techniques developed in \cite{DDMP2VP2023}. These bounds are then used to obtain a priori estimates on the aforementioned linearized transport operator, which allows us to construct our solution on $\left[T_{0},T\right]$. 

\medskip
The above construction gives us a sequence of solutions indexed by $T$, for arbitrarily large $T$, satisfying \eqref{inj}, and good time decay estimates on $\left[T_{0},T\right]$. To construct a solution as stated in Theorem \ref{teo1}, we then use this time decay to obtain equicontinuity for this sequence of solutions on $\left[T_{0},T\right]$, which along with uniform boundedness, allows us to extract a subsequence, whose uniform limit solves \eqref{2d-euler-vorticity-stream} on $\left[T_{0},\infty\right)$, as desired.

\section{Linearisation Around Point Vortices}\label{point-vortex-trajectory-section}
In this section we analyse the linearised operator around the self similar spiral solution to the $3$ point vortex ODE described in \eqref{self-similar-spiral-ode-solution-complex}--\eqref{no-equilateral-triangle-condition} and Section \ref{modified-vortices}. We know that the vortex trajectories $\left(z_{*1},z_{*2},z_{*3}\right)$ can be written in complex form as
\begin{align*}
    z_{*j}(t)=z_{*j}\left(1+\frac{t}{\tau}\right)^{\frac{1}{2}+i\Lambda\tau},
\end{align*}
 and satisfy the system \eqref{KRComplex}, which can also be written as
\begin{align*}
    \frac{dz_{*j}}{dt}=-\frac{1}{2\pi i}\sum_{i\neq j}\frac{m_{i}\left(z_{*j}-z_{*i}\right)}{\left|z_{*j}-z_{*i}\right|^{2}}.
\end{align*}
To first formally calculate the linearised operator, for $j=1,2,3$, consider the quantity
\begin{align}
    -\frac{1}{2\pi i}\sum_{i\neq j}\frac{m_{i}\left(z_{*j}-z_{*i}+\zeta_{j}-\zeta_{i}\right)}{\left|z_{*j}-z_{*i}+\zeta_{j}-\zeta_{i}\right|^{2}}.\label{ode-linearisation-calculation-1}
\end{align}
We can write each summand as
\begin{align}
    &\frac{m_{i}\left(z_{*j}-z_{*i}+\zeta_{j}-\zeta_{i}\right)}{\left|z_{*j}-z_{*i}\right|^{2}}\frac{1}{1+\frac{\left(\zeta_{j}-\zeta_{i}\right)\cdot\overline{\left(z_{*j}-z_{*i}\right)}}{\left|z_{*j}-z_{*i} \right|^{2}}+\frac{\overline{\left(\zeta_{j}-\zeta_{i}\right)}\cdot\left(z_{*j}-z_{*i}\right)}{\left|z_{*j}-z_{*i} \right|^{2}}+\frac{\left|\zeta_{j}-\zeta_{i}\right|^{2}}{\left|z_{*j}-z_{*i}\right|^{2}}}\nonumber\\
    &=\frac{m_{i}\left(z_{*j}-z_{*i}\right)}{\left|z_{*j}-z_{*i}\right|^{2}}-\frac{m_{i}\overline{\left(\zeta_{j}-\zeta_{i}\right)}\left(z_{*j}-z_{*i}\right)^{2}}{\left|z_{*j}-z_{*i} \right|^{4}}+\frac{m_{i}\left(\zeta_{j}-\zeta_{i}\right)}{\left|z_{*j}-z_{*i}\right|^{2}}-\frac{m_{i}\left(\zeta_{j}-\zeta_{i}\right)}{\left|z_{*j}-z_{*i}\right|^{2}}+\bigO{\left(\frac{\left|\zeta_{j}-\zeta_{i}\right|^{2}}{\left|z_{*j}-z_{*i}\right|^{2}}\right)}\nonumber\\
    &=\frac{m_{i}\left(z_{*j}-z_{*i}\right)}{\left|z_{*j}-z_{*i}\right|^{2}}-\frac{m_{i}\overline{\left(\zeta_{j}-\zeta_{i}\right)}\left(z_{*j}-z_{*i}\right)^{2}}{\left|z_{*j}-z_{*i} \right|^{4}}+\bigO{\left(\frac{\left|\zeta_{j}-\zeta_{i}\right|^{2}}{\left|z_{*j}-z_{*i}\right|^{2}}\right)}.\label{ode-linearisation-calculation-2}
\end{align}
Thus our linearised operator takes the form
\begin{align}
    &\frac{d\zeta_{j}}{dt}-\frac{1}{2\pi i}\sum_{i\neq j}\frac{m_{j}\left(z_{*j}-z_{*i}\right)^{2}}{\left|z_{*j}-z_{*i}\right|^{4}}\overline{\left(\zeta_{j}-\zeta_{i}\right)}\nonumber\\
    &=\frac{d\zeta_{j}}{dt}-\frac{1}{2\pi i}\sum_{i\neq j}\frac{m_{j}\left(1+\frac{t}{\tau}\right)^{2i\Lambda\tau}\left(z_{*j}(0)-z_{*i}(0)\right)^{2}}{\left(1+\frac{t}{\tau}\right)\left|z_{*j}(0)-z_{*i}(0)\right|^{4}}\overline{\left(\zeta_{j}-\zeta_{i}\right)}.\label{ode-linearisation-calculation-3}
\end{align}
We wish to further rewrite the above operator. To this end, define $\theta_{ji}$ via the relation
\begin{align}
    z_{*j}(0)-z_{*i}(0)=L_{ji}(0)e^{i\theta_{ji}}, \quad i\neq j,\label{theta-ji-definition}
\end{align}
with $L_{ji}$ defined in \eqref{point-vortex-pairwise-distance-function}. Then \eqref{ode-linearisation-calculation-3} takes the form
\begin{align}
    \frac{d\zeta_{j}}{dt}-\frac{1}{2\pi i}\sum_{i\neq j}\frac{m_{j}e^{i\left(2\left(\Lambda\tau\log{\left(1+\frac{t}{\tau}\right)}+\theta_{ji}\right)\right)}}{L_{ji}(0)^{2}\left(1+\frac{t}{\tau}\right)}\overline{\left(\zeta_{j}-\zeta_{i}\right)}.\label{ode-linearisation-calculation-4}
\end{align}
Now, to be able to write the above operator as a time derivative plus a time-dependent matrix, we define the real and imaginary parts of $\zeta_{j}$ as
\begin{align}
    \zeta_{j}(t)=\zeta^{R}_{j}(t)+i\zeta^{I}_{j}(t).\label{zeta-j-real-imaginary-parts}
\end{align}
We also define the trigonometric functions
\begin{align}
    S_{ji}&=\frac{1}{L_{ji}(0)^{2}}\sin{\left(2\left(\Lambda\tau\log{\left(1+\frac{t}{\tau}\right)}+\theta_{ji}\right)\right)},\nonumber\\
    C_{ji}&=\frac{1}{L_{ji}(0)^{2}}\sin{\left(2\left(\Lambda\tau\log{\left(1+\frac{t}{\tau}\right)}+\theta_{ji}\right)\right)}.\label{ode-linearisation-sin-cos-shorthand}
\end{align}
Note that by \eqref{theta-ji-definition}, $S_{ji}=S_{ij}$ for $i\neq j$, and similarly for $C_{ij}$. Then, for this decomposition, we can write the linearised operator given in \eqref{ode-linearisation-calculation-4} in the following component-wise way
\begin{align}
    \frac{d\zeta^{R}_{j}}{dt}-\sum_{i\neq j}\frac{m_{i}}{2\pi\left(1+\frac{t}{\tau}\right)}\left(\left(\zeta^{R}_{j}-\zeta^{R}_{i}\right)S_{ji}-\left(\zeta^{I}_{j}-\zeta^{I}_{i}\right)C_{ji}\right),\nonumber\\
    \frac{d\zeta^{I}_{j}}{dt}-\sum_{i\neq j}\frac{m_{i}}{2\pi \left(1+\frac{t}{\tau}\right)}\left(\left(\zeta^{R}_{j}-\zeta^{R}_{i}\right)C_{ji}+\left(\zeta^{R}_{j}-\zeta^{R}_{i}\right)S_{ji}\right).\label{ode-linearisation-calculation-5}
\end{align}
Now, let $\bm{\zeta}$ be the column vector defined by
\begin{align*}
    \bm{\zeta}=\begin{pmatrix}
    \zeta^{R}_{1} & \zeta^{I}_{1} & \zeta^{R}_{2} & \zeta ^{I}_{2} & \zeta^{R}_{3} & \zeta^{I}_{3}
    \end{pmatrix}^{\intercal}.
\end{align*}
Thus \eqref{ode-linearisation-calculation-5} can be written as
\begin{align}
    \left(\frac{d}{dt}-\frac{1}{2\pi\left(1+\frac{t}{\tau}\right)}\mathcal{A}\right)\bm{\zeta},\label{ode-linearisation-calculation-6}
\end{align}
where $\mathcal{A}$ is the $6\times 6$ matrix given by
\begin{align}
    \begin{psmallmatrix}
     m_{2}S_{12}+m_{3}S_{13} & -\left(m_{2}C_{12}+m_{3}C_{13}\right) & -m_{2}S_{12} & m_{2}C_{12} & -m_{3}S_{13} & m_{3}C_{13}\\
     -\left(m_{2}C_{12}+m_{3}C_{13}\right) & -\left(m_{2}S_{12}+m_{3}S_{13}\right) & m_{2}C_{12} & m_{2}S_{12} & m_{3}C_{13} & m_{3}S_{13}\\
     -m_{1}S_{12} & m_{1}C_{12} & m_{1}S_{12}+m_{3}C_{23} & -\left(m_{1}C_{12}+m_{3}C_{23}\right) & -m_{3}S_{23} & m_{3}C_{23}\\
     m_{1}C_{12} & m_{1}S_{12} & -\left(m_{1}C_{12}+m_{3}C_{23}\right) & -\left(m_{1}S_{12}+m_{3}S_{23}\right) & m_{3}C_{23} & m_{3}S_{23}\\
     -m_{1}S_{13} & m_{1}C_{13} & -m_{2}S_{23} & m_{2}C_{23} & m_{1}S_{13}+m_{2}S_{23} & -\left(m_{1}C_{13}+m_{2}C_{23}\right)\\
     m_{1}C_{13} & m_{1}S_{13} & m_{2}C_{23} & m_{2}S_{23} & -\left(m_{1}C_{13}+m_{2}S_{23}\right) & -\left(m_{1}S_{13}+m_{2}S_{23}\right)
    \end{psmallmatrix}.\label{matrix-A-definition}
\end{align}
We now note that, for $k=i$ or $k=j$, we have
\begin{align}
    \int_{0}^{t}\frac{m_{k}S_{ji}(s)}{2\pi\left(1+\frac{s}{\tau}\right)}ds&=\frac{m_{k}}{4\pi L_{ji}(0)^{2}\Lambda}\int_{0}^{t}\frac{2\Lambda}{\left(1+\frac{s}{\tau}\right)}\sin{\left(\left(2\left(\Lambda\tau\log{\left(1+\frac{s}{\tau}\right)}+\theta_{ji}\right)\right)\right)}ds\nonumber\\
    &=\frac{m_{k}\left[\cos{2\theta_{ji}}-\cos{\left(\left(2\left(\Lambda\tau\log{\left(1+\frac{t}{\tau}\right)}+\theta_{ji}\right)\right)\right)}\right]}{4\pi L_{ji}(0)^{2}\omega},\label{form-of-entries-of-matrix-B-1}
\end{align}
and analogously for integrals of the form \eqref{form-of-entries-of-matrix-B-1} but with $S_{ji}$ replaced with $C_{ji}$. Thus we have
\begin{align*}
    -\frac{1}{2\pi\left(1+\frac{t}{\tau}\right)}\mathcal{A}=\frac{d\mathcal{B}}{dt},
\end{align*}
where the entries of $\mathcal{B}$ can be found using \eqref{matrix-A-definition}--\eqref{form-of-entries-of-matrix-B-1}, and have the form
\begin{align}
    \frac{m_{k}}{4\pi L_{ji}(0)^{2}\Lambda}\left(\pm B_{1}\pm B_{2}\right),\label{form-of-entries-of-matrix-B-2}
\end{align}
where each $B_{k}$ takes the form
\begin{align}
    B_{k}=\begin{cases}
        0\\
        \sin{2\theta_{ji}}-\sin{\left(\left(2\left(\omega\tau\log{\left(1+\frac{t}{\tau}\right)}+\theta_{ji}\right)\right)\right)}\\
        \cos{2\theta_{ji}}-\cos{\left(\left(2\left(\omega\tau\log{\left(1+\frac{t}{\tau}\right)}+\theta_{ji}\right)\right)\right)},
    \end{cases}\label{form-of-entries-of-matrix-B-3}
\end{align}
some $i,j=1,2,3$, with $j\neq i$.

\medskip

Thus we can write \eqref{ode-linearisation-calculation-6} as
\begin{align}
    e^{-\mathcal{B}}\frac{d}{dt}\left(e^{\mathcal{B}}\bm{\zeta}\right),\label{ode-linearisation-calculation-7}
\end{align}
and it is left to calculate the behaviour of $e^{\pm \mathcal{B}}$ as functions of time. Note that for $p,q=1,\dots,6$
\begin{align*}
    \left[e^{\mathcal{B}}\right]_{pq}=\sum_{n=0}^{\infty}\frac{\left[\left(\mathcal{B}\right)^{n}\right]_{pq}}{n!}
\end{align*}
Now, each entry in $\mathcal{B}^{n}$ is a sum of $6^{n-1}\cdot 2^{n}$ terms, with each term a product of the form
\begin{align}
    \frac{m_{l_{1}}\dots m_{l_{n}}B_{k_{1}}\dots B_{k_{n}}}{\left(4\pi\right)^{n}\Lambda^{n}L_{i_{1}j_{1}}(0)^{2}\dots L_{i_{n}j_{n}}(0)^{2}},\label{bounds-on-B-entries-1}
\end{align}
where $m_{l_{1}},\dots,m_{l_{n}}\in\left\{m_{1},m_{2},m_{3}\right\}$, $B_{k_{1}},\dots,B_{k_{n}}$ are of the form \eqref{form-of-entries-of-matrix-B-3}, and $L_{i_{1}j_{1}}(0),\dots,L_{i_{n}j_{n}}(0)\in\left\{L_{12}(0),L_{13}(0),L_{23}(0)\right\}$.

\medskip
Recalling the mass constraints \eqref{mass-constraint}, the form of $\Lambda$,
\begin{align*}
    \Lambda=\frac{\left(m_{1}+m_{2}\right)^{3}+m_{1}^{3}\left(1+\Pi^{-1}\right)+m_{2}^{3}\left(1+\Pi\right)}{4\pi\left(m_{1}+m_{2}\right)^{2}L_{12}(0)^{2}},
\end{align*}
and that we made the choice in Section \ref{modified-vortices} that $L_{23}(t)>L_{12}(t)>L_{13}(t)$ for all $t\geq0$, we have
\begin{align}
    \left|\frac{m_{l_{1}}\dots m_{l_{n}}B_{k_{1}}\dots B_{k_{n}}}{\left(4\pi\right)^{n}\Lambda^{n}L_{i_{1}j_{1}}(0)^{2}\dots L_{i_{n}j_{n}}(0)^{2}}\right|\leq \frac{2^{n}L_{12}(0)^{2n}}{L_{13}(0)^{2n}}.\label{bounds-on-B-entries-2}
\end{align}
Thus for all $p,q=1,\dots,6$, we have
\begin{align}
    \left|\left[\left(\mathcal{B}\right)^{n}\right]_{pq}\right|\leq\frac{1}{6}\left(\frac{24L_{12}(0)^{2}}{L_{13}(0)^{2}}\right)^{n},\label{bounds-on-B-entries-3}
\end{align}
which finally implies 
\begin{align}
    \left|\left[e^{\pm \mathcal{B}}\right]_{pq}\right|\leq \frac{1}{6}e^{\frac{24L_{12}(0)^{2}}{L_{13}(0)^{2}}},\label{bounds-on-B-entries-4}
\end{align}
since a very similar argument works for $e^{-\mathcal{B}}$.
\section{First Approximation on $\left[T_{0},T\right]$}\label{first-approximation-section}
To first build a solution to the Euler equations \eqref{2d-euler-vorticity-stream} that approximate the point vortex trajectories given in Section \ref{modified-vortices} in the sense of \eqref{vortices-to-deltas}, we must first construct a good enough first approximation on a finite time interval $\left[T_{0},T\right]$ defined in \eqref{T0-definition}--\eqref{T-definition-2}. As discussed in Section~\ref{scheme},   $T_{0}$ will be some absolute constant large enough to ensure sufficient separation of the three vortices, and $T$ is an arbitrarily large parameter that we will take $T\to\infty$ to construct our full solution on $\left[T_{0},\infty\right)$. To start, we wish to  calculate the size of 
\begin{align}
    E_{1}\left[\omega_{*},\Psi_{*}\right]&=\dell_{t}\omega_{*}(x,t)+\gradperp_{x}\Psi_{*}(x,t)\cdot\nabla_{x}\omega_{*}(x,t),\label{first-approximation-error-calculation-1a}\\
    E_{2}\left[\omega_{*},\Psi_{*}\right]&=\Delta_{x}\Psi_{*}(x,t)+\omega_{*}(x,t).\label{first-error-approximation-calculation-1b}
\end{align}
for well chosen $\omega_{*}$ and $\Psi_{*}$. First, for $j=1,2,3$, let $\xi_{j}(t)$ be as in \eqref{point-vortex-trajectory-function-decomposition-new}, that is
\begin{align}
    \xi_{j}(t)=\xi_{*j}(t)+\tilde{\xi}_{j}(t),\label{point-vortex-trajectory-function-decomposition}
\end{align}
with $\xi_{*j}$ constructed in Section \ref{modified-vortices}, and with both terms in the decomposition satisfying the bounds \eqref{xi0-bound}--\eqref{tildexi-bound}. 

\medskip
From the bounds on the $\xi_{j}(t)$ in \eqref{xi0-bound}--\eqref{tildexi-bound}, we also have the following estimates on the linearisation of certain functions of $\xi_{j}$ around $\xi_{*j}$ that we will use repeatedly.
\begin{lemma}\label{p-q-rational-function-linearisation-lemma}
    Let $n\in\mathbb{Z}_{\geq1}$. Then, on $\left[T_{0},T\right]$, for $l_{1},\dots,l_{n}\in\left\{1,2\right\}$, we have the bounds
    \begin{align*}
        \left|\frac{\left(\xi_{j}^{l_{1}}-\xi_{i}^{l_{1}}\right)\dots\left(\xi_{j}^{l_{n}}-\xi_{i}^{l_{n}}\right)}{\left|\xi_{j}-\xi_{i}\right|^{2n}}-\frac{\left(\xi_{*j}^{l_{1}}-\xi_{*i}^{l_{1}}\right)\dots\left(\xi_{*j}^{l_{n}}-\xi_{*i}^{l_{n}}\right)}{\left|\xi_{*j}-\xi_{*i}\right|^{2n}}\right|\leq \frac{C\varepsilon^{4-\sigma}}{t^{\frac{n}{2}+2}},
    \end{align*}
    for $i,j=1,2,3$, $j\neq i$, for $\sigma>0$ arbitrarily small.
\end{lemma}
From \eqref{gamma-j-u-j-definition}, we have
$$
   -\Delta_{x}\Gamma_{j}=\varepsilon^{-2}U_{j},\ \ \ j=1,2,3.
   $$
Next, define cutoffs by $\eta_{j}$ defined \eqref{cutoffs-definition-1}. We also recall the choice of $K$ in Section \ref{cs} given by
\begin{align}
    K=L_{13}(0),\label{K-definition}
\end{align}
and for the rest of the main body of the paper until Section \ref{conclusion}, we also consider the full dynamic problem \eqref{first-approximation-error-calculation-1a}--\eqref{first-error-approximation-calculation-1b} from times $t\geq T_{0}$, where, as in Theorem \eqref{K-definition-new}, $T_{0}$ is some absolute constant large enough to ensure
\begin{align}
    L_{ij}\left(T_0\right)\geq2L_{23}(0), \quad i,j=1,2,3, \quad i<j.\label{T0-definition}
\end{align}
Then, due to \eqref{xi0-bound}--\eqref{tildexi-bound}, for small enough $\varepsilon>0$,
\begin{align*}
    \left|\xi_{i}(T_{0})-\xi_{j}(T_{0})\right|\geq\frac{3}{2}L_{23}(0), \quad i,j=1,2,3, \quad i<j.
\end{align*}
As per the strategy laid out in Section \ref{scheme}, we will first construct a sequence of solutions on time intervals $\left[T_{0},T\right]$ for each $T>0$ large enough, from which we will extract a solution to \eqref{2d-euler-vorticity-stream} on $\left[T_{0},\infty\right)$. Accordingly, following \eqref{T-definition}, we also fix $T>0$, a large time satisfying
\begin{align}
    T\gg T_{0}.\label{T-definition-2}
\end{align}
\begin{remark}\label{choice-of-T0-remark}
    We remark here that as long we define $T_{0}$ in a manner independent of $\varepsilon$, we can enlarge it as we like as long as it satisfies \eqref{T0-definition}. Indeed, we will do so in order to close a crucial a priori estimate in Lemma \ref{L2-interior-a-priori-estimate-lemma}.
\end{remark}
Now we define $\psi_{j}$ and $\phi_{j}$ of the form
\begin{align*}
    \psi_{j}&=\psi_{j}\left(\frac{x-\xi_{j}(t)}{\varepsilon},t\right),\\
    \phi_{j}&=\phi_{j}\left(\frac{x-\xi_{j}(t)}{\varepsilon},t\right),
\end{align*}
such that for $j=1,2,3$, and $w\in\mathbb{R}^{2}$
\begin{align}
    \Delta_{w}\psi_{j}\left(w,t\right)+\phi_{j}\left(w,t\right)=0\label{psi-j-phi-j-equation}
\end{align}
on $\mathbb{R}^{2}$, noting that $\phi_{j}\left(w,t\right)$ has support contained in a ball radius $1$ around $0$.

\medskip
We set
\begin{equation} \label{psi-in-definition}
    \psi^{in}=\left(\psi_{1},\psi_{2},\psi_{3}\right),\quad
    \phi^{in}=\left(\phi_{1},\phi_{2},\phi_{3}\right).
\end{equation}
Given these definitions, we then set
\begin{align}
    \Psi_{*}&=\Psi_{0}+\sum_{j=1}^{3}\kappa_{j}\eta_{j}\psi_{j}+\psi^{out},\label{psi-star-definition}\\
    \omega_{*}&=\omega_{0}+\varepsilon^{-2}\sum_{j=1}^{3}\kappa_{j}\phi_{j},\label{omega-star-definition}
\end{align}
for some $\psi^{out}=\psi^{out}(x,t)$. Recalling the definition \eqref{point-vortex-decomposition-vector-form}, from \eqref{psi-star-definition}--\eqref{omega-star-definition}, we can see that
\begin{align}
    E_{1}\left[\omega_{*},\Psi_{*}\right]&=\varepsilon^{-4}\sum_{j=1}^{3}\kappa_{j}E_{j1}\left(\phi_{j},\psi_{j},\psi^{out};\xi\right),\label{first-approximation-inner-error}
\end{align}
and
$$
    E_{2}\left[\omega_{*},\Psi_{*}\right]=E_{2}^{out}\left(\psi^{in},\psi^{out};\xi\right),
    $$
where,
\begin{align}
    E_{j1}&=\varepsilon^{2}\dell_{t}\left(U_{j}+\phi_{j}\right)+\varepsilon^{2}\gradperp_{x}\left(\Psi_{0}+\kappa_{j}\psi_{j}+\psi^{out}\right)\cdot\nabla_{x}\left(U+\phi_{j}\right),\label{first-approximation-inner-error-ER1-definition}
\end{align}
\begin{align}
    E_{2}^{out}=\Delta_{x}\psi^{out}+\sum_{j=1}^{3}\kappa_{j}\left[\psi_{j}\Delta_{x}\eta_{j}+2\nabla_{x}\eta_{j}\cdot\nabla_{x}\psi_{j}\right].\label{first-approximation-outer-error-E2out-definition}
\end{align}
We note that the specific forms of the $E_{j1}$, and $E_{2}^{out}$ are due to the fact that $\phi_{j}$ will have the same support as $U_{j}$, and so due to the choice of $K$ and $T_{0}$ in \eqref{K-definition}--\eqref{T0-definition}, $\eta_{j}\equiv1$ on $\supp{\phi_{j}}$ for $j=1,2,3$, and $\eta_{j}\equiv0$ on $\supp{\phi_{j+1}}$ and $\supp{\phi_{j+2}}$, where again, the addition in the subscript is taken modulo 3.

\medskip
Let
\begin{align}
    y=\frac{x-\xi_{j}(t)}{\varepsilon}.\label{change-of-coordinates-dynamic-problem}
\end{align}
In these coordinates we have
\begin{align}
    \Psi_{0}&=\kappa_{j}\Gamma\left(\frac{x-\xi_{j}(t)}{\varepsilon}\right)+\sum_{i\neq j}\kappa_{i}\Gamma\left(\frac{x-\xi_{i}(t)}{\varepsilon}\right)\nonumber\\
    &=\kappa_{j}\Gamma\left(y\right)+\sum_{i\neq j}\kappa_{i}\Gamma\left(y+\frac{\xi_{j}(t)-\xi_{i}(t)}{\varepsilon}\right).\label{Psi-0-y-coordinates}
\end{align}
Then, using \eqref{power-nonlinearity-vorticity-definition}--\eqref{gamma-j-u-j-definition}, \eqref{psi-j-phi-j-equation}, and applying the chain rule, we have the following lemma.
\begin{lemma}\label{ER1-lemma}
    With change of coordinates given by \eqref{change-of-coordinates-dynamic-problem}, we can write $E_{j1}$, defined in \eqref{first-approximation-inner-error-ER1-definition}, as
    \begin{align}
        &E_{j1}=\varepsilon^{2}\dell_{t}\phi_{j}+\kappa_{j}\gradperp_{y}\psi_{j}\cdot\grad_{y}\phi_{j}-\kappa_{j}\gradperp_{y}\Gamma\cdot\grad_{y}\left(\Delta_{y}\psi_{j}+\gamma\Gamma^{\gamma-1}_{+}\psi_{j}\right)\nonumber\\
        &+\left(\gradperp_{y}\left(\sum_{i\neq j}\kappa_{i}\Gamma\left(y+\frac{\xi_{j}(t)-\xi_{i}(t)}{\varepsilon}\right)\right)-\varepsilon\dot{\xi}_{j}\right)\cdot\grad_{y}\left(U+\phi_{j}\right)+\gradperp_{y}\psi^{out}\cdot\grad_{y}\left(U+\phi_{j}\right).\label{ER1-lemma-statement-1}
    \end{align}
\end{lemma}
Now we calculate how the quantity
\begin{align*}
    \gradperp_{y}\left(\sum_{i\neq j}\kappa_{i}\Gamma\left(y+\frac{\xi_{j}(t)-\xi_{i}(t)}{\varepsilon}\right)\right)-\varepsilon\dot{\xi}_{j}
\end{align*}
behaves on $B_{1}(0)$ in the $y$ coordinates defined in \eqref{change-of-coordinates-dynamic-problem}. Define, for $i,l\in\left\{1,2,3\right\}$ and suitable functions $h_{i},k_{l}:I\to\mathbb{R}^{2}$, for some time interval $I$, the nonlinear functionals
\begin{align}
    \mathcal{N}_{j}\left(h_{1},h_{2},h_{3}\right)\left[k_{1},k_{2},k_{3}\right]=\sum_{i\neq j}\frac{m_{i}}{2\pi}\frac{\left(h_{j}+k_{j}-h_{i}-k_{i}\right)^{\perp}}{\left|h_{j}+k_{j}-h_{i}-k_{i}\right|^{2}}-\sum_{i\neq j}\frac{m_{i}}{2\pi}\frac{\left(h_{j}-h_{i}\right)^{\perp}}{\left|h_{j}-h_{i}\right|^{2}},\label{point-vortex-nonlinear-functionals}
\end{align}
for $j\in\left\{1,2,3\right\}$.
\begin{lemma}\label{remainder-modal-calculation-lemma}
    For all $\varepsilon>0$ small enough, and $t\geq T_{0}$ as in \eqref{T0-definition}, for $y\in B_{1}(0)$, we have
    \begin{align}
        &\gradperp_{y}\left(\sum_{i\neq j}\kappa_{i}\Gamma\left(y+\frac{\xi_{j}(t)-\xi_{i}(t)}{\varepsilon}\right)\right)-\varepsilon\dot{\xi}_{j}=-\varepsilon\left(\dot{\tilde{\xi}}_{j}+\mathcal{N}_{j}\left(\xi_{*}\right)\left[\tilde{\xi}\right]\right)\nonumber\\
        &+\gradperp_{y}\left(-\sum_{i\neq j}\frac{m_{i}}{4\pi}\left(\mathcal{E}^{\left(ji\right)}_{2}(y,\xi)+\mathcal{E}^{\left(ji\right)}_{3}(y,\xi)+\mathcal{E}^{\left(ji\right)}_{4}(y,\xi)+\mathcal{R}^{\left(ji\right)}_{5}\left(y,\xi\right)\right)\right),\label{remainder-modal-calculation-statement}
    \end{align}
    where for each $j\in\left\{1,2,3\right\}$ and each $i\in\left\{1,2,3\right\}$ such that $i\neq j$, we have that, for $l=2,3,4$, the $\mathcal{E}^{\left(ji\right)}_{l}\left(y,\xi\right)$ are a sum of mode $l$ terms in the $y$ coordinate system, and have size $\bigO{\left(\varepsilon^{l}\left|y\right|^{l}\left|\xi_{j}-\xi_{i}\right|^{-l}\right)}=\bigO{\left(\varepsilon^{l}t^{-\frac{l}{2}}\right)}$, and $\mathcal{R}^{\left(ji\right)}_{5}=\bigO{\left(\varepsilon^{5}\left|y\right|^{5}\left|\xi_{j}-\xi_{i}\right|^{-5}\right)}=\bigO{\left(\varepsilon^{5}t^{-\frac{5}{2}}\right)}$.
\end{lemma}
\begin{proof}
     First we note that on $B_{1}(0)$, by shrinking $\varepsilon$ if necessary we have
    \begin{align}
        \gradperp_{y}\left(\sum_{i\neq j}\kappa_{i}\Gamma\left(y+\frac{\xi_{j}(t)-\xi_{i}(t)}{\varepsilon}\right)\right)=\gradperp_{y}\left(-\sum_{i\neq j}\frac{m_{i}}{4\pi}\log{\left(1+\frac{\varepsilon y\cdot\left(\xi_{j}-\xi_{i}\right)}{\left|\xi_{j}-\xi_{i}\right|^{2}}+\frac{\varepsilon^{2}\left|y\right|^{2}}{\left|\xi_{j}-\xi_{i}\right|^{2}}\right)}\right).\label{remainder-modal-calculation-1}
    \end{align}
    In the above, we have used \eqref{vortex-mass-definition}, \eqref{Gamma-piecewise-definition}, \eqref{mass-in-terms-of-nu}, and \eqref{kappa-j-definition}.

    \medskip
    
    We can expand both of the summands inside the perpendicular gradient above as
    \begin{align}
        &-\frac{m_{i}}{4\pi}\left(\frac{2\varepsilon y\cdot\left(\xi_{j}-\xi_{i}\right)}{\left|\xi_{j}-\xi_{i}\right|^{2}}+\frac{\varepsilon^{2}\left|y\right|^{2}}{\left|\xi_{j}-\xi_{i}\right|^{2}}-\frac{2\varepsilon^{2}\left(y\cdot\left(\xi_{j}-\xi_{i}\right)\right)^{2}}{\left|\xi_{j}-\xi_{i}\right|^{4}}\right)\nonumber\\
        &-\frac{m_{i}}{4\pi}\left(\frac{8\varepsilon^{3}\left(y\cdot\left(\xi_{j}-\xi_{i}\right)\right)^{3}}{3\left|\xi_{j}-\xi_{i}\right|^{6}}-\frac{2\varepsilon^{3}\left|y\right|^{2}y\cdot\left(\xi_{j}-\xi_{i}\right)}{\left|\xi_{j}-\xi_{i}\right|^{4}}\right)\nonumber\\
        &-\frac{m_{i}}{4\pi}\left(\frac{4\varepsilon^{4}\left|y\right|^{2}\left(y\cdot\left(\xi_{j}-\xi_{i}\right)\right)^{2}}{\left|\xi_{j}-\xi_{i}\right|^{6}}-\frac{4\varepsilon^{4}\left(y\cdot\left(\xi_{j}-\xi_{i}\right)\right)^{4}}{\left|\xi_{j}-\xi_{i}\right|^{8}}-\frac{\varepsilon^{4}\left|y\right|^{4}}{\left|\xi_{j}-\xi_{i}\right|^{4}}\right)+\mathcal{R}^{\left(ji\right)}_{5},\label{remainder-modal-calculation-1a}
    \end{align}
    where $\mathcal{R}^{\left(ji\right)}_{5}=\bigO{\left(\varepsilon^{5}\left|y\right|^{5}\left|\xi_{j}-\xi_{i}\right|^{-5}\right)}=\bigO{\left(\varepsilon^{5}t^{-\frac{5}{2}}\right)}$. 
    
    \medskip
    Now let
    \begin{align}
        \mathcal{E}^{\left(ji\right)}_{2}(y,\xi)&=\frac{\varepsilon^{2}}{\left|\xi_{j}-\xi_{i}\right|^{4}}\left(\left(\xi_{j}^{1}-\xi_{i}^{1}\right)^{2}-\left(\xi_{j}^{2}-\xi_{i}^{2}\right)^{2}\right)\left(y_{1}^{2}-y_{2}^{2}\right)\nonumber\\
        &-\frac{4\varepsilon^{2}}{\left|\xi_{j}-\xi_{i}\right|^{4}}\left(\xi_{j}^{1}-\xi_{i}^{1}\right)\left(\xi_{j}^{2}-\xi_{i}^{2}\right)y_{1}y_{2},\label{remainder-modal-calculation-mode-2}
    \end{align}
    \begin{align}
        \mathcal{E}^{\left(ji\right)}_{3}(y,\xi)&=\frac{2\varepsilon^{3}}{3\left|\xi_{j}-\xi_{i}\right|^{6}}\left(\left(\xi_{j}^{1}-\xi_{i}^{1}\right)^{3}-3\left(\xi_{j}^{1}-\xi_{i}^{1}\right)\left(\xi_{j}^{2}-\xi_{i}^{2}\right)^{2}\right)\left(y_{1}^{3}-3y_{1}y_{2}^{2}\right)\nonumber\\
        &+\frac{2\varepsilon^{3}}{3\left|\xi_{j}-\xi_{i}\right|^{6}}\left(\left(\xi_{j}^{2}-\xi_{i}^{2}\right)^{3}-3\left(\xi_{j}^{2}-\xi_{i}^{2}\right)\left(\xi_{j}^{1}-\xi_{i}^{1}\right)^{2}\right)\left(y_{2}^{3}-3y_{2}y_{1}^{2}\right),\label{remainder-modal-calculation-mode-3}
    \end{align}
    and
    \begin{align}
        &\mathcal{E}^{\left(ji\right)}_{4}\left(y,\xi\right)=\frac{8\varepsilon^{4}}{\left|\xi_{j}-\xi_{i}\right|^{8}}\left(\xi_{j}^{1}-\xi_{i}^{1}\right)\left(\xi_{j}^{2}-\xi_{i}^{2}\right)\left(\left(\xi_{j}^{1}-\xi_{i}^{1}\right)^{2}-\left(\xi_{j}^{2}-\xi_{i}^{2}\right)^{2}\right)y_{1}y_{2}\left(y_{1}^{2}-y_{2}^{2}\right)\nonumber\\
        &+\frac{8\varepsilon^{4}}{\left|\xi_{j}-\xi_{i}\right|^{8}}\left(\left(\left(\xi_{j}^{1}-\xi_{i}^{1}\right)^{2}-\left(\xi_{j}^{2}-\xi_{i}^{2}\right)^{2}\right)^{2}-4\left(\xi_{j}^{1}-\xi_{i}^{1}\right)^{2}\left(\xi_{j}^{2}-\xi_{i}^{2}\right)^{2}\right)\cdot\left(\left(y_{1}^{2}-y_{2}^{2}\right)^{2}-4y_{1}^{2}y_{2}^{2}\right).\label{remainder-modal-calculation-mode-4}
    \end{align}
    It is clear that $\mathcal{E}_{l}^{ji}$, for $l=2,3,4$ are a sum of mode $l$ terms in the $y$ coordinate system, and have size $\bigO{\left(\varepsilon^{l}\left|y\right|^{l}\left|\xi_{j}-\xi_{i}\right|^{-l}\right)}=\bigO{\left(\varepsilon^{l}t^{-\frac{l}{2}}\right)}$.

    \medskip
    Next, we note that 
    \begin{align}
        \gradperp_{y}\left(-\sum_{i\neq j}\frac{m_{i}}{2\pi}\frac{\varepsilon y\cdot\left(\xi_{j}-\xi_{i}\right)}{\left|\xi_{j}-\xi_{i}\right|^{2}}\right)=-\varepsilon\sum_{i\neq j}\frac{m_{i}}{2\pi}\frac{\left(\xi_{j}-\xi_{i}\right)^{\perp}}{\left|\xi_{j}-\xi_{i}\right|^{2}}.\label{remainder-modal-calculation-mode-1-1}
    \end{align}
    Thus, using \eqref{point-vortex-trajectory-function-decomposition}, we can write
    \begin{align*}
        &\gradperp_{y}\left(-\sum_{i\neq j}\frac{m_{i}}{2\pi}\frac{\varepsilon y\cdot\left(\xi_{j}-\xi_{i}\right)}{\left|\xi_{j}-\xi_{i}\right|^{2}}\right)-\varepsilon\dot{\xi}_{j}=\varepsilon\left(-\dot{\xi}_{*j}-\sum_{i\neq j}\frac{m_{i}}{2\pi}\frac{\left(\xi_{*j}-\xi_{*i}\right)^{\perp}}{\left|\xi_{*j}-\xi_{*i}\right|^{2}}\right)\\
        &+\varepsilon\left(-\dot{\tilde{\xi}}_{j}-\left(\sum_{i\neq j}\frac{m_{i}}{2\pi}\frac{\left(\xi_{j}-\xi_{i}\right)^{\perp}}{\left|\xi_{j}-\xi_{i}\right|^{2}}-\sum_{i\neq j}\frac{m_{i}}{2\pi}\frac{\left(\xi_{*j}-\xi_{*i}\right)^{\perp}}{\left|\xi_{*j}-\xi_{*i}\right|^{2}}\right)\right).
    \end{align*}
    Recalling the definitions \eqref{point-vortex-decomposition-vector-form} and \eqref{point-vortex-nonlinear-functionals}, as well as the fact that $\xi_{*}$ solves the system \eqref{KR}, we therefore have that 
    \begin{align}
        \gradperp_{y}\left(-\sum_{i\neq j}\frac{m_{i}}{2\pi}\frac{\varepsilon y\cdot\left(\xi_{j}-\xi_{i}\right)}{\left|\xi_{j}-\xi_{i}\right|^{2}}\right)-\varepsilon\dot{\xi}_{j}=-\varepsilon\left(\dot{\tilde{\xi}}_{j}+\mathcal{N}_{j}\left(\xi_{*}\right)\left[\tilde{\xi}\right]\right).\label{remainder-modal-calculation-mode-1-2}
    \end{align}
    Thus,
    \begin{align*}
        &\gradperp_{y}\left(\sum_{i\neq j}\kappa_{i}\Gamma\left(y+\frac{\xi_{j}(t)-\xi_{i}(t)}{\varepsilon}\right)\right)-\varepsilon\dot{\xi}_{j}=-\varepsilon\left(\dot{\tilde{\xi}}_{j}+\mathcal{N}_{j}\left(\xi_{*}\right)\left[\tilde{\xi}\right]\right)\nonumber\\
        &+\gradperp_{y}\left(-\sum_{i\neq j}\frac{m_{i}}{4\pi}\left(\mathcal{E}^{\left(ji\right)}_{2}(y,\xi)+\mathcal{E}^{\left(ji\right)}_{3}(y,\xi)+\mathcal{E}^{\left(ji\right)}_{4}(y,\xi)+\mathcal{R}^{\left(ji\right)}_{5}\left(y,\xi\right)\right)\right),
    \end{align*}
    as required.
\end{proof}
Now we construct the first approximation. Recall the definitions of $\psi^{in}$, $\phi^{in}$, $E_{j1}$, and $E_{2}^{out}$ from \eqref{psi-in-definition}--\eqref{first-approximation-outer-error-E2out-definition}.
\begin{theorem}\label{first-approximation-construction-theorem}
    Let $t\in\left[T_{0},T\right]$ as in \eqref{T0-definition}--\eqref{T-definition-2}. Then for all $\varepsilon>0$ small enough, and for any $\xi$ as in \eqref{point-vortex-trajectory-function-decomposition}, there exists functions $\psi_{jk}$, $\phi_{jk}$, $j=1,2,3$, $k=1,2$, $\psi^{out}$ such that $\phi_{jk}$ has the same support as $U_{j}$, and
    \begin{align*}
        \left\|\psi_{j1}\right\|_{L^{\infty}\left(\mathbb{R}^{2}\right)}+\left\|\nabla\psi_{j1}\right\|_{L^{\infty}\left(\mathbb{R}^{2}\right)}+\left\|\phi_{j1}\right\|_{L^{\infty}\left(\mathbb{R}^{2}\right)}&\leq \frac{C\varepsilon^{2}}{t},\\
        \left\|\psi_{j2}\right\|_{L^{\infty}\left(\mathbb{R}^{2}\right)}+\left\|\nabla\psi_{j2}\right\|_{L^{\infty}\left(\mathbb{R}^{2}\right)}+\left\|\phi_{j2}\right\|_{L^{\infty}\left(\mathbb{R}^{2}\right)}&\leq \frac{C\varepsilon^{4}}{t^{2}}.
    \end{align*}
     Moreover,
    \begin{align*}
        \|\left(\log{\left(\left|\cdot\right|+2\right)}\right)^{-1}\psi^{out}\|_{L^{\infty}\left(\mathbb{R}^{2}\right)}+\left\|\nabla\psi^{out}\right\|_{L^{\infty}\left(\mathbb{R}^{2}\right)}\leq\frac{C\varepsilon^{4}}{t^{2}},
    \end{align*}
    such that if $\psi_{j}=\psi_{j1}+\psi_{j2}$, $\phi_{j}=\phi_{j1}+\phi_{j2}$ in \eqref{psi-star-definition}--\eqref{omega-star-definition}, in coordinates $y=\frac{x-\xi_{j}}{\varepsilon}$, we have
    \begin{align*}
        &E_{j1}\left(\phi_{j},\psi_{j},\psi^{out};\xi\right)=-\varepsilon\nabla_{y}\left(U\left(y\right)+\phi_{j}(y)\right)\cdot\left(\dot{\tilde{\xi}}_{j}+\mathcal{N}_{j}\left(\xi_{*}\right)\left[\tilde{\xi}\right]\right)\nonumber\\
        &+\gamma\Gamma^{\gamma-1}_{+}\left(y\right)\mathcal{R}^{*}_{j1}\left(y,\xi\right)+\gamma\left(\gamma-1\right)\Gamma^{\gamma-2}_{+}\left(y\right)\mathcal{R}^{*}_{j2}\left(y,\xi\right)+\gamma\left(\gamma-1\right)\left(\gamma-2\right)\Gamma^{\gamma-3}_{+}\left(y\right)\mathcal{R}^{*}_{j3}\left(y,\xi\right),
    \end{align*}
    where, for $l=1,2,3$ and $j=1,2,3$
    \begin{align*}
        \left|\dell_{\Gamma}^{l}\left(\Gamma^{\gamma}_{+}\right)\mathcal{R}^{*}_{jl}\left(y,\xi\right)\right|&\leq\frac{C\varepsilon^{5}\dell_{\Gamma}^{l}\left(\Gamma^{\gamma}_{+}\right)}{t^{\frac{5}{2}}}.
    \end{align*}
    Finally,
    \begin{align*}
        E_{2}^{out}\left(\phi^{in},\psi^{in},\psi^{out};\xi\right)=\bigO{\left(\frac{\varepsilon^{6}}{t^{4}}\right)}.
    \end{align*}
\end{theorem}

The construction will proceed as follows: we rewrite $E_{j1}$ from its form in Lemma \ref{ER1-lemma} and then construct an elliptic improvement. We then use this elliptic improvement to construct $\psi^{out}$. Then we analyze the errors created by the first elliptic improvement and $\psi^{out}$ in order to both construct a second elliptic improvement, and finally check that the errors left after constructing the second elliptic improvement have the desired properties as stated in Theorem \ref{first-approximation-construction-theorem}.
\begin{proof} 
    \textbf{Step I: Rewriting $E_{j1}$ Part A}

    \medskip
    For $j=1,2,3$, we first decompose $\phi_{j}=\phi_{j1}+\phi_{j2}$ and $\psi_{j}=\psi_{j1}+\psi_{j2}$ so that
    \begin{align*}
        -\Delta_{y}\psi_{jk}=\phi_{jk},\ k=1,2,
    \end{align*}
    where the Poisson equations mean above solving on $\mathbb{R}^2$, where we note the right hand side for all $6$ Poisson equations above will have compact support. Let
    \begin{align*}
        \mathcal{R}_{j6}\left(y,\xi\right)&=\sum_{i\neq j}\frac{m_{i}}{4\pi}\left(\mathcal{E}^{\left(ji\right)}_{2}(y,\xi_{*})+\mathcal{E}^{\left(ji\right)}_{3}(y,\xi_{*})+\mathcal{E}^{\left(ji\right)}_{4}(y,\xi_{*})\right)\\
        &-\sum_{i\neq j}\frac{m_{i}}{4\pi}\left(\mathcal{E}^{\left(ji\right)}_{2}(y,\xi)+\mathcal{E}^{\left(ji\right)}_{3}(y,\xi)+\mathcal{E}^{\left(ji\right)}_{4}(y,\xi)+\mathcal{R}^{\left(ji\right)}_{5}\left(y,\xi\right)\right),
    \end{align*}
    and correspondingly let
    \begin{align}
        \mathcal{E}_{j6}\left(y,\xi_{*}\right)=-\sum_{i\neq j}\frac{m_{i}}{4\pi}\left(\mathcal{E}^{\left(ji\right)}_{2}(y,\xi_{*})+\mathcal{E}^{\left(ji\right)}_{3}(y,\xi_{*})+\mathcal{E}^{\left(ji\right)}_{4}(y,\xi_{*})\right).\label{E-j6-definition}
    \end{align}
    Using \eqref{xi0-bound}--\eqref{tildexi-bound}, and Lemmas \ref{p-q-rational-function-linearisation-lemma} and \ref{remainder-modal-calculation-lemma}, we have that
    \begin{align}
        \left|\nabla_{y}\left(U\left(y\right)\right)\cdot\gradperp_{y}\mathcal{R}_{j6}\left(y,\xi\right)\right|\leq \frac{C\varepsilon^{5}\gamma\Gamma^{\gamma-1}_{+}\left(y\right)}{t^{3}}.\label{first-approximation-construction-2a}
    \end{align}
    Then we can rewrite \eqref{ER1-lemma-statement-1} on $\mathbb{R}^{2}$ as
    \begin{align}
        &E_{j1}=\varepsilon^{2}\dell_{t}\phi_{j1}+\kappa_{j}\gradperp_{y}\psi_{j1}\cdot\grad_{y}\phi_{j1}+\left(\varepsilon^{2}\dell_{t}\phi_{j2}+\kappa_{j}\gradperp_{y}\psi_{j2}\cdot\grad_{y}\phi_{j}+\kappa_{j}\gradperp_{y}\psi_{j1}\cdot\grad_{y}\phi_{j2}\right)\nonumber\\
        &-\kappa_{j}\gradperp_{y}\Gamma\cdot\grad_{y}\left(\Delta_{y}\psi_{j1}+\gamma\Gamma^{\gamma-1}_{+}\psi_{j1}-\kappa_{j}^{-1}\gamma\Gamma^{\gamma-1}_{+}\mathcal{E}_{j6}\right)-\kappa_{j}\gradperp_{y}\Gamma\cdot\grad_{y}\left(\Delta_{y}\psi_{j2}+\gamma\Gamma^{\gamma-1}_{+}\psi_{j2}\right)\nonumber\\
        &-\grad_{y}(U+\phi_{j})\cdot\left(\varepsilon\left(\dot{\tilde{\xi}}_{j}+\mathcal{N}_{j}\left(\xi_{*}\right)\left[\tilde{\xi}\right]\right)\right)+\nabla_{y}U\cdot\gradperp_{y}\mathcal{R}_{j6}+\gradperp_{y}\left(\mathcal{E}_{j6}+\mathcal{R}_{j6}\right)\cdot\grad_{y}\phi_{j}\nonumber\\
        &+\gradperp_{y}\psi^{out}\cdot\grad_{y}\left(U+\phi_{j}\right).\label{first-approximation-construction-8}
    \end{align}
    \medskip
    \noindent \textbf{Step II: Constructing $\psi_{j1}$}

    \medskip
    Noting that $\mathcal{E}_{j6}(y,\xi_{*})$ defined in \eqref{E-j6-definition} is a sum of mode $2$, $3$, and $4$ terms in $y$, we can use Lemma \ref{vortex-linearised-equation-fourier-coefficients-behaviour-lemma} to define $\psi_{j1}^{\left(2\right)}$, $\psi_{j1}^{\left(3\right)}$, and $\psi_{j1}^{\left(4\right)}$ as the solutions to 
    \begin{align}
        \Delta_{y}\psi_{j1}^{\left(l\right)}(y)+\gamma\Gamma^{\gamma-1}_{+}(y)\psi_{j1}^{\left(l\right)}(y)&=-\kappa_{j}^{-1}\gamma\Gamma^{\gamma-1}_{+}(y)\sum_{i\neq j}\frac{m_{i}}{4\pi}\mathcal{E}^{\left(ji\right)}_{l}(y,\xi_{*})\nonumber\\
        &=-\kappa_{j}^{-1}\gamma\Gamma^{\gamma-1}_{+}(y)\sum_{i\neq j}\frac{m_{i}}{4\pi}r^{l}\tilde{\mathcal{E}}^{\left(ji\right)}_{l}(\sin{\left(l\theta\right)},\cos{\left(l\theta\right)},\xi_{*}),\label{psi-j1-ODEs}
    \end{align}
    for $l=2,3,4$, where we recall $y=re^{i\theta}$. Then defining $\phi_{j1}^{\left(l\right)}=-\Delta_{y}\psi_{j1}^{\left(l\right)}$, we obtain
    \begin{align}
        \phi_{j1}^{\left(l\right)}=-\kappa_{j}^{-1}\gamma\Gamma^{\gamma-1}_{+}(y)\sum_{i\neq j}\frac{m_{i}}{4\pi}\left(\frac{\varrho_{l}\left(r\right)}{r^{l}}+1\right)\mathcal{E}^{\left(ji\right)}_{l}(y,\xi_{*}).\label{phi-j1-l-formula}
    \end{align}
    where $\varrho_{l}$ is defined in \ref{vortex-linearised-equation-rk-error-solution}. Then, from Lemmas \ref{remainder-modal-calculation-lemma} and \ref{vortex-linearised-equation-fourier-coefficients-behaviour-lemma}, we have the bounds, for $l=2,3,4$
    \begin{align}
        \|\psi^{\left(l\right)}_{j1}\|_{L^{\infty}\left(\mathbb{R}^{2}\right)}+\|\nabla\psi^{\left(l\right)}_{j1}\|_{L^{\infty}\left(\mathbb{R}^{2}\right)}+\|\phi^{\left(l\right)}_{j1}\|_{L^{\infty}\left(\mathbb{R}^{2}\right)}\leq \frac{C\varepsilon^{l}}{t^{\frac{l}{2}}}.\label{psi-j1-bounds}
    \end{align}
    Finally, letting $\psi_{j1}=\sum_{l=2}^{4}\psi_{j1}^{\left(l\right)}$, and analogously for $\phi_{j1}$, we obtain 
    \begin{align*}
        -\kappa_{j}\gradperp_{y}\Gamma\cdot\grad_{y}\left(\Delta_{y}\psi_{j1}+\gamma\Gamma^{\gamma-1}_{+}\psi_{j1}-\kappa_{j}^{-1}\gamma\Gamma^{\gamma-1}_{+}\mathcal{E}_{j6}\right)=0.
    \end{align*}
    \medskip
    \noindent \textbf{Step III: Constructing $\psi^{out}$}

    \medskip
    Now we construct $\psi^{out}$ in the original $x$ variable. As $E_{2}^{out}$ has the form \eqref{first-approximation-outer-error-E2out-definition}, we obtain $\psi^{out}$ by solving the equation
    \begin{align}
        -\Delta_{x}\psi^{out}=\sum_{j=1}^{3}\kappa_{j}\left[\psi_{j1}\left(\frac{x-\xi_{j}}{\varepsilon},\xi_{*}\right)\Delta_{x}\eta_{j}+2\nabla_{x}\eta_{j}\cdot\nabla_{x}\psi_{j1}\left(\frac{x-\xi_{j}}{\varepsilon},\xi_{*}\right)\right]\label{psi-out-equation}
    \end{align}
    on $\mathbb{R}^{2}$. The size of $\psi^{out}$ depends on the size of $\psi_{j1}$ and $\nabla\psi_{j1}$ on the support of the derivatives of the cutoffs $\eta_{j}$, $j=1,2,3$. For $\psi_{j1}$, the definition of the cutoff $\eta_{j}$ in \eqref{cutoffs-definition-1} means we can consider $\psi_{j1}\left(y,\xi_{*}\right)$ for $\left|y\right|\geq C\varepsilon^{-1}t^{\frac{1}{2}}$.
    
    \medskip
    We can write 
    \begin{align*}
        \psi_{j1}\left(y,\xi_{*}\right)=-\frac{1}{4\pi}\int_{B_{1}\left(0\right)}\log\left|z-y\right|^{2}\phi_{j1}\left(z,\xi_{*}\right)dz,
    \end{align*}
    where the region of integration comes from \eqref{phi-j1-l-formula} and the fact that $\supp\Gamma^{\gamma-1}_{+}(y)=B_{1}\left(0\right)$. Then, since $\left|y\right|\geq C\varepsilon^{-1}t^{\frac{1}{2}}$ and $\left|z\right|\leq 1$, we have
    \begin{align*}
        \psi_{j1}\left(y,\xi_{*}\right)=-\frac{1}{4\pi}\int_{B_{1}\left(0\right)}\log\left(1-\frac{2z\cdot y}{\left|y\right|^{2}}+\frac{\left|z\right|^{2}}{\left|y\right|^{2}}\right)\phi_{j1}\left(z,\xi_{*}\right)dz-\frac{\log{\left|y\right|}}{2\pi}\int_{B_{1}\left(0\right)}\phi_{j1}\left(z,\xi_{*}\right)dz.
    \end{align*}
    Note that from \eqref{phi-j1-l-formula}, the second integrand on the right hand side above is a sum of mode $2$, mode $3$, and mode $4$ terms with respect to $z$. Thus the second integral is equal to $0$. Next, expanding 
    \begin{align*}
        \log\left(1-\frac{2z\cdot y}{\left|y\right|^{2}}+\frac{\left|z\right|^{2}}{\left|y\right|^{2}}\right)
    \end{align*}
    as a series gives, with respect to coordinates $z$, a mode $1$ term of order $\bigO{\left(\left|y\right|^{-1}\right)}$, a mode $2$ term of order $\bigO{\left(\left|y\right|^{-2}\right)}$, and a remainder of order $\bigO{\left(\left|y\right|^{-3}\right)}$. The mode $1$ term when integrated against $\phi_{j1}$ again evaluates to $0$ due, again, to the fact that $\phi_{j1}$ is a finite sum of mode $2,3,$ and $4$ terms with respect to $z$. Thus the first term to give a nonzero contribution is the mode $2$ term, and hence
    \begin{align*}
        \psi_{j1}\left(y,\xi_{*}\right)=\bigO{\left(\varepsilon^{2}t^{-1}\left|y\right|^{-2}\right)}=\bigO{\left(\varepsilon^{4}t^{-2}\right)}
    \end{align*}
    on the region $\left|y\right|\geq C\varepsilon^{-1}t^{\frac{1}{2}}$.
    For $\nabla\psi_{j1}$ on this region, first note
    \begin{align*}
        \nabla_{x}\psi_{j1}\left(\frac{x-\xi_{j}}{\varepsilon},\xi_{*}\right)=\frac{1}{\varepsilon}\nabla_{y}\psi_{j1}\left(y,\xi_{*}\right).
    \end{align*}
    Then note that
    \begin{align*}
        \nabla_{y}\log\left(1-\frac{2z\cdot y}{\left|y\right|^{2}}+\frac{\left|z\right|^{2}}{\left|y\right|^{2}}\right)
    \end{align*}
    is, with respect to the coordinate $z$, a mode $1$ term of order $\bigO{\left(\left|y\right|^{-2}\right)}$, and a remainder of order $\bigO{\left(\left|y\right|^{-3}\right)}$. Thus, for $\left|y\right|\geq C\varepsilon^{-1}t^{\frac{1}{2}}$,
    \begin{align*}
        \frac{1}{\varepsilon}\nabla_{y}\psi_{j1}\left(y,\xi_{*}\right)=\bigO{\left(\varepsilon t^{-1}\left|y\right|^{-3}\right)}=\bigO{\left(\varepsilon^{4}t^{-\frac{5}{2}}\right)}.
    \end{align*}
    Finally, noting that $\nabla_{x}\eta_{j}$ and $\Delta_{x}\eta_{j}$ are of order $\bigO{\left(t^{-\frac{1}{2}}\right)}$ and $\bigO{\left(t^{-1}\right)}$ respectively, we obtain that the right hand side of \eqref{psi-out-equation} is $\bigO{\left(\varepsilon^{4}t^{-3}\right)}$. Thus, noting that this right hand side is supported on an annulus with inner and outer radius $\sim\sqrt{t}$, implying the area of its support is $\sim t$, we have that
    \begin{align}
        \left\|\left(\log{\left(\left|\cdot\right|+2\right)}\right)^{-1}\psi^{out}\right\|_{L^{\infty}\left(\mathbb{R}^{2}\right)}+\left\|\nabla\psi^{out}\right\|_{L^{\infty}\left(\mathbb{R}^{2}\right)}\leq\frac{C\varepsilon^{4}}{t^{2}}.\label{first-approximation-construction-22}
    \end{align}
    Now we estimate $\gradperp_{y}\psi\left(\varepsilon y+\xi_{j}\right)\cdot\nabla_{y}U(y)$ for each $j=1,2,3$.  First, we see that 
    \begin{align*}
        \gradperp_{y}\psi^{out}\left(\varepsilon y+\xi_{j}\right)\cdot\grad_{y}U=-\gamma\Gamma^{\gamma-1}_{+}(r)\gradperp_{y}\Gamma\cdot\nabla_{y}\psi^{out}=\gamma\Gamma^{\gamma-1}_{+}(r)\frac{\Gamma'(r)}{r}\dell_{\theta}\psi^{out}\left(\varepsilon y+\xi_{j}\right),
    \end{align*}
    Next, note that since $\supp{U}=B_{1}(0)$, and $\psi^{out}$ satisfies \eqref{psi-out-equation}, we can write, for $\left|y\right|\leq 1$ 
    \begin{align*}
        \psi^{out}\left(\varepsilon y+\xi_{j}\right)&=-\frac{1}{4\pi}\int_{\left|z-\xi_{j}\right|\geq K\left(1+\frac{t}{\tau}\right)^{\frac{1}{2}}}\log{\left(1-\frac{2\varepsilon\left(z-\xi_{j}\right).y}{\left|z-\xi_{j}\right|^{2}}+\frac{\varepsilon^{2}\left|y\right|^{2}}{\left|z-\xi_{j}\right|^{2}}\right)}\mathscr{S}(z)\ dz\nonumber\\
        &-\frac{1}{2\pi}\int_{\left|z-\xi_{j}\right|\geq K\left(1+\frac{t}{\tau}\right)^{\frac{1}{2}}}\log{\left|z-\xi_{j}\right|}\ \mathscr{S}(z)\ dz,
    \end{align*}
    where $\mathscr{S}$ is the source term on the right hand side of \eqref{psi-out-equation}, and as stated above, is supported on three annuli of areas $\sim t$ each and has order $\bigO{\left(\varepsilon^{4} t^{-3}\right)}$. Then, taking an angular derivative of both sides above clearly makes the second term on the right hand side disappear, and note that on the region $\left|y\right|\leq 1$, we have that
    \begin{align*}
        \log{\left(1-\frac{2\varepsilon\left(z-\xi_{j}\right).y}{\left|z-\xi_{j}\right|^{2}}+\frac{\varepsilon^{2}\left|y\right|^{2}}{\left|z-\xi_{j}\right|^{2}}\right)}=\bigO{\left(\frac{\varepsilon}{t^{\frac{1}{2}}}\right)},
    \end{align*}
    then we have the estimate
    \begin{align}
        \left|\gradperp_{y}\psi^{out}\left(\varepsilon y+\xi_{j}\right)\cdot\nabla_{y}U(y)\right|\leq \frac{C\varepsilon^{5}}{t^{\frac{5}{2}}}.\label{psi-out-grad-y-bound}
    \end{align}
    \medskip
    \textbf{Step IV: Constructing $\psi_{j2}$}

    \medskip
    From \eqref{E-j6-definition} and \eqref{psi-j1-ODEs}, we can rewrite \eqref{first-approximation-construction-8} as
    \begin{align}
        &E_{j1}=\left(\varepsilon^{2}\dell_{t}\phi_{j2}+\kappa_{j}\gradperp_{y}\psi_{j2}\cdot\grad_{y}\phi_{j}+\kappa_{j}\gradperp_{y}\psi_{j1}\cdot\grad_{y}\phi_{j2}\right)+\varepsilon^{2}\dell_{t}\left(\phi_{j1}^{\left(3\right)}+\phi_{j1}^{\left(4\right)}\right)\nonumber\\
        &+\kappa_{j}\gradperp_{y}\left(\sum_{l=3}^{4}\psi^{\left(l\right)}_{j1}\right)\cdot\grad_{y}\phi^{\left(2\right)}_{j1}+\kappa_{j}\gradperp_{y}\psi^{\left(2\right)}_{j1}\cdot\grad_{y}\left(\sum_{l=3}^{4}\phi^{\left(l\right)}_{j1}\right)+\kappa_{j}\gradperp_{y}\left(\sum_{l=3}^{4}\psi^{\left(l\right)}_{j1}\right)\cdot\grad_{y}\left(\sum_{l=3}^{4}\phi^{\left(l\right)}_{j1}\right)\nonumber\\
        &+\gradperp_{y}\left(\sum_{i\neq j}\frac{m_{i}}{4\pi}\left(\mathcal{E}^{\left(ji\right)}_{3}\left(\xi_{*}\right)+\mathcal{E}^{\left(ji\right)}_{4}\left(\xi_{*}\right)\right)\right)\cdot\grad_{y}\phi_{j1}+\gradperp_{y}\left(\sum_{i\neq j}\frac{m_{i}}{4\pi}\mathcal{E}^{\left(ji\right)}_{2}(y,\xi_{*})\right)\cdot\grad_{y}\left(\phi_{j1}^{\left(3\right)}+\phi_{j1}^{\left(4\right)}\right)\nonumber\\
        &-\kappa_{j}\gradperp_{y}\Gamma\cdot\grad_{y}\left(\Delta_{y}\psi_{j2}+\gamma\Gamma^{\gamma-1}_{+}\psi_{j2}\right)+\gradperp_{y}\left(\sum_{i\neq j}\frac{m_{i}}{4\pi}\mathcal{E}^{\left(ji\right)}_{2}(y,\xi_{*})\right)\cdot\grad_{y}\phi_{j1}^{\left(2\right)}\nonumber\\
        &+\varepsilon^{2}\dell_{t}\phi_{j1}^{\left(2\right)}+\kappa_{j}\gradperp_{y}\psi^{\left(2\right)}_{j1}\cdot\grad_{y}\phi^{\left(2\right)}_{j1}+\nabla_{y}\left(U+\phi_{j}\right)\cdot\gradperp_{y}\mathcal{R}_{j6}+\gradperp_{y}\mathcal{E}_{j6}\cdot\grad_{y}\phi_{j2}\nonumber\\
        &+\gradperp_{y}\psi^{out}\cdot\grad_{y}\left(U+\phi_{j}\right)-\varepsilon\grad_{y}\left(U+\phi_{j}\right)\cdot\left(\dot{\tilde{\xi}}_{j}+\mathcal{N}_{j}\left(\xi_{*}\right)\left[\tilde{\xi}\right]\right).\label{first-approximation-construction-22a}
    \end{align}
    Now, to construct $\psi_{j2}$, and therefore $\phi_{j2}$, we would like to solve the elliptic problem given by
    \begin{align}
        \kappa_{j}\gradperp_{y}\Gamma\cdot\grad_{y}\left(\Delta_{y}\psi_{j2}+\gamma\Gamma^{\gamma-1}_{+}\psi_{j2}\right)&=\gradperp_{y}\left(\sum_{i\neq j}\frac{m_{i}}{4\pi}\mathcal{E}^{\left(ji\right)}_{2}(y,\xi_{*})\right)\cdot\grad_{y}\phi_{j1}^{\left(2\right)}\nonumber\\
        &+\varepsilon^{2}\dell_{t}\phi_{j1}^{\left(2\right)}+\kappa_{j}\gradperp_{y}\psi^{\left(2\right)}_{j1}\cdot\grad_{y}\phi^{\left(2\right)}_{j1}\label{psi-j2-equation}
    \end{align}
    on $\mathbb{R}^{2}$. We note that the first and third terms on the right hand side of \eqref{psi-j2-equation} are a sum of mode $4$ terms of order $\varepsilon^{4}t^{-4}$ on $B_{1}(0)$, and using  \eqref{phi-j1-l-formula}, the second term is a sum of mode $2$ terms of order $\varepsilon^{4}t^{-4}$ on $B_{1}(0)$.

    \medskip
    
    The form of the $\phi_{j1}^{\left(l\right)}$ in \eqref{phi-j1-l-formula} implies that the first three terms come with either a factor of either $\gamma\Gamma^{\gamma-1}_{+}(y)$ or $\gamma\left(\gamma-1\right)\Gamma^{\gamma-2}_{+}(y)$, and this is also clear for the last term from above, and so the terms on the right hand side are indeed are supported on $B_{1}(0)$.

    \medskip
    
    Then, noting that both $\Gamma'(r)/r$, its reciprocal, and both of their gradients are bounded at the origin, we can rewrite \eqref{psi-j2-equation} as
    \begin{align}
        \dell_{\theta}\left(\Delta_{y}\psi_{j2}+\gamma\Gamma^{\gamma-1}_{+}\psi_{j2}\right)&=-\frac{r}{\kappa_{j}\Gamma'(r)}\gradperp_{y}\left(\sum_{i\neq j}\frac{m_{i}}{4\pi}\mathcal{E}^{\left(ji\right)}_{2}(y,\xi_{*})\right)\cdot\grad_{y}\phi_{j1}^{\left(2\right)}\nonumber\\
        &-\frac{r}{\kappa_{j}\Gamma'(r)}\varepsilon^{2}\dell_{t}\phi_{j1}^{\left(2\right)}-\frac{r}{\Gamma'(r)}\gradperp_{y}\psi^{\left(2\right)}_{j1}\cdot\grad_{y}\phi^{\left(2\right)}_{j1}.\label{psi-j2-equation-2}
    \end{align}
    Thus, splitting into the relevant modes as when constructing $\psi_{j1}$, we wish to solve, for $l=2,4$, the equations
    \begin{align}
        \Delta_{y}\psi_{j2}^{\left(l\right)}+\gamma\Gamma^{\gamma-1}_{+}\psi_{j2}^{\left(l\right)}=\mathcal{F}_{j}^{\left(l\right)},\label{psi-j2-modal-equations}
    \end{align}
    where the $\mathcal{F}_{j}^{\left(l\right)}$ are mode $l$ terms corresponding to the antiderivatives in $\theta$ of the mode $l$ terms for the three terms on the right hand side of \eqref{psi-j2-equation-2}. We impose the necessary bounds on integration to ensure no mode $0$ terms on the right hand side of \eqref{psi-j2-modal-equations}, which is possible as we only require equality in the sense of \eqref{psi-j2-equation-2} to be satisfied. 
    
    \medskip
    Then, since one can show that the $\mathcal{F}_{j}^{\left(l\right)}$ also behave as $r^{\left|l\right|}$ as $r\to0$, we can once again apply Lemma \ref{vortex-linearised-equation-fourier-coefficients-behaviour-lemma}, with a slight modification to accommodate right hand sides with a factor of $\gamma\left(\gamma-1\right)\Gamma^{\gamma-2}_{+}(y)$ as well as $\gamma\Gamma^{\gamma-1}_{+}(y)$, to obtain existence and uniqueness for $\psi_{j2}^{\left(l\right)}$ solving \eqref{psi-j2-modal-equations}, along with bounds
    \begin{align}
        \|\psi_{j2}^{\left(l\right)}\|_{L^{\infty}\left(\mathbb{R}^{2}\right)}+\|\nabla\psi_{j2}^{\left(l\right)}\|_{L^{\infty}\left(\mathbb{R}^{2}\right)}+\|\phi_{j2}^{\left(l\right)}\|_{L^{\infty}\left(\mathbb{R}^{2}\right)}\leq \frac{C\varepsilon^{4}}{t^{2}},\label{psi-j2-elliptic-estimates}
    \end{align}
    for $l=2,4$. Once again, $-\Delta_{y}\psi_{j2}^{\left(l\right)}=\phi_{j2}^{\left(l\right)}$, and one can use this fact as well as \eqref{psi-j2-equation-2} and \eqref{psi-j2-modal-equations} to show that
    \begin{align}
        \phi_{j2}^{\left(l\right)}=\gamma\Gamma^{\gamma-1}_{+}(y)\mathcal{H}_{1}^{\left(l\right)}\left(y,\xi_{*}\right)+\gamma\left(\gamma-1\right)\Gamma^{\gamma-2}_{+}(y)\mathcal{H}_{2}^{\left(l\right)}\left(y,\xi_{*}\right),\label{form-of-phi-j2}
    \end{align}
    where the $\mathcal{H}_{i}^{\left(l\right)}$ are $C^{1}$ functions in $y$, of order $\varepsilon^{4}t^{-2}$ on $B_{1}(0)$. Then, define
    \begin{align}
        \psi_{j2}=\sum_{l=2,4}\psi_{j2}^{\left(l\right)}\quad \phi_{j2}=\sum_{l=2,4}\phi_{j2}^{\left(l\right)}.\label{psi-j2-phi-j2-definitions}
    \end{align}
    Given $\psi_{j2}$ by construction satisfies \eqref{psi-j2-equation}, we can rewrite $E_{j1}$ from \eqref{first-approximation-construction-22a} as
    \begin{align}
        &E_{j1}=\left(\varepsilon^{2}\dell_{t}\phi_{j2}+\kappa_{j}\gradperp_{y}\psi_{j2}\cdot\grad_{y}\phi_{j}+\kappa_{j}\gradperp_{y}\psi_{j1}\cdot\grad_{y}\phi_{j2}\right)+\varepsilon^{2}\dell_{t}\left(\phi_{j1}^{\left(3\right)}+\phi_{j1}^{\left(4\right)}\right)\nonumber\\
        &+\kappa_{j}\gradperp_{y}\left(\sum_{l=3}^{4}\psi^{\left(l\right)}_{j1}\right)\cdot\grad_{y}\phi^{\left(2\right)}_{j1}+\kappa_{j}\gradperp_{y}\psi^{\left(2\right)}_{j1}\cdot\grad_{y}\left(\sum_{l=3}^{4}\phi^{\left(l\right)}_{j1}\right)+\kappa_{j}\gradperp_{y}\left(\sum_{l=3}^{4}\psi^{\left(l\right)}_{j1}\right)\cdot\grad_{y}\left(\sum_{l=3}^{4}\phi^{\left(l\right)}_{j1}\right)\nonumber\\
        &+\gradperp_{y}\left(\sum_{i\neq j}\frac{m_{i}}{4\pi}\left(\mathcal{E}^{\left(ji\right)}_{3}\left(\xi_{*}\right)+\mathcal{E}^{\left(ji\right)}_{4}\left(\xi_{*}\right)\right)\right)\cdot\grad_{y}\phi_{j1}+\gradperp_{y}\left(\sum_{i\neq j}\frac{m_{i}}{4\pi}\mathcal{E}^{\left(ji\right)}_{2}(y,\xi_{*})\right)\cdot\grad_{y}\left(\phi_{j1}^{\left(3\right)}+\phi_{j1}^{\left(4\right)}\right)\nonumber\\
        &+\nabla_{y}\left(U+\phi_{j}\right)\cdot\gradperp_{y}\mathcal{R}_{j6}+\gradperp_{y}\mathcal{E}_{j6}\cdot\grad_{y}\phi_{j2}+\gradperp_{y}\psi^{out}\cdot\grad_{y}\left(U+\phi_{j}\right)\nonumber\\
        &-\varepsilon\grad_{y}\left(U+\phi_{j}\right)\cdot\left(\dot{\tilde{\xi}}_{j}+\mathcal{N}_{j}\left(\xi_{*}\right)\left[\tilde{\xi}\right]\right).\label{first-approximation-construction-22b}
    \end{align}
    Then, using \eqref{phi-j1-l-formula}, \eqref{psi-j1-bounds}, \eqref{psi-out-equation}, \eqref{psi-out-grad-y-bound}, \eqref{psi-j2-elliptic-estimates}, and \eqref{form-of-phi-j2}, we can further rewrite \eqref{first-approximation-construction-22b} as
    \begin{align}
        &E_{j1}\left(\phi_{j},\psi_{j},\psi^{out};\xi\right)=-\varepsilon\nabla_{y}\left(U\left(y\right)+\phi_{j}(y)\right)\cdot\left(\dot{\tilde{\xi}}_{j}+\mathcal{N}_{j}\left(\xi_{*}\right)\left[\tilde{\xi}\right]\right)\nonumber\\
        &+\gamma\Gamma^{\gamma-1}_{+}\left(y\right)\mathcal{R}^{*}_{j1}\left(y,\xi\right)+\gamma\left(\gamma-1\right)\Gamma^{\gamma-2}_{+}\left(y\right)\mathcal{R}^{*}_{j2}\left(y,\xi\right)+\gamma\left(\gamma-1\right)\left(\gamma-2\right)\Gamma^{\gamma-3}_{+}\left(y\right)\mathcal{R}^{*}_{j3}\left(y,\xi\right),\label{first-approximation-construction-22c}
    \end{align}
    where, for $l=1,2,3$ and $j=1,2,3$
    \begin{align*}
        \left|\dell_{\Gamma}^{l}\left(\Gamma^{\gamma}_{+}\right)\mathcal{R}^{*}_{jl}\left(y,\xi\right)\right|&\leq\frac{C\varepsilon^{5}\dell_{\Gamma}^{l}\left(\Gamma^{\gamma}_{+}\right)}{t^{\frac{5}{2}}}.
    \end{align*}
    Finally, due to the fact that $\phi_{j2}$ is a sum of mode $\neq0$ terms from \eqref{psi-j2-modal-equations}--\eqref{form-of-phi-j2}, we can do a modal analysis similar to when constructing $\psi^{out}$ in Step III, and find that due \eqref{psi-j2-elliptic-estimates}, we have
    \begin{align}
        E_{2}^{out}=\sum_{j=1}^{3}\kappa_{j}\left(\psi_{j2}\Delta_{x}\eta_{j}+2\nabla_{x}\eta_{j}\cdot\nabla_{x}\psi_{j2}\right)=\bigO{\left(\frac{\varepsilon^{6}}{t^{4}}\right)}.\label{first-approximation-outer-error-E2out-final-bound}
    \end{align}
    Then \eqref{phi-j1-l-formula}--\eqref{first-approximation-outer-error-E2out-final-bound} give the Theorem.
    \end{proof}
    \begin{remark}\label{first-approximation-zero-mass}
        Noting that every term on the right hand side of \eqref{first-approximation-construction-22b} is supported on $B_{1}(0)$ in $y$ coordinates, and integrating both \eqref{first-approximation-construction-22b} and \eqref{first-approximation-construction-22c} on $B_{2}(0)$ in $y$ coordinates and equating, it is immediate that
        \begin{align}
            &\int_{B_{2}(0)}\left(\gamma\Gamma^{\gamma-1}_{+}\left(y\right)\mathcal{R}^{*}_{j1}+\gamma\left(\gamma-1\right)\Gamma^{\gamma-2}_{+}\left(y\right)\mathcal{R}^{*}_{j2}+\gamma\left(\gamma-1\right)\left(\gamma-2\right)\Gamma^{\gamma-3}_{+}\left(y\right)\mathcal{R}^{*}_{j3}\right) dy\nonumber\nonumber\\
            &=\varepsilon^{2}\int_{B_{2}(0)}\dell_{t}\left(\phi_{j2}+\phi_{j1}^{\left(3\right)}+\phi_{j1}^{\left(4\right)}\right)dy.\label{first-approximation-zero-mass-1}
        \end{align}
        Then, using \eqref{phi-j1-l-formula}, \eqref{psi-j2-modal-equations}, and \eqref{psi-j2-phi-j2-definitions}, we can see that the right hand side of \eqref{first-approximation-zero-mass-1} is an integral of a finite sum of mode $2,3,$ and $4$ terms, each supported on $B_{1}(0)$ and so this integral is immediately $0$. Thus
        \begin{align}
            \int_{B_{2}(0)}\left(\gamma\Gamma^{\gamma-1}_{+}\left(y\right)\mathcal{R}^{*}_{j1}+\gamma\left(\gamma-1\right)\Gamma^{\gamma-2}_{+}\left(y\right)\mathcal{R}^{*}_{j2}+\gamma\left(\gamma-1\right)\left(\gamma-2\right)\Gamma^{\gamma-3}_{+}\left(y\right)\mathcal{R}^{*}_{j3}\right) dy=0,\label{first-approximation-zero-mass-2}
        \end{align}
        a fact that we will use to make the construction of our solutions on $\left[T_{0},T\right]$ simpler.
    \end{remark}
\section{Constructing Solutions on $\left[T_{0},T\right]$}\label{constructing-full-solution-section}
Having constructed the first approximation $\left(\omega_{*},\Psi_{*}\right)$ of a solution to \eqref{2d-euler-vorticity-stream} on $\left[T_{0},T\right]$ in Theorem \ref{first-approximation-construction-theorem}, we now look for a solution on the same time interval to \eqref{2d-euler-vorticity-stream} of the form
\begin{align}
    \omega\left(x,t\right)&=\omega_{*}\left(x,t\right)+\phi_{*}\left(x,t\right),\label{omega-full-solution-decomposition}\\
    \Psi\left(x,t\right)&=\Psi_{*}\left(x,t\right)+\psi_{*}\left(x,t\right).\label{Psi-full-solution-decomposition}
\end{align}
We set
\begin{align*}
        \psi_{*}&=\sum_{j=1}^{3}\kappa_{j}\eta_{j}\psi_{*j}\left(\frac{x-\xi_{j}(t)}{\varepsilon},t\right)+\psi^{out}_*, \\
        \phi_{*}&=\varepsilon^{-2}\sum_{j=1}^{3}\kappa_{j}\phi_{*j}\left(\frac{x-\xi_{j}(t)}{\varepsilon},t\right),
\end{align*}
with $\eta_{j}$ defined in \eqref{cutoffs-definition-1}. Each $\phi_{*j}$ will have support contained in $B_{2\varepsilon}\left(\xi_{j}(t)\right)$. We define
\begin{align*}
    \psi_*^{in} = (\psi_{*1},\psi_{*2},\psi_{*3}), \quad \psi^{out}_*, \quad \phi_*^{in} = (\phi_{*1},\phi_{*2},\phi_{*3}).
\end{align*}
Finally we again have $\xi_{j}(t)$, $j=1,2,3$ as in \eqref{point-vortex-trajectory-function-decomposition}:
\begin{align*}
    \xi_{j}(t)=\xi_{*j}(t)+\tilde{\xi}_{j}(t),
\end{align*}
for $\xi_{*j}$ solving \eqref{KR}, and $\xi_{*j}$, and $\Tilde{\xi}_{j}$ satisfying \eqref{xi0-bound}--\eqref{tildexi-bound}. We also recall the definition of $\xi=\left(\xi_{1},\xi_{2},\xi_{3}\right)$ from \eqref{point-vortex-decomposition-vector-form}, with analogous definitions for $\xi_{*}$, $\tilde{\xi}$. 

\medskip
We also impose three orthogonality conditions on $\phi_{*j}$, $0$ total mass, and centre of mass at the origin:
\begin{align}
    \int_{\mathbb{R}^{2}}\phi_{*j}\ dy=0,\quad \int_{\mathbb{R}^{2}}y_{1}\phi_{*j}\ dy=0,\quad \int_{\mathbb{R}^{2}}y_{2}\phi_{*j}\ dy=0.\label{tilde-phi-star-R-0-mass-condition}
\end{align}
On $\mathbb{R}^{2}$, in $y$ variables, $-\Delta_{y}\psi_{*j}=\phi_{*j}$. That is,
\begin{align}
    \psi_{*j}\left(y,t\right)=-\frac{1}{2\pi}\int_{\mathbb{R}^{2}}\log{\left|z-y\right|}\ \phi_{*j}\left(z,t\right) dz.\label{psi-star-r-green-function-representation-upper-half-plane}
\end{align}
We let
\be 
\phi_{*} =\left(\phi_{*1},\phi_{*2},\phi_{*3}\right), \quad \psi_{*}=\left(\psi_{*1},\psi_{*2}, \psi_{*3}\right).\label{tilde-psi-star-definition}
\ee

\subsection{Homotopic Operators}\label{full-linearised-euler-operator-section}
In this section we define, with motivation, operators depending on a homotopy parameter $\lambda\in\left[0,1\right]$. At $\lambda=0$, these operators will be linear, and at $\lambda=1$, their annihilation will imply a solution to the Euler system \eqref{2d-euler-vorticity-stream}. To begin with, we insert $\left(\omega,\Psi\right)$, defined in \eqref{omega-full-solution-decomposition}--\eqref{Psi-full-solution-decomposition} into $E_{1}$ and $E_{2}$ defined in \eqref{first-approximation-error-calculation-1a}--\eqref{first-error-approximation-calculation-1b}:
\begin{align}\nonumber
E_{1}\left[\omega,\Psi\right]&=\dell_{t}\left(\omega_{*}+\phi_{*}\right)+\gradperp_{x}\left(\Psi_{*}+\psi_{*}\right)\cdot\nabla_{x}\left(\omega_{*}+\phi_{*}\right),
\\\nonumber
E_{2}\left[\omega,\Psi\right]&=\Delta_{x}\left(\Psi_{*}+\psi_{*}\right)+\left(\omega_{*}+\phi_{*}\right).
\end{align}
Similarly to \eqref{first-approximation-inner-error}--\eqref{first-approximation-outer-error-E2out-definition}, we have
\begin{align}\nonumber 
E_{1}\left[\omega,\Psi\right]&=\varepsilon^{-4}\sum_{j=1}^{3}\kappa_{j}E_{*j}\left(\phi_{*j},\psi_{*j},\psi_*^{out};\xi\right),\nonumber
\end{align}
and
\begin{align}
\nonumber E_{2}\left[\omega,\Psi\right]=E_{*}^{out}\left(\psi^{in}_{*}, \psi_{*}^{out};\xi\right),
\end{align}
where,
\begin{align} 
E_{*j}&=\varepsilon^{2}\dell_{t}\left(U_{j}+\phi_{j}+\phi_{*j}\right)
\label{homotopic-operator-inner-error-E-star-R-definition}
\\
&+\varepsilon^{2}\gradperp_{x}\left(\Psi_{0}+\kappa_{j}\psi_{j}+\kappa_{j}\psi_{*j}+\psi^{out}+\psi^{out}_{*}\right)\cdot\nabla_{x}\left(U_{j}+\phi_{j}+\phi_{*j}\right),\nonumber
\end{align}
\begin{align}
E_{*}^{out}&=\Delta_{x}\psi_*^{out}+\sum_{j=1}^{3}\kappa_{j}\left(\psi_{*j}\Delta_{x}\eta_{j}+2\nabla_{x}\eta_{j}\cdot\nabla_{x}\psi_{*j}\right)+E_{2}^{out}\left(\psi^{in},\psi^{out};\xi\right)\label{homotopic-operator-outer-error-E-star-out-definition}
\end{align}
with $\Psi_{0}$ defined in \eqref{psi-0-definition}, and $U_{j}$ defined in \eqref{gamma-j-u-j-definition}. Next, $\psi_{j}$, $\phi_{j}$, $j=1,2,3$, and $\psi^{out}$ constructed in the proof of Theorem \ref{first-approximation-construction-theorem}, and $E_{2}^{out}$ is given in \eqref{first-approximation-outer-error-E2out-final-bound}. We recall that the specific forms of $E_{j1}$ and $E_{2}^{out}$ given in \eqref{first-approximation-inner-error-ER1-definition}--\eqref{first-approximation-outer-error-E2out-definition} are due to the choice of $K$ and $T_{0}$ in \eqref{K-definition}--\eqref{T0-definition} implying $\eta_{j}\equiv1$ on $\supp{\phi_{j}}$ for $j=1,2,3$, $\eta_{k}\equiv0$ on $\supp{\phi_{j}}$, where $j,k=1,2,3$, $j\neq k$. Due to \eqref{tildexi-bound} , the supports of $\phi_{*j}$ will be close enough to the supports of the $U_{j}$ to imply the same behaviour, hence the specific forms of $E_{*j}$ and $E_{*}^{out}$ given in \eqref{homotopic-operator-inner-error-E-star-R-definition}--\eqref{homotopic-operator-outer-error-E-star-out-definition}.

\medskip
Now we consider
\begin{align}
\nonumber E_{*j}&=\varepsilon^{2}\dell_{t}\left(U_{j}+\phi_{j}+\phi_{*j}\right)\\
&+\varepsilon^{2}\gradperp_{x}\left(\Psi_{0}+\kappa_{j}\psi_{j}+\kappa_{j}\psi_{*j}+\psi^{out}+\psi^{out}_{*}\right)\cdot\nabla_{x}\left(U_{j}+\phi_{j}+\phi_{*j}\right).\nonumber
\end{align}
Regrouping terms, this becomes
\begin{align}
E_{*j}&=\varepsilon^{2}\dell_{t}\left(U_{j}+\phi_{j}\right)+\varepsilon^{2}\gradperp_{x}\left(\Psi_{0}+\kappa_{j}\psi_{j}+\psi^{out}\right)\cdot\nabla_{x}\left(U_{j}+\phi_{j}\right)+\varepsilon^{2}\dell_{t}\phi_{*j}\label{homotopic-operator-inner-error-E-star-R-definition-3}
\\
&+\varepsilon^{2}\gradperp_{x}\left(\Psi_{0}+\kappa_{j}\psi_{j}+\kappa_{j}\psi_{*j}+\psi^{out}+\psi^{out}_{*}\right)\cdot\nabla_{x}\phi_{*j}+\varepsilon^{2}\gradperp_{x}\left(\kappa_{j}\psi_{*j}+\psi^{out}_{*}\right)\cdot\nabla_{x}\left(U_{j}+\phi_{j}\right).\nonumber
\end{align}
Using Theorem \ref{first-approximation-construction-theorem}, in $y$ coordinates, \eqref{homotopic-operator-inner-error-E-star-R-definition-3} becomes
\begin{align}
    &E_{*j}=-\varepsilon\nabla_{y}\left(U+\phi_{j}\right)\cdot\left(\dot{\tilde{\xi}}_{j}+\mathcal{N}_{j}\left(\xi_{*}\right)\left[\tilde{\xi}\right]\right)+\gradperp_{y}\psi^{out}_{*}\cdot\nabla_{y}U\nonumber\\
    &+\gamma\Gamma^{\gamma-1}_{+}\left(y\right)\mathcal{R}^{*}_{j1}\left(y,\xi\right)+\gamma\left(\gamma-1\right)\Gamma^{\gamma-2}_{+}\left(y\right)\mathcal{R}^{*}_{j2}\left(y,\xi\right)+\gamma\left(\gamma-1\right)\left(\gamma-2\right)\Gamma^{\gamma-3}_{+}\left(y\right)\mathcal{R}^{*}_{j3}\left(y,\xi\right)\label{homotopic-operator-inner-error-E-star-R-definition-5}\\
    &+\gradperp_{y}\psi^{out}_{*}\cdot\nabla_{y}\phi_{j}+\varepsilon^{2}\dell_{t}\phi_{*j}\left(y,t\right)+\kappa_{j}\gradperp_{y}\psi_{*j}\cdot\nabla_{y}\left(U+\phi_{j}\right)\nonumber\\
    &+\gradperp_{y}\left(\Psi_{0}-\varepsilon \dot{\xi}\cdot y^{\perp}+\kappa_{j}\psi_{j}+\kappa_{j}\psi_{*j}+\psi^{out}+\psi^{out}_{*}\right)\cdot\nabla_{y}\phi_{*j}.\nonumber
\end{align}
where the $\mathcal{R}^{*}_{j}$ were constructed in the proof of Theorem \ref{first-approximation-construction-theorem}.

\medskip
We now concentrate on the last four terms of \eqref{homotopic-operator-inner-error-E-star-R-definition-5}.  These terms can be written as
\begin{align}
    &\varepsilon^{2}\dell_{t}\phi_{*j}\left(y,t\right)+\kappa_{j}\gradperp_{y}\psi_{*j}\cdot\nabla_{y}\left(U+\phi_{j}\right)+\gradperp_{y}\psi^{out}_{*}\cdot\nabla_{y}\phi_{j}\label{homotopic-operator-inner-error-E-star-R-definition-7}\\
    &+\gradperp_{y}\left(\Psi_{0}-\varepsilon \dot{\xi}\cdot y^{\perp}+\kappa_{j}\psi_{j}+\kappa_{j}\psi_{*j}+\psi^{out}+\psi^{out}_{*}\right)\cdot\nabla_{y}\phi_{*j}\nonumber
\end{align}
We make the decomposition
\begin{align}
    &\Psi_{0}-\varepsilon \dot{\xi}\cdot y^{\perp}+\kappa_{j}\psi_{j}+\kappa_{j}\psi_{*j}+\psi^{out}+\psi^{out}_{*}\label{homotopic-operator-inner-error-E-star-R-definition-10}\\
    &=\kappa_{j}\Gamma(y)+\underbrace{\sum_{i\neq j}\kappa_{i}\Gamma\left(y+\frac{\xi_{j}(t)-\xi_{i}(t)}{\varepsilon}\right)-\varepsilon \dot{\xi}\cdot y^{\perp}+\kappa_{j}\psi_{j}+\kappa_{j}\psi_{*j}+\psi^{out}+\psi^{out}_{*}}_{\mathfrak{a}_{j}}.\nonumber
\end{align}
Given that $U\left(y\right)$ is a function of $\Gamma(y)$, we define $\mathscr{E}_{j,\lambda}\left(\phi_{*j},\psi^{out}_{*},\Tilde{\xi}\right)$ by
\begin{align}
    &\mathscr{E}_{j,\lambda}\left(\phi_{*j},\psi^{out}_{*},\Tilde{\xi}\right)\coloneqq\varepsilon^{2}\dell_{t}\phi_{*j}\left(y,t\right)+\kappa_{j}\gradperp_{y}\psi_{*j}\cdot\nabla_{y}\left(U+\lambda\phi_{j}\right)\label{homotopic-operator-inner-error-main-definition}\\
    &+\gradperp_{y}\left(\kappa_{j}\Gamma+\lambda\mathfrak{a}_{j}\right)\cdot\nabla_{y}\phi_{*j}-\varepsilon\nabla_{y}\left(U+\phi_{j}\right)\cdot\left(\dot{\tilde{\xi}}_{j}+\mathcal{N}_{j}\left(\xi_{*}\right)\left[\tilde{\xi}\right]\right)\nonumber\\
    &+\lambda\tilde{\mathscr{E}}_{j}\left(\psi_{*j},\psi^{out}_{*},\Tilde{\xi}\right)+\gradperp_{y}\psi^{out}_{*}\cdot\nabla_{y}U\nonumber
\end{align}
where
\begin{align}
    \tilde{\mathscr{E}}_{j}\left(\psi_{*j},\psi^{out}_{*},\Tilde{\xi}\right)&=\gradperp_{y}\psi^{out}_{*}\cdot\nabla_{y}\phi_{j}+\gamma\Gamma^{\gamma-1}_{+}\left(y\right)\mathcal{R}^{*}_{j1}\left(y,\xi\right)+\gamma\left(\gamma-1\right)\Gamma^{\gamma-2}_{+}\left(y\right)\mathcal{R}^{*}_{j2}\left(y,\xi\right)\label{homotopic-operator-inner-error-main-definition-lower-order-terms}\\
    &+\gamma\left(\gamma-1\right)\left(\gamma-2\right)\Gamma^{\gamma-3}_{+}\left(y\right)\mathcal{R}^{*}_{j3}\left(y,\xi\right).\nonumber
\end{align}
We next define
\begin{align} 
    \mathscr{E}_{\lambda}^{out}\left(\tilde{\psi},\psi^{out}_{*},\tilde{\xi}\right)&=\Delta_{x}\psi_{*}^{out}+\lambda\sum_{j=1}^{3}\kappa_{j}\left(\psi_{*j}\Delta_{x}\eta_{j}+2\nabla_{x}\eta_{j}\cdot\nabla_{x}\psi_{*j}\right)+\lambda E_{2}^{out}\left(\psi^{in},\psi^{out};\xi\right).\label{homotopic-operator-outer-error-main-definition}
\end{align}
We note that at $\lambda=1$, we recover \eqref{homotopic-operator-inner-error-E-star-R-definition-5} and \eqref{homotopic-operator-outer-error-E-star-out-definition} respectively:
\begin{align}
    \mathscr{E}_{j,1}\left(\phi_{*j},\psi^{out}_{*},\Tilde{\xi}\right)&=E_{*j}\left(\phi_{j},\phi_{*j},\psi_{j},\psi_{*j},\psi^{out}, \psi_{*}^{out};\xi\right),\label{homotopic-operator-inner-error-lambda-1-equivalence}\\
    \mathscr{E}_{1}^{out}\left(\tilde{\psi},\psi^{out}_{*},\tilde{\xi}\right)&=E_{*}^{out}\left(\psi^{in},\psi^{in}_{*},\psi^{out}, \psi_{*}^{out};\xi\right).\label{homotopic-operator-outer-error-lambda-1-equivalence}
\end{align}
Thus a solution to the system \eqref{2d-euler-vorticity-stream} is constructed if we find  $\left(\phi_{*j},\psi^{out}_{*},\Tilde{\xi}\right)$ that make the quantities stated in \eqref{homotopic-operator-inner-error-lambda-1-equivalence}--\eqref{homotopic-operator-outer-error-lambda-1-equivalence} equal to $0$. We will do this by a continuation argument that involves finding a priori estimates along the deformation parameter $\lambda$ for the equations one obtains by setting \eqref{homotopic-operator-inner-error-main-definition} and \eqref{homotopic-operator-outer-error-main-definition} equal to $0$, in addition imposing suitable bounds for the parameter functions for $t\in\left[T_{0},T\right]$.

\medskip
Then, given $\left(\psi^{out}_{*},\Tilde{\xi}\right)$, we require that $\phi_{*j}$ satisfies the terminal value problem on $\mathbb{R}^{2}\times\left[T_{0},T\right]$,
\begin{subequations}\label{tilde-phi-star-R-initial-value-problem}
\begin{align}
&\mathscr{E}_{j,\lambda}\left(\phi_{*j},\psi^{out}_{*},\Tilde{\xi}\right)=c_{j1}\left(t\right)\dell_{1}U(y)+c_{j2}\left(t\right)\dell_{2}U(y)\label{tilde-phi-star-R-initial-value-problem-equation}\\
&\phi_{*j}(\blank,T)=0\ \ \text{in}\ \mathbb{R}^{2},\label{tilde-phi-star-R-initial-value-problem-initial-data}\\
&\phi_{*j}\to0,\ \left|y\right|\to\infty,\ \forall t\in\left[T_{0},T\right],\label{tilde-phi-star-R-initial-value-problem-boundary-condition}
\end{align}
\end{subequations}
where we recall that $T_{0}$ and $T$ are defined in \eqref{T0-definition}--\eqref{T-definition-2}. Upon integrating against $1$, $y_{1}$, and $y_{2}$ in turn, for $\phi_{*j}$ satisfying \eqref{tilde-phi-star-R-initial-value-problem}, we can derive explicit formulae for the $c_{jl}$ that depend linearly on $\phi_{*j}$ using integration by parts.

\medskip
Given a solution to \eqref{tilde-phi-star-R-initial-value-problem}, $\mathscr{E}_{j,\lambda}$ will be annihilated once we impose, for $l=1,2$, the terminal value problem
\begin{align}
    c_{jl}\left(\phi_{*j},\psi^{out}_{*},\Tilde{\xi},\lambda\right)=0\ \ \textrm{for all}\ t\in\left[T_{0},T\right],\label{cRj-initial-value-problem}
\end{align}
with $\Tilde{\xi}_{j}$ satisfying
\begin{align}
    \|t^{\frac{3}{2}}\tilde{\xi}_{j}\|_{\left[T_{0},T\right]}+\|t^{\frac{5}{2}}\dot{\tilde{\xi}}_{j}\|_{\left[T_{0},T\right]}\leq\varepsilon^{4-\sigma},\ \ \tilde{\xi}_{j}\left(T\right)=0\label{p-tilde-bound-2}
\end{align}
for some small $\sigma>0$. We also impose that $\psi^{out}_{*}$, in $x$ variables, solves the boundary value problem
\begin{align}
    \mathscr{E}_{\lambda}^{out}\left(\tilde{\psi},\psi^{out}_{*},\tilde{\xi}\right)&=0\ \ \text{in}\ \mathbb{R}^{2}\times \left[T_{0},T\right],\label{psi-out-star-boundary-value-problem-equation}
\end{align}
which we interpret to be, for all $t\in\left[T_{0},T\right]$,
\begin{align*}
    -\Delta_{x}\psi^{out}_{*}=\lambda\sum_{j=1}^{3}\kappa_{j}\left(\psi_{*j}\Delta_{x}\eta_{j}+2\nabla_{x}\eta_{j}\cdot\nabla_{x}\psi_{*j}\right)+\lambda E_{2}^{out}\left(\psi^{in},\psi^{out};\xi\right),
\end{align*}
on $\mathbb{R}^{2}$, where we note the source term on the right hand side has compact support for every fixed $t$.

\medskip
To construct a solution to \eqref{2d-euler-vorticity-stream} via the above scheme, we consider the vector of parameter functions $\mathfrak{P}=\left(\phi_{*}, \psi^{out}_{*},\Tilde{\xi}\right)$,  belonging to a Banach space $\left(\mathfrak{X},\left\|\cdot\right\|_{\mathfrak{X}}\right)$ and reformulate the equations \eqref{psi-star-r-green-function-representation-upper-half-plane}, \eqref{tilde-phi-star-R-initial-value-problem}--\eqref{psi-out-star-boundary-value-problem-equation} as a fixed point problem of the form
\begin{align}
    \mathcal{T}\left(\mathfrak{P},\lambda\right)=\mathfrak{P},\ \ \mathfrak{P}\in\mathscr{O},\label{full-solution-construction-fixed-point-formulation}
\end{align}
where $\mathscr{O}$ is a bounded open set in $\mathfrak{X}$, and $\mathcal{T}\left(\cdot,\lambda\right)$ is a homotopy of nonlinear compact operators with $\mathcal{T}\left(\cdot,0\right)$ being linear.

\medskip
For a suitable choice of $\mathfrak{X}$ and $\mathscr{O}$, we will show that for all $\lambda\in[0,1]$, no solution $\mathcal{Z}\in\dell\mathscr{O}$ exists. Existence of a solution to \eqref{full-solution-construction-fixed-point-formulation} at $\lambda=1$ will then follow by standard degree theory, which, as stated, will give a solution to \eqref{2d-euler-vorticity-stream}. The choices of $\mathfrak{X}$ and $\mathscr{O}$ will also give the desired properties we claimed the solution would possess in Theorem \ref{teo1}.
\begin{remark}\label{no-alpha-needed-remark}
    Note that unlike \cite{DDMW2020,DDMP2VP2023}, we do not need an ansatz of the form $\phi_{*j}=\tilde{\phi}_{*j}+\alpha_{j}(t)U$ to construct a solution. Equivalently, the projected linear problem \eqref{tilde-phi-star-R-initial-value-problem} does not require a term on the right hand side related to the mass.

    \medskip
    Indeed, if we integrate \eqref{homotopic-operator-inner-error-main-definition}, using \eqref{first-approximation-zero-mass-1}--\eqref{first-approximation-zero-mass-2} from Remark \ref{first-approximation-zero-mass}, and \eqref{homotopic-operator-inner-error-main-definition-lower-order-terms}, we obtain
    \begin{align*}
       \int_{\mathbb{R}^{2}}\mathscr{E}_{j,\lambda}dx=\int_{\mathbb{R}^{2}}\varepsilon^{2}\dell_{t}\phi_{*j}\ dx. 
    \end{align*}
    The reason this is the case is that unlike in \cite{DDMW2020}, we can construct the first approximation purely elliptically, and unlike in \cite{DDMP2VP2023}, the construction of the first approximation relies on inverting an operator coming from linearising around a radial profile. Both of these facts give us precise control over the modes of the first approximation which lets us obtain \eqref{first-approximation-zero-mass-2}, the key fact leading to this simplification.
\end{remark}
\medskip
We now move on to a priori estimates for the system  \eqref{psi-star-r-green-function-representation-upper-half-plane}, \eqref{tilde-phi-star-R-initial-value-problem}--\eqref{psi-out-star-boundary-value-problem-equation}. 

\subsection{Weighted $L^{2}$ Function Space}

When proving a priori estimates for some linearised operators, it will be convenient to have shorthand for several quantities we define in this section. Firstly, let $\eta_{0}(r)$ once again be the smooth cutoff that is identically $1$ for $r\leq1$ and identically $0$ for $r\geq2$. Define the interior cutoff, $\eta_{\mathfrak{i}}$, by
\begin{align}
    \eta_{\mathfrak{i}}\left(y,t\right)=1-\eta_{0}\left(\left(\frac{t}{\varepsilon}\right)^{\frac{\gamma}{4}}U(y)\right),\label{weighted-L2-interior-cutoff}
\end{align}
Then $\eta_{\mathfrak{i}}$ is supported on the region 
\begin{align}
    U(y)\geq2\left(\frac{\varepsilon}{t}\right)^{\frac{\gamma}{4}} \iff \Gamma(y)\geq2^{\frac{1}{\gamma}}\left(\frac{\varepsilon}{t}\right)^{\frac{1}{4}}.\label{weighted-L2-interior-cutoff-support}
\end{align}
For future convenience we also define the set
\begin{align}
    \mathfrak{W}=\left\{y:U(y)\geq\left(\frac{\varepsilon}{t}\right)^{\frac{\gamma}{4}}\right\}.\label{approximate-vorticity-lower-bound-region}
\end{align}
We next define $\eta_{1}(r)$ to be the smooth cutoff function that is identically $1$ for $r\leq3$ and identically $0$ for $r\geq4$. We let the exterior cutoff, $\eta_{\mathfrak{e}}$, be defined by
\begin{align}
    \eta_{\mathfrak{e}}\left(y,t\right)=\twopartdef{\eta_{1}\left(\left(\frac{t}{\varepsilon}\right)^{\frac{\gamma}{4}}U(y)\right)}{y\in B_{1}(0)}{\eta_{0}\left(\frac{\left|y\right|}{4}\right)}{y\in B_{1}(0)^{\mathsf{c}}}.\label{weighted-L2-exterior-cutoff}
\end{align}
We note the two functions that are glued to construct $\eta_{\mathfrak{e}}$ match at the boundary of their domains of definition as they are both identically $1$ on a region around said boundary. We note that $\eta_{\mathfrak{e}}$ is supported on the region
 \begin{align}
U(y)\leq 4\left(\frac{\varepsilon}{t}\right)^{\frac{\gamma}{4}}.\label{weighted-L2-exterior-cutoff-support}
 \end{align}
It is clear that
 \begin{align}
1\leq\eta_{\mathfrak{i}}+\eta_{\mathfrak{e}}&\leq2\ \ 
\mathrm{on}\ B_{4}\left(0\right),\label{interior-exterior-cutoff-sum-bounds}\\ 
\nonumber
\left|\grad_{y}\eta_{\mathfrak{i}}\right|+\left|\grad_{y}\eta_{\mathfrak{e}}\right|&\leq C\left(\frac{t}{\varepsilon}\right)^{\frac{1}{4}},
 \end{align}
 some absolute constant $C>0$. Moreover, using the definitions \eqref{weighted-L2-interior-cutoff} and \eqref{weighted-L2-exterior-cutoff}, as well as the bounds \eqref{xi0-bound}--\eqref{tildexi-bound}, we have that $\supp{\dell_{t}\eta_{\mathfrak{i}}}\cup\supp{\nabla_{y}\eta_{\mathfrak{i}}}\subset\supp{\eta_{\mathfrak{e}}}$, and that on the supports of $\dell_{t}\eta_{\mathfrak{i}}$ and $\nabla_{y}\eta_{\mathfrak{i}}$,
 \begin{align}
    \eta_{\mathfrak{i}}\leq\eta_{\mathfrak{e}},\quad \left|\dell_{t}\eta_{\mathfrak{i}}\right|\leq \frac{C\eta_{\mathfrak{e}}}{t},\quad \left|\nabla_{y}\eta_{\mathfrak{i}}\right|\leq C\eta_{\mathfrak{e}}.\label{derivative-of-inner-cutoff-outer-cutoff-comparison-bound-1}
 \end{align} 
 With the definitions \eqref{weighted-L2-interior-cutoff} and \eqref{weighted-L2-exterior-cutoff} in hand we define, for $1\leq r\leq\infty$, the space $\mathscr{Y}_{r}$ as all functions on $\mathbb{R}^{2}$ with
 \begin{align}
     \left\|\Phi\right\|_{\mathscr{Y}_{r}}\coloneqq \|\left(\gamma\Gamma^{\gamma-1}_{+}\right)^{\frac{1}{r}-1}\eta_{\mathfrak{i}}\Phi\|_{L^{r}\left(\mathbb{R}^{2}\right)}+\left(\frac{t}{\varepsilon}\right)^{\left(\frac{\gamma-1}{4}\right)\left(1-\frac{1}{r}\right)}\left\|\eta_{\mathfrak{e}}\Phi\right\|_{L^{r}\left(\mathbb{R}^{2}\right)}<\infty.\label{weighted-Lr-norm-definition}
 \end{align}
 We have that if $\supp\Phi\subset B_{4}\left(0\right)$, then by \eqref{interior-exterior-cutoff-sum-bounds} we have
 \begin{align}
     \left\|\Phi\right\|_{L^{r}\left(\mathbb{R}^{2}\right)}\leq C\left\|\Phi\right\|_{\mathscr{Y}_{r}}.\label{weighted-Lr-norm-bounds-Lr-norm}
 \end{align}
Now we record some inequalities related to the Poisson equation $-\Delta_{y}\psi_{*j}=\phi_{*j}$ that we will need later on.
\subsection{The Poisson Equation}\label{poisson-equation-section}
We would like to prove effective bounds on $\bar{\psi}$ on $\mathbb{R}^{2}$ to $\bar{\phi}$ by $-\Delta_{y}\bar{\psi}=\bar{\phi}$, so that
\begin{align}
\bar{\psi}\left(y\right)=-\frac{1}{2\pi}\int_{\mathbb{R}^{2}}\log{\left|z-y\right|}\ \bar{\phi}\left(z\right) dz,
\label{bar-psi-green-function-representation-upper-half-plane}
\end{align}
with the assumptions
\begin{align}
    \supp\bar{\phi}\subset B_{2}\left(0\right),\ \ \ \left\|\bar{\phi}\right\|_{L^{2}\left(\mathbb{R}^{2}\right)}<\infty,\ \ \ \int_{\mathbb{R}^{2}}\bar{\phi}\left(y\right) dy=0,\label{bar-phi-support-weighted-L2-space-and-mass-condition}
\end{align}
Note that \eqref{bar-phi-support-weighted-L2-space-and-mass-condition} implies for $1\leq r\leq\infty$
\begin{align}\nonumber   \left\|\bar{\phi}\right\|_{L^{r}\left(\mathbb{R}^{2}\right)}\leq C\left\|\bar{\phi}\right\|_{\mathscr{Y}_{r}}.
\end{align}
To help state the bounds we will need, we recall the definition of the H\"older seminorm of order $\beta$ for some $\beta\in\left(0,1\right)$:
\begin{align}
\nonumber
\left[F\right]_{\beta}\left(y\right)\coloneqq\sup_{y_{1},y_{2}\in B_{1}\left(y\right)}\frac{\left|F\left(y_{1}\right)-F\left(y_{2}\right)\right|}{\left|y_{1}-y_{2}\right|^{\beta}}.
\end{align}
Then the following result holds.
\begin{lemma}\label{poisson-equation-zero-mass-estimates-lemma}
Let $\bar{\phi}$ satisfy \eqref{bar-phi-support-weighted-L2-space-and-mass-condition}, and let $\bar{\psi}$ depend on $\bar{\phi}$ by \eqref{bar-psi-green-function-representation-upper-half-plane}. Suppose further that $\bar{\phi}\in L^{r_{0}}\left(\mathbb{R}^{2}\right)$, some $r_{0}>2$. Then for all $y\in\mathbb{R}^{2}$ and $\infty> r_{1}>2$, and some $\beta\in\left(0,1\right)$ we have
\begin{align}
\left|\bar{\psi}\left(y\right)\right|&\leq C\left\|\bar{\phi}\right\|_{L^{2}\left(\mathbb{R}^{2}\right)},\label{poisson-equation-zero-mass-estimates-lemma-statement-1}
\\    \left\|\nabla\bar{\psi}\left(y\right)\right\|_{L^{r_{1}}\left(\mathbb{R}^{2}\right)}&\leq C\left\|\bar{\phi}\right\|_{L^{2}\left(\mathbb{R}^{2}\right)},\label{poisson-equation-zero-mass-estimates-lemma-statement-2}
\\
\left|\nabla\bar{\psi}\left(y\right)\right|+\left[\nabla\bar{\psi}\right]_{\beta}\left(y\right)&\leq C\left\|\bar{\phi}\right\|_{L^{r_{0}}\left(\mathbb{R}^{2}\right)},
\label{poisson-equation-zero-mass-estimates-lemma-statement-3}
\end{align}
where $C>0$ is an absolute constant.
\end{lemma}
The proof for the inequalities in Lemma \ref{poisson-equation-zero-mass-estimates-lemma} are standard elliptic estimates and we omit them. If estimates \eqref{poisson-equation-zero-mass-estimates-lemma-statement-1}, \eqref{poisson-equation-zero-mass-estimates-lemma-statement-2}, and \eqref{poisson-equation-zero-mass-estimates-lemma-statement-3} hold, and moreover $\bar{\phi}\in L^{\infty}\left(\mathbb{R}^{2}\right)$, we have the following interpolation estimate.
\begin{lemma}\label{phi-interpolation-estimate-lemma}Suppose $\bar{\phi}$ satisfies \eqref{bar-phi-support-weighted-L2-space-and-mass-condition}, and in addition, $\bar{\phi}\in L^{\infty}\left(\mathbb{R}^{2}\right)$. Then given $\sigma\in\left(0,1\right)$, there exist numbers $C_{\sigma}>0$ and $\beta\in\left(0,1\right)$ such that
\begin{align}
    \left|\bar{\psi}\left(y\right)\right|+\left|\nabla\bar{\psi}\left(y\right)\right|+\left[\nabla\bar{\psi}\right]_{\beta}\left(y\right)\leq C_{\sigma}\left\|\bar{\phi}\right\|_{L^{\infty}\left(\mathbb{R}^{2}\right)}^{\sigma}\left\|\bar{\phi}\right\|_{L^{2}\left(\mathbb{R}^{2}\right)}^{1-\sigma}.\label{poisson-equation-zero-mass-interpolation-estimate-lemma-statement}
\end{align}
\end{lemma}
\subsection{Lower Bound on a Quadratic Form}
We now consider functions $\bar{\phi}$ on the upper half plane $\mathbb{R}^{2}$ once again satisfying the assumptions 
\begin{align}
    \supp\bar{\phi}\subset B_{2}\left(0\right),\ \ \ \left\|\bar{\phi}\right\|_{L^{2}\left(\mathbb{R}^{2}\right)}<\infty,\ \ \ \int_{\mathbb{R}^{2}}\bar{\phi}\left(y\right) dy=0,\label{bar-phi-support-weighted-L2-space-and-mass-condition-2}
\end{align}
as well as the centre of mass conditions
\begin{align}
    \int_{\mathbb{R}^{2}}y_{1}\bar{\phi}\left(y\right) dy=0,\ \ \ \int_{\mathbb{R}^{2}}y_{2}\bar{\phi}\left(y\right) dy=0.\label{bar-phi-centre-of-mass-conditions}
\end{align}
If $\bar{\phi}$ satisfies the above conditions, we gain the following crucial estimate. Recall the definitions of $\eta_{\mathfrak{i}}$, $\eta_{\mathfrak{e}}$, and $\mathfrak{W}$ from \eqref{weighted-L2-interior-cutoff}--\eqref{weighted-L2-exterior-cutoff}. 
\begin{lemma}\label{quadratic-form-estimate-lemma}
    Suppose $\bar{\phi}$ satisfies \eqref{bar-phi-support-weighted-L2-space-and-mass-condition-2}--\eqref{bar-phi-centre-of-mass-conditions}, and $\xi_{j}$ satisfies \eqref{xi0-bound}--\eqref{tildexi-bound}. Then for all $t\in\left[T_{0},T\right]$ there exists an absolute constant $C_{0}>0$ such that
    \begin{align}
        \int_{\mathbb{R}^{2}}\frac{\eta_{\mathfrak{i}}\bar{\phi}}{\gamma\Gamma^{\gamma-1}_{+}}\left(\eta_{\mathfrak{i}}\bar{\phi}-\gamma\Gamma^{\gamma-1}_{+}\left(-\Delta_{y}\right)^{-1}\left(\eta_{\mathfrak{i}}\bar{\phi}\right)\right)\geq C_{0}\int_{\mathbb{R}^{2}}\frac{\eta_{\mathfrak{i}}^{2}\left|\bar{\phi}\right|^{2}}{\gamma\Gamma^{\gamma-1}_{+}}+\mathcal{R}_{1},\label{quadratic-form-estimate-lemma-statement-1}
    \end{align}
    with
    \begin{align}
        \mathcal{R}_{1}=\bigO\left(\left\|\eta_{\mathfrak{e}}\bar{\phi}\right\|_{L^{2}\left(\mathbb{R}^{2}\right)}^{2}+\left\|\eta_{\mathfrak{e}}\bar{\phi}\right\|_{L^{2}\left(\mathbb{R}^{2}\right)}\left\|\eta_{\mathfrak{i}}\bar{\phi}\right\|_{L^{2}\left(\mathbb{R}^{2}\right)}\right).\label{quadratic-form-estimate-lemma-statement-2}
    \end{align}
\end{lemma}
\begin{proof}
    Due to the fact that $\gamma\geq19$, Lemma \ref{poisson-equation-zero-mass-estimates-lemma}, \eqref{weighted-L2-interior-cutoff-support}, and \eqref{bar-phi-support-weighted-L2-space-and-mass-condition-2}, we have that the integrals on both sides of the inequality in \eqref{quadratic-form-estimate-lemma-statement-1} are well defined.

    \medskip
    In the weighted space $L^{2}\left(\mathbb{R}^{2},\gamma\Gamma^{\gamma-1}_{+}dy\right)$, we can expand in the orthonormal eigenbasis $\left(e_{j}\right)_{j\geq0}$ defined by 
    \begin{align*}
        -\Delta_{y}e_{j}=\gamma\Gamma^{\gamma-1}_{+}\mu_{j}e_{j}.
    \end{align*}
    Note that using Lemmas \ref{dancer-yan-non-degeneracy-lemma} and \ref{vortex-linearised-equation-fourier-coefficients-behaviour-lemma} gives us that $\mu_{0}=0$, $e_{0}$ is a constant, $\mu_{1}=\mu_{2}=1$ with $e_{1}$ and $e_{2}$ given by scalar multiples of $\dell_{1}\Gamma$ and $\dell_{2}\Gamma$ respectively, and $0<\mu_{k}^{-1}\leq \delta<1 $, for $k\geq 3$. So let
    \begin{align}
        \frac{\eta_{\mathfrak{i}}\bar{\phi}}{\gamma\Gamma^{\gamma-1}_{+}}=\sum_{j=0}^{\infty}\bar{\phi}_{j}e_{j}.
    \end{align}
    \begin{align}
        &\int_{\mathbb{R}^{2}}\frac{\eta_{\mathfrak{i}}\bar{\phi}}{\gamma\Gamma^{\gamma-1}_{+}}\left(\eta_{\mathfrak{i}}\bar{\phi}-\gamma\Gamma^{\gamma-1}_{+}\left(-\Delta_{y}\right)^{-1}\left(\eta_{\mathfrak{i}}\bar{\phi}\right)\right)=\int_{\mathbb{R}^{2}}\frac{\eta_{\mathfrak{i}}\bar{\phi}}{\gamma\Gamma^{\gamma-1}_{+}}\left(\eta_{\mathfrak{i}}\bar{\phi}-\gamma\Gamma^{\gamma-1}_{+}\left(-\Delta_{y}\right)^{-1}\left(\gamma\Gamma^{\gamma-1}_{+}\frac{\eta_{\mathfrak{i}}\bar{\phi}}{\gamma\Gamma^{\gamma-1}_{+}}\right)\right)\nonumber\\\nonumber
        &=\left(\frac{\eta_{\mathfrak{i}}\bar{\phi}}{\gamma\Gamma^{\gamma-1}_{+}},\frac{\eta_{\mathfrak{i}}\bar{\phi}}{\gamma\Gamma^{\gamma-1}_{+}}\right)-\left(\frac{\eta_{\mathfrak{i}}\bar{\phi}}{\gamma\Gamma^{\gamma-1}_{+}},\left(-\Delta_{y}\right)^{-1}\left(\gamma\Gamma^{\gamma-1}_{+}\frac{\eta_{\mathfrak{i}}\bar{\phi}}{\gamma\Gamma^{\gamma-1}_{+}}\right)\right),
    \end{align}
    where $\left(\blank, \blank\right)$ is the inner product associated to the space $L^{2}\left(\mathbb{R}^{2},\gamma\Gamma^{\gamma-1}_{+}dy\right)$. Thus we have
    \begin{align}
        &\int_{\mathbb{R}^{2}}\frac{\eta_{\mathfrak{i}}\bar{\phi}}{\gamma\Gamma^{\gamma-1}_{+}}\left(\eta_{\mathfrak{i}}\bar{\phi}-\gamma\Gamma^{\gamma-1}_{+}\left(-\Delta_{y}\right)^{-1}\left(\eta_{\mathfrak{i}}\bar{\phi}\right)\right)=\sum_{j=3}^{\infty}\left(1-\mu_{j}^{-1}\right)\left|\bar{\phi}_{j}\right|^{2}\nonumber\\
        &+\left|\bar{\phi}_{0}e_{0}\right|^{2}\int_{\mathbb{R}^{2}}\gamma\Gamma^{\gamma-1}_{+}\left(-\Delta_{y}\right)^{-1}\left(\gamma\Gamma^{\gamma-1}_{+}\right)+\bar{\phi}_{0}e_{0}\int_{\mathbb{R}^{2}}\gamma\Gamma^{\gamma-1}_{+}\left(-\Delta_{y}\right)^{-1}\left(\eta_{\mathfrak{i}}\bar{\phi}\right)\nonumber\\
        &+\bar{\phi}_{0}e_{0}\int_{\mathbb{R}^{2}}\eta_{\mathfrak{i}}\bar{\phi}\left(-\Delta_{y}\right)^{-1}\left(\gamma\Gamma^{\gamma-1}_{+}\right),\label{quadratic-form-estimate-proof-5}
    \end{align}
    where we have used the fact that $\mu_{1}=\mu_{2}=1$. As stated above, Lemma \ref{dancer-yan-non-degeneracy-lemma} gives us
    \begin{align}
        \sum_{j=3}^{\infty}\left(1-\mu_{j}^{-1}\right)\left|\bar{\phi}_{j}\right|^{2}\geq \left(1-\delta\right)\sum_{j=3}^{\infty}\left|\bar{\phi}_{j}\right|^{2},\label{quadratic-form-estimate-proof-6}
    \end{align}
    some absolute constant $\delta\in\left(0,1\right)$. For the last three terms on the right hand side of \eqref{quadratic-form-estimate-proof-5}, we first note that the support and zero mass assumption in \eqref{bar-phi-support-weighted-L2-space-and-mass-condition-2} as well as \eqref{interior-exterior-cutoff-sum-bounds} gives us
    \begin{align}
        \bar{\phi}_{0}&=\int_{\mathbb{R}^{2}}\gamma\Gamma^{\gamma-1}_{+}\frac{\eta_{\mathfrak{i}}\bar{\phi}}{\gamma\Gamma^{\gamma-1}_{+}}e_{0}=e_{0}\left(\int_{\mathbb{R}^{2}}\bar{\phi}+\int_{\mathbb{R}^{2}}\left(1-\mathfrak{\eta}_{i}\right)\bar{\phi}\right)=e_{0}\int_{\mathbb{R}^{2}}\left(1-\mathfrak{\eta}_{i}\right)\bar{\phi}\nonumber\\
        &=\bigO{\left(\left\|\eta_{\mathfrak{e}}\bar{\phi}\right\|_{L^{2}\left(\mathbb{R}^{2}\right)}\right)}.\label{quadratic-form-estimate-proof-7}
    \end{align}
    Next, using the fact that $\supp \eta_{\mathfrak{i}}\cup \supp \gamma\Gamma^{\gamma-1}_{+}\subset B_{1}(0)$, one can directly calculate that
    \begin{align}
    \left|\int_{\mathbb{R}^{2}}\gamma\Gamma^{\gamma-1}_{+}\left(-\Delta_{y}\right)^{-1}\left(\eta_{\mathfrak{i}}\bar{\phi}\right)+\int_{\mathbb{R}^{2}}\eta_{\mathfrak{i}}\bar{\phi}\left(-\Delta_{y}\right)^{-1}\left(\gamma\Gamma^{\gamma-1}_{+}\right)\right|\leq C\left\|\eta_{\mathfrak{i}}\bar{\phi}\right\|_{L^{2}\left(\mathbb{R}^{2}\right)}.\label{quadratic-form-estimate-proof-7a}
    \end{align}
    Next, we use the centre of mass condition with respect to $y_{1}$ given in \eqref{bar-phi-centre-of-mass-conditions} to get
    \begin{align}
        0&=\int_{\mathbb{R}^{2}}y_{1}\bar{\phi}\ dy=\int_{\mathbb{R}^{2}}\gamma\Gamma^{\gamma-1}_{+}y_{1}\frac{\eta_{\mathfrak{i}}\bar{\phi}}{\gamma\Gamma^{\gamma-1}_{+}}\ dy+\int_{\mathbb{R}^{2}}\left(1-\eta_{\mathfrak{i}}\right)y_{1}\bar{\phi}\ dy\nonumber\\
        &=\bar{\phi}_{1}\int_{\mathbb{R}^{2}}\gamma\Gamma^{\gamma-1}_{+}y_{1}e_{1}\ dy+\int_{\mathbb{R}^{2}}\gamma\Gamma^{\gamma-1}_{+}y_{1}\left(\sum_{j=3}^{\infty}\bar{\phi}_{j}e_{j}\right)\ dy +\bigO\left(\left\|\eta_{\mathfrak{e}}\bar{\phi}\right\|_{L^{2}\left(\mathfrak{W}^{\mathsf{c}}\right)}\right),
        \nonumber
    \end{align}
    where we have again used \eqref{interior-exterior-cutoff-sum-bounds}, and the fact that $\gamma\Gamma^{\gamma-1}_{+}$, $e_{0}$ and $e_{2}$ are even in $y_{1}$. Since $e_{1}$ is, up to some positive normalization constant $c$, equal to $\dell_{1}\Gamma$, we have that
    \begin{align*}
        \int_{\mathbb{R}^{2}}\gamma\Gamma^{\gamma-1}_{+}y_{1}e_{1}\ dy=c\int_{\mathbb{R}^{2}}\left(\gamma\Gamma^{\gamma-1}_{+}\dell_{1}\Gamma\right) y_{1}\ dy=-c\int_{\mathbb{R}^{2}}U(y)dy= -cM<0
    \end{align*}
    where $M$ is the absolute constant defined in \eqref{vortex-mass-definition}. Thus, we have
    \begin{align}
        c_{1}\left(\left|\bar{\phi}_{1}\right|^{2}+\left|\bar{\phi}_{2}\right|^{2}\right)+\bigO\left(\left\|\eta_{\mathfrak{e}}\bar{\phi}\right\|_{L^{2}\left(\mathbb{R}^{2}\right)}^{2}\right)\leq \sum_{j=3}^{\infty}\left|\bar{\phi}_{j}\right|^{2},\label{quadratic-form-estimate-proof-14}
    \end{align}
    as the argument for $\bar{\phi}_{2}$ using the centre of mass condition with respect to $y_{2}$ in \eqref{bar-phi-centre-of-mass-conditions} is completely analogous.

    \medskip
    Using \eqref{quadratic-form-estimate-proof-6}--\eqref{quadratic-form-estimate-proof-14}, we obtain for \eqref{quadratic-form-estimate-proof-5}
    \begin{align*} 
        &\sum_{j=3}^{\infty}\left(1-\mu_{j}\right)\left|\bar{\phi}_{j}\right|^{2}\geq\frac{\min{\{c_{1},1\}}}{2}\left(1-\delta\right)\sum_{j=0}^{\infty}\left|\bar{\phi}_{j}\right|^{2}+\mathcal{R}_{1},
        \nonumber
    \end{align*}
    where $\mathcal{R}_{1}$ satisfies \eqref{quadratic-form-estimate-lemma-statement-2}. This gives \eqref{quadratic-form-estimate-lemma-statement-1}, as required.
\end{proof}
\subsection{Weighted $L^{2}$ A Priori Estimates}
As in \cite{DDMP2VP2023,DDMW2020}, we first wish to study solutions to a suitable linear-in-$\phi_{*j}$ problem. We will derive a priori estimates for the system
\begin{subequations}\label{linear-transport-operator-system}
\begin{align}
&\varepsilon^{2}\phi_{t}+\kappa_{j}\gradperp_{y}\left(\Gamma+a_{*}+a\right)\cdot\nabla_{y}\left(\phi-\gamma\left(\Gamma+a_{*}\right)^{\gamma-1}_{+}\psi\right)+E=0\ \ \text{in}\ \mathbb{R}^{2}\times \left[T_{0},T\right],\label{linear-transport-equation}\\
&\phi(\blank,T)=0\ \ \text{in}\ \mathbb{R}^{2},\label{linear-transport-initial-data}\\
&\phi\to0,\ \left|y\right|\to\infty,\ \forall t\in\left[T_{0},T\right].\label{linear-transport-boundary-condition}
\end{align}
\end{subequations}
Here $a\left(y,t\right)$, and $a_{*}\left(y,t\right)$ are some given functions. Next, $E\left(y,t\right)$ is some error, and $\psi$ is the inverse Laplacian of $\phi$ on $\mathbb{R}^{2}$ as in \eqref{psi-star-r-green-function-representation-upper-half-plane}.

\medskip
We impose that for all $t\in\left[T_{0},T\right]$, 
\begin{align}
    \int_{\mathbb{R}^{2}}\phi\left(y,t\right)\ dy=0,\ \ \ \int_{\mathbb{R}^{2}}y_{1}\phi\left(y,t\right)\ dy=0,\ \ \ \int_{\mathbb{R}^{2}}y_{2}\phi\left(y,t\right)\ dy=0.\label{linear-transport-operator-phi-mass-conditions}
\end{align}
On the functions $a_{*}$ and $a$, we assume
\begin{align}
    \Delta_{y}\left(a_{*}+a\right)\in L^{\infty}\left(\mathbb{R}^{2}\times\left[T_{0},T\right]\right),\label{linear-transport-operator-laplacian-a-star-a-bounded}
\end{align}
We note that this implies $\nabla_{y}a$ has a uniform Log-Lipschitz bound with respect to its space variable. That is, for all $t\in\left[T_{0},T\right]$, and $y, \hat{y}\in\mathbb{R}^{2}$,
\begin{align}
    \left|\nabla_{y}a\left(y,t\right)-\nabla_{y}a\left(\hat{y},t\right)\right|\leq C\left|y-\hat{y}\right|\left|\log{\left|y-\hat{y}\right|}\right|,\label{linear-transport-operator-a-log-lipschitz-bound}
\end{align}
for an absolute constant $C>0$.

\medskip
We also assume that for some constants $C>0$ and $\nu>3/4$, we have, for all $t\in\left[T_{0},T\right]$, and some $r_{1}>2$,
\begin{align}
    \left|a_{*}\right|+\left|\nabla_{y} a_{*}\right|+\left|D^{2}_{y}a_{*}\right|+t\left|\dell_{t}a_{*}\right|&\leq \frac{C\varepsilon^{2}}{t},\label{linear-transport-operator-a-star-bounds}\\
    \left|a\right|+\left\|\nabla_{y}a\right\|_{L^{r_{1}}\left(\mathbb{R}^{2}\right)}&\leq \frac{\varepsilon^{2+\nu}}{t^{\frac{3}{2}}},\label{linear-transport-operator-a-bounds-0}\\
    \left|\nabla_{y}a\right|&\leq \frac{\varepsilon^{2+\nu}}{t^{\frac{1}{2}+\nu}}.\label{linear-transport-operator-a-bounds}
\end{align}
We make an assumption of compact support on $E$, and only study solutions to \eqref{linear-transport-operator-system} that enjoy a similar property:
\begin{align}
    \supp E\left(t\right)\subset B_{\frac{3}{2}}\left(0\right), \quad \supp \phi \left(t\right)\subset B_{2}\left(0\right) \quad t\in\left[T_{0},T\right].\label{linear-transport-operator-error-solution-supports}
\end{align}
With these assumptions in mind, we now prove a priori estimates for solutions to \eqref{linear-transport-operator-system} satisfying \eqref{linear-transport-operator-phi-mass-conditions}--\eqref{linear-transport-operator-error-solution-supports} over the next two lemmas. The first is an ``exterior" estimate on a region where the vorticity at main order is small in a suitable sense.
\begin{lemma}\label{L2-exterior-a-priori-estimate-lemma}
     Let $E\left(y,t\right)\in\mathscr{Y}_{\infty}$, defined in \eqref{weighted-Lr-norm-definition}. Suppose $\phi$ solves \eqref{linear-transport-operator-system} with assumptions \eqref{linear-transport-operator-phi-mass-conditions}--\eqref{linear-transport-operator-error-solution-supports} holding, and suppose $\xi=\left(\xi_{1},\xi_{2},\xi_{3}\right)$ is of the form \eqref{point-vortex-trajectory-function-decomposition}, with \eqref{xi0-bound}--\eqref{tildexi-bound} satisfied. Then for all $\varepsilon>0$ small enough, and all $t\in\left[T_{0},T\right]$,
    \begin{align}
        \left\|\phi\right\|_{L^{2}\left(\mathfrak{W}^{\mathsf{c}}\right)}&\leq \frac{C\varepsilon^{\frac{\gamma}{4}-\frac{3}{2}}}{t^{\frac{\gamma}{4}-\frac{1}{2}}}\int_{t}^{T}\left(\left\|E\right\|_{\mathscr{Y}_{\infty}}+\left\|\phi\right\|_{L^{2}\left(\mathbb{R}^{2}\right)}\right)d\tau,\label{L2-exterior-a-priori-estimate-lemma-statement-1}\\
        \left\|\eta_{\mathfrak{e}}\phi\right\|_{L^{2}\left(\mathbb{R}^{2}\right)}&\leq \frac{C\varepsilon^{\frac{\gamma}{4}-\frac{3}{2}}}{t^{\frac{\gamma}{4}-\frac{1}{2}}}\int_{t}^{T}\left(\left\|E\right\|_{\mathscr{Y}_{\infty}}+\|\left(\gamma\Gamma^{\gamma-1}_{+}\right)^{-\frac{1}{2}}\eta_{\mathfrak{i}}\phi\|_{L^{2}\left(\mathbb{R}^{2}\right)}\right)d\tau.\label{L2-exterior-a-priori-estimate-lemma-statement-2}
    \end{align}
\end{lemma}
\begin{proof}
    The proof is very similar to an analogous exterior estimate proved in \cite{DDMP2VP2023}, and uses the transport-like structure of the linear problem. Let $t\in\left[T_{0},T\right]$, and define characteristics $\bar{y}\left(\tau,t,y\right)$ by the ODE problem
    \begin{subequations}\label{characteristics-ode-system}
    \begin{align}
        &\frac{d\bar{y}}{d\tau}=\varepsilon^{-2}\gradperp_{\bar{y}}\left(\Gamma+a_{*}+a\right)\left(\bar{y}\left(\tau,t,y\right),\tau\right),\ \ \ \tau\in\left[t,T\right],\label{characteristics-ode}\\
        &\bar{y}\left(t,t,z\right)=y.
        \label{characteristics-initial-data}
   \end{align}
    \end{subequations}
    The first thing to note is that due to \eqref{linear-transport-operator-a-log-lipschitz-bound}--\eqref{linear-transport-operator-a-star-bounds}, we have definiteness of the characteristics. Next, as the right hand side of \eqref{characteristics-ode} is divergence free, if we define the Jacobian of the transformation between $y$ and $\bar{y}$ as $\mathscr{J}\left(y,\bar{y}\right)$, using Jacobi's formula we obtain,
    \begin{align}
        \mathscr{J}\left(y,\bar{y}\right)=1\label{L2-exterior-a-priori-estimate-lemma-12}
    \end{align}
    for all $\tau\in\left[t,T\right]$.

    \medskip
    We also have a good estimate on $\left(\Gamma+a_{*}\right)\left(\bar{y}\left(\tau\right),\tau\right)$ as $\tau$ varies. Let $t<\tau_{1}$ and note
    \begin{align}
        \nonumber\left(\Gamma+a_{*}\right)\left(\bar{y}\left(\tau_{1}\right),\tau_{1}\right)-\left(\Gamma+a_{*}\right)\left(y,t\right)=\int_{t}^{\tau_{1}}\dell_{\tau}\left(\Gamma+a_{*}\right) d\tau.
    \end{align}
    Then, because of \eqref{linear-transport-operator-a-star-bounds}--\eqref{linear-transport-operator-a-bounds}, and the fact that $\bar{y}$ satisfies \eqref{characteristics-ode-system}, we have for all $\tau\in\left[t,\tau_{1}\right]$,
    \begin{align}
       \nonumber \left|\dell_{\tau}\left(\Gamma+a_{*}\right)\left(\bar{y}\left(\tau\right),\tau\right)\right|\leq \frac{C\varepsilon^{\nu}}{t^{\nu+\frac{1}{2}}},
    \end{align}
    some absolute constants $C>0$ and $\nu>3/4$. Thus, for all $\tau_{1}\in\left[t,T\right]$, we have
    \begin{align}
        \left|\left(\Gamma+a_{*}\right)\left(\bar{y}\left(\tau_{1}\right),\tau_{1}\right)-\left(\Gamma+a_{*}\right)\left(y,t\right)\right|\leq C\left(\frac{\varepsilon}{t}\right)^{\frac{1}{4}}.\label{stream-function-level-sets-characteristics-estimate-4}
    \end{align}
    Using \eqref{weighted-Lr-norm-bounds-Lr-norm}, we have, for all $\tau\in\left[t,T\right]$,
    \begin{align}
\nonumber\left|E\left(\bar{z}\left(\tau\right),\tau\right)\right|\leq C\left(\gamma\Gamma^{\gamma-1}_{+}\left(\bar{y}\left(\tau\right)\right)+\left(\frac{\varepsilon}{t}\right)^{\left(\frac{\gamma-1}{4}\right)}\right)\left\|E\right\|_{\mathscr{Y}_{\infty}}.
    \end{align}
    We note that since $y\in\mathfrak{W}^{\mathsf{c}}$, $\left|\gamma\Gamma^{\gamma-1}_{+}(y)\right|\leq \left(\varepsilon t^{-1}\right)^{\frac{1}{4}}$ for all $\tau\in\left[t,T\right]$. Then, this alongside \eqref{linear-transport-operator-a-star-bounds}, and \eqref{stream-function-level-sets-characteristics-estimate-4} give, for all $\tau\in\left[t,T\right]$,
    \begin{align}
        \left|\gamma\Gamma^{\gamma-1}_{+}\left(\bar{y}\left(\tau\right)\right)\right|\leq C\left(\frac{\varepsilon}{t}\right)^{\left(\frac{\gamma-1}{4}\right)},\quad \left|\gamma\left(\Gamma+a_{*}\right)^{\gamma-1}_{+}\left(\bar{y}\left(\tau\right).\tau\right)\right|\leq C\left(\frac{\varepsilon}{t}\right)^{\left(\frac{\gamma-1}{4}\right)}\label{L2-exterior-a-priori-estimate-lemma-4}
    \end{align}
    which in turn gives, for $z\in\mathfrak{W}^{\mathsf{c}}$ and all $\tau\in\left[t,T\right]$,
    \begin{align}
        \left|E\left(\bar{y}\left(\tau,t,y\right),\tau\right)\right|\leq C\left(\frac{\varepsilon}{t}\right)^{\left(\frac{\gamma-1}{4}\right)}\left\|E\right\|_{\mathscr{Y}_{\infty}}\left(\tau\right).
        \label{L2-exterior-a-priori-estimate-lemma-5}
    \end{align}
    Then, using the characteristics to defined in \eqref{characteristics-ode-system} to write a solution to \eqref{linear-transport-operator-system}, $\phi$, as
    \begin{align}
        \phi\left(y,t\right)=\varepsilon^{-2}\int_{t}^{T}\left(-E\left(\bar{y}\right)+\kappa_{j}\gradperp_{\bar{y}}\left(\Gamma+a_{*}+a\right)\cdot\nabla_{\bar{y}}\left(\gamma\left(\Gamma+a_{*}\right)^{\gamma-1}_{+}\psi\left(\bar{y}\right)\right)\right) d\tau,\label{L2-exterior-a-priori-estimate-lemma-1}
    \end{align}
    we can use \eqref{L2-exterior-a-priori-estimate-lemma-12}--\eqref{L2-exterior-a-priori-estimate-lemma-1} to obtain Lemma \ref{L2-exterior-a-priori-estimate-lemma}.
\end{proof}
Now we move on to the a priori estimate on the interior region where $\eta_{\mathfrak{i}}$ is supported.
\begin{lemma}\label{L2-interior-a-priori-estimate-lemma}
    Let $E\left(y,t\right)\in\mathscr{Y}_{\infty}$, defined in \eqref{weighted-Lr-norm-definition}, and suppose $E$ has the bound
    \begin{align}
        \left\|E\right\|_{\mathscr{Y}_{\infty}}<\frac{\varepsilon^{5-\sigma_{0}}}{t^{\frac{5}{2}}},\label{L2-interior-a-priori-estimate-lemma-error-bound-statement}
    \end{align}
    for some small $\sigma_{0}>0$.

    \medskip
    Suppose further that $\phi$ solves \eqref{linear-transport-operator-system} with assumptions \eqref{linear-transport-operator-phi-mass-conditions}--\eqref{linear-transport-operator-error-solution-supports} holding, and suppose $\xi=\left(\xi_{1},\xi_{2},\xi_{3}\right)$ is of the form \eqref{point-vortex-trajectory-function-decomposition}, with \eqref{xi0-bound}--\eqref{tildexi-bound} satisfied. Then for all $\varepsilon>0$ small enough, and all $T_{0}>0$ large enough, we have for all $t\in\left[T_{0},T\right]$,
    \begin{align}
        \|\left(\gamma\Gamma^{\gamma-1}_{+}\right)^{-\frac{1}{2}}\eta_{\mathfrak{i}}\phi\|_{L^{2}\left(\mathbb{R}^{2}\right)}\leq\frac{C\varepsilon^{3-\sigma_{0}}\left|\log{\varepsilon}\right|^{\frac{3}{2}}}{t^{\frac{3}{2}}},\label{L2-interior-a-priori-estimate-lemma-statement}
    \end{align}
    for some absolute constant $C>0$.
\end{lemma}
\begin{remark}
    As mentioned in Remark \ref{choice-of-T0-remark}, we can enlarge $T_{0}$ if necessary as long as we keep it independent of $\varepsilon$.
\end{remark}
\begin{proof}
    Once again the strategy is very similar to the interior estimate proved in \cite{DDMP2VP2023}. Take the equality \eqref{linear-transport-equation}, multiply by
    \begin{align}
        \eta_{\mathfrak{i}}G\coloneqq\left(\frac{\eta_{\mathfrak{i}}^{2}\phi}{\gamma\left(\Gamma+a_{*}\right)^{\gamma-1}_{+}}-\eta_{\mathfrak{i}}\left(-\Delta_{y}\right)^{-1}\left(\eta_{\mathfrak{i}}\phi\right)\right),\label{L2-interior-a-priori-estimate-lemma-1}
    \end{align}
    and integrate over $\mathbb{R}^{2}$ to obtain
    \begin{align}
        &\varepsilon^{2}\int_{B_{2}\left(0\right)}\frac{\eta_{\mathfrak{i}}^{2}\phi\ \dell_{t}\phi\ dy}{\gamma\left(\Gamma+a_{*}\right)^{\gamma-1}_{+}}-\varepsilon^{2}\int_{B_{2}\left(0\right)}\eta_{\mathfrak{i}}\dell_{t}\phi\left(-\Delta_{y}\right)^{-1}\left(\eta_{\mathfrak{i}}\phi\right)dy\label{L2-interior-a-priori-estimate-lemma-2}\\
        &+\kappa_{j}\int_{B_{2}\left(0\right)}\eta_{\mathfrak{i}}\left(y,t\right)G\left(y,t\right)\ \gradperp_{y}\left(\Gamma+a_{*}+a\right)\cdot\nabla_{y}\left(\phi-\gamma\left(\Gamma+a_{*}\right)^{\gamma-1}_{+}\psi\right)dy\nonumber\\
        &+\int_{B_{2}\left(0\right)}\left(\frac{\eta_{\mathfrak{i}}^{2}\phi}{\gamma\left(\Gamma+a_{*}\right)^{\gamma-1}_{+}}-\eta_{\mathfrak{i}}\left(-\Delta_{y}\right)^{-1}\left(\eta_{\mathfrak{i}}\phi\right)\right)E\ dy=0,\nonumber
    \end{align}
    where the region of integration is due to the support of $\eta_{\mathfrak{i}}$ given in \eqref{weighted-L2-interior-cutoff-support}. In the remainder of this proof, we ignore the factor of $\kappa_{j}$ on the second integral above, as it does not play an important role in the estimates.

    \medskip
    For the first term on the left hand side of \eqref{L2-interior-a-priori-estimate-lemma-2}, we have
    \begin{align}
        \varepsilon^{2}\int_{B_{2}\left(0\right)}\frac{\eta_{\mathfrak{i}}^{2}\phi\ \dell_{t}\phi\ dy}{\gamma\left(\Gamma+a_{*}\right)^{\gamma-1}_{+}}&=\frac{\varepsilon^{2}}{2}\frac{d}{dt}\left(\int_{B_{2}\left(0\right)}\frac{\eta_{\mathfrak{i}}^{2}\phi^{2}\ dy}{\gamma\left(\Gamma+a_{*}\right)^{\gamma-1}_{+}}\right)\nonumber\\
        &-\frac{\varepsilon^{2}}{2}\int_{B_{2}\left(0\right)}\dell_{t}\left(\frac{\eta_{\mathfrak{i}}^{2}}{\gamma\left(\Gamma+a_{*}\right)^{\gamma-1}_{+}}\right)\phi^{2}\ dy.\label{L2-interior-a-priori-estimate-lemma-3}
    \end{align}
    In order to bound the last term on the right hand side of \eqref{L2-interior-a-priori-estimate-lemma-3}, we first note that on $B_{2}\left(0\right)$, \eqref{weighted-L2-interior-cutoff}--\eqref{weighted-L2-exterior-cutoff} as well as \eqref{linear-transport-operator-a-star-bounds} give that on the support of $\eta_{\mathfrak{i}}$,
    
    \begin{align}
        \left(1+C\left(\frac{\varepsilon}{t}\right)^{\frac{3}{4}}\right)\gamma\left(\Gamma+a_{*}\right)^{\gamma-1}_{+}\geq\gamma\Gamma^{\gamma-1}_{+}\geq\left(1-C\left(\frac{\varepsilon}{t}\right)^{\frac{3}{4}}\right)\gamma\left(\Gamma+a_{*}\right)^{\gamma-1}_{+},\label{f-prime-L2-inner-weight-comparison-bound}
    \end{align}
    for some constant $C>0$. Then \eqref{f-prime-L2-inner-weight-comparison-bound}, \eqref{weighted-L2-interior-cutoff-support}, and \eqref{derivative-of-inner-cutoff-outer-cutoff-comparison-bound-1} give
    \begin{align}
        \left|\eta_{\mathfrak{i}}\dell_{t}\eta_{\mathfrak{i}}\right|\left|\frac{1}{\gamma\left(\Gamma+a_{*}\right)^{\gamma-1}_{+}}\right|\leq \frac{C\eta_{\mathfrak{e}}^{2}}{t}\left(\frac{t}{\varepsilon}\right)^{\frac{\gamma-1}{4}}.\label{L2-interior-a-priori-estimate-lemma-4}
    \end{align}
    Next, we note that due to \eqref{linear-transport-operator-a-star-bounds}, and \eqref{weighted-L2-interior-cutoff-support}, we have
    \begin{align}
        \left|\eta_{\mathfrak{i}}^{2}\dell_{t}\left(\frac{1}{\gamma\left(\Gamma+a_{*}\right)^{\gamma-1}_{+}}\right)\right|&\leq \frac{C\eta_{\mathfrak{i}}^{2}}{\gamma\left(\Gamma+a_{*}\right)^{\gamma}_{+}}\left(\frac{\varepsilon}{t}\right)^{2}\leq \frac{C\eta_{\mathfrak{i}}^{2}}{\gamma\left(\Gamma+a_{*}\right)^{\gamma-1}_{+}}\left(\frac{\varepsilon}{t}\right)^{\frac{7}{4}}.
   \nonumber \end{align}
    Thus
    \begin{align}
        \left|\int_{B_{2}\left(0\right)}\dell_{t}\left(\frac{\eta_{\mathfrak{i}}^{2}}{\gamma\left(\Gamma+a_{*}\right)^{\gamma-1}_{+}}\right)\phi^{2}\ dy\right|&\leq C\left(\frac{\varepsilon}{t}\right)^{\frac{7}{4}}\|\left(\gamma\left(\Gamma+a_{*}\right)^{\gamma-1}_{+}\right)^{-\frac{1}{2}}\eta_{\mathfrak{i}}\phi\|_{L^{2}\left(\mathbb{R}^{2}\right)}^{2}+\frac{C\alpha_{\mathfrak{e}}^{2}\left(t\right)}{t},\label{L2-interior-a-priori-estimate-lemma-6}
    \end{align}
    where $\alpha_{\mathfrak{e}}$ is defined as
    \begin{align}
        \alpha_{\mathfrak{e}}=\left(\frac{t}{\varepsilon}\right)^{\frac{\gamma-1}{8}}\left\|\eta_{\mathfrak{e}}\phi\right\|_{L^{2}\left(\mathbb{R}^{2}\right)}.\label{alpha-exterior-definition}
    \end{align}
    Next we have the second term on the left hand side of \eqref{L2-interior-a-priori-estimate-lemma-2} which we can write as
    \begin{align}
        \frac{\varepsilon^{2}}{2}\frac{d}{dt}\left(\int_{B_{2}\left(0\right)}\left(\eta_{\mathfrak{i}}\phi\right)\left(-\Delta_{y}\right)^{-1}\left(\eta_{\mathfrak{i}}\phi\right)dy\right)-\varepsilon^{2}\int_{B_{2}\left(0\right)}\left(\phi\dell_{t}\eta_{\mathfrak{i}}\right)\left(-\Delta_{y}\right)^{-1}\left(\eta_{\mathfrak{i}}\phi\right)dy.\label{L2-interior-a-priori-estimate-lemma-7}
    \end{align}
    Similarly to \eqref{L2-interior-a-priori-estimate-lemma-4}, we have
    \begin{align}
        \left|\int_{B_{2}\left(0\right)}\left(\phi\dell_{t}\eta_{\mathfrak{i}}\right)\left(-\Delta_{y}\right)^{-1}\left(\eta_{\mathfrak{i}}\phi\right)dy\right|\leq\frac{C}{t}\int_{B_{2}\left(0\right)}\eta_{\mathfrak{e}}\left|\phi\right|\left|\left(-\Delta_{y}\right)^{-1}\left(\eta_{\mathfrak{i}}\phi\right)\right| dy.\label{L2-interior-a-priori-estimate-lemma-8}
    \end{align}
    Bounding $\left(-\Delta_{y}\right)^{-1}\left(\eta_{\mathfrak{i}}\phi\right)$ in $L^{\infty}$, and then using the compact support of the integral on the right hand side of \eqref{L2-interior-a-priori-estimate-lemma-8} to deal with the $\log{\left|y\right|}$ term, as well as to bound $\eta_{\mathfrak{e}}\phi$ in $L^{2}$ gives us
    \begin{align}
        \left|\int_{B_{2}\left(0\right)}\left(\phi\dell_{t}\eta_{\mathfrak{i}}\right)\left(-\Delta_{y}\right)^{-1}\left(\eta_{\mathfrak{i}}\phi\right)dy\right|\leq\frac{C}{t}\left(\frac{\varepsilon}{t}\right)^{\left(\frac{\gamma-1}{8}\right)}\alpha_{\mathfrak{e}}\left(t\right)\left\|\phi\right\|_{L^{2}\left(\mathbb{R}^{2}\right)}.\nonumber
    \end{align}
    Then we use \eqref{weighted-Lr-norm-bounds-Lr-norm} for $\phi$ and obtain
    \begin{align}
        &\left|\int_{B_{2}\left(0\right)}\left(\phi\dell_{t}\eta_{\mathfrak{i}}\right)\left(-\Delta_{y}\right)^{-1}\left(\eta_{\mathfrak{i}}\phi\right)dy\right|\label{L2-interior-a-priori-estimate-lemma-10}\\
        &\leq \frac{C}{t^{2}}\left(\frac{\varepsilon}{t}\right)^{\left(\frac{\gamma-1}{8}\right)}\alpha_{\mathfrak{e}}\left(t\right)^{2}+C\left(\frac{\varepsilon}{t}\right)^{\left(\frac{\gamma-1}{8}\right)}\|\left(\gamma\left(\Gamma+a_{*}\right)^{\gamma-1}_{+}\right)^{-\frac{1}{2}}\eta_{\mathfrak{i}}\phi\|_{L^{2}\left(\mathbb{R}^{2}\right)}^{2},\nonumber
    \end{align}
    where we have applied Young's Inequality.

    \medskip
    Let $\bar{\eta}_{\mathfrak{i}}=1-\eta_{\mathfrak{i}}$. For the term on the second line on the left hand side of \eqref{L2-interior-a-priori-estimate-lemma-2}, we first note that since $\psi=\left(-\Delta_{y}\right)^{-1}\left(\phi\right)$,
    \begin{align}
        &\nabla_{y}\left(\gamma\left(\Gamma+a_{*}\right)^{\gamma-1}_{+}\psi\right)=\nabla_{y}\left(\gamma\left(\Gamma+a_{*}\right)^{\gamma-1}_{+}\left(-\Delta_{y}\right)^{-1}\left(\bar{\eta}_{\mathfrak{i}}\phi\right)\right)+\nabla_{y}\left(\gamma\left(\Gamma+a_{*}\right)^{\gamma-1}_{+}\left(-\Delta_{y}\right)^{-1}\left(\eta_{\mathfrak{i}}\phi\right)\right).\label{L2-interior-a-priori-estimate-lemma-11}
    \end{align}
    Substituting \eqref{L2-interior-a-priori-estimate-lemma-11} in to the integral on the second line of \eqref{L2-interior-a-priori-estimate-lemma-2}, we obtain
    \begin{align}
        &\int_{B_{2}\left(0\right)}\eta_{\mathfrak{i}}G\  \gradperp_{y}\left(\Gamma+a_{*}+a\right)\cdot\nabla_{y}\left(\phi-\gamma\left(\Gamma+a_{*}\right)^{\gamma-1}_{+}\psi\right)dy\label{L2-interior-a-priori-estimate-lemma-12}\\
        &=\int_{B_{2}\left(0\right)}\eta_{\mathfrak{i}}G\  \gradperp_{y}\left(\Gamma+a_{*}+a\right)\cdot\nabla_{y}\left(\gamma\left(\Gamma+a_{*}\right)^{\gamma-1}_{+}G\right)dy+\sum_{i=1}^{7}\mathcal{I}_{i},\nonumber
    \end{align}
    where, noting that $\gamma\left(\Gamma+a_{*}\right)^{\gamma-1}_{+}$ is a function of $\Gamma+a_{*}$, we have
    \begin{align}
        \mathcal{I}_{1}=-\int_{B_{2}\left(0\right)}\frac{\eta_{\mathfrak{i}}^{2}\phi}{\gamma\left(\Gamma+a_{*}\right)^{\gamma-1}_{+}}\left(-\Delta_{y}\right)^{-1}\left(\bar{\eta}_{\mathfrak{i}}\phi\right)\gradperp_{y}\left(a\right)\cdot\nabla_{y}\left(\gamma\left(\Gamma+a_{*}\right)^{\gamma-1}_{+}\right)\ dy,\label{L2-interior-a-priori-estimate-lemma-13}
    \end{align}
    \begin{align}
        \mathcal{I}_{2}=\int_{B_{2}\left(0\right)}\eta_{\mathfrak{i}}\left(-\Delta_{y}\right)^{-1}\left(\eta_{\mathfrak{i}}\phi\right)\left(-\Delta_{y}\right)^{-1}\left(\bar{\eta}_{\mathfrak{i}}\phi\right)\gradperp_{y}\left(a\right)\cdot\nabla_{y}\left(\gamma\left(\Gamma+a_{*}\right)^{\gamma-1}_{+}\right)\ dy,\label{L2-interior-a-priori-estimate-lemma-14}
    \end{align}
    \begin{align}
        \mathcal{I}_{3}=-\int_{B_{2}\left(0\right)}\eta_{\mathfrak{i}}^{2}\phi\nabla_{y}\left(\left(-\Delta_{y}\right)^{-1}\left(\bar{\eta}_{\mathfrak{i}}\phi\right)\right)\cdot\gradperp_{y}\left(\Gamma+a_{*}\right)^{\gamma}_{+}\ dy,\label{L2-interior-a-priori-estimate-lemma-15}
    \end{align}
    \begin{align}
        &\mathcal{I}_{4}=\int_{B_{2}\left(0\right)}\eta_{\mathfrak{i}}\left(-\Delta_{y}\right)^{-1}\left(\eta_{\mathfrak{i}}\phi\right)\nabla_{y}\left(\left(-\Delta_{y}\right)^{-1}\left(\bar{\eta}_{\mathfrak{i}}\phi\right)\right)\cdot\gradperp_{y}\left(\Gamma+a_{*}\right)^{\gamma}_{+}\ dy,\label{L2-interior-a-priori-estimate-lemma-16}
    \end{align}
    \begin{align}
        \mathcal{I}_{5}=\int_{B_{2}\left(0\right)}\eta_{\mathfrak{i}}G\phi\  \gradperp_{y}\left(\Gamma+a_{*}+a\right)\cdot\nabla_{y}\bar{\eta}_{\mathfrak{i}}\ dy,\label{L2-interior-a-priori-estimate-lemma-16a}
    \end{align}
    \begin{align}
        \mathcal{I}_{6}=-\int_{B_{2}\left(0\right)}\frac{\eta_{\mathfrak{i}}^{2}\bar{\eta}_{\mathfrak{i}}\phi^{2}}{2\left(\Gamma+a_{*}\right)^{\gamma}_{+}}\nabla_{y}\left(\Gamma+a_{*}\right)\cdot\gradperp_{y}\left(a\right)dy,\label{L2-interior-a-priori-estimate-lemma-16b}
    \end{align}
    \begin{align}
        \mathcal{I}_{7}=-\int_{B_{2}\left(0\right)}\frac{\phi^{2}}{2\gamma\left(\Gamma+a_{*}\right)^{\gamma-1}_{+}}\  \gradperp_{y}\left(\Gamma+a_{*}+a\right)\cdot\nabla_{y}\left(\eta_{\mathfrak{i}}^{2}\bar{\eta}_{\mathfrak{i}}\right) dy,\label{L2-interior-a-priori-estimate-lemma-16c}
    \end{align}
    where we have used integration by parts for \eqref{L2-interior-a-priori-estimate-lemma-16b}--\eqref{L2-interior-a-priori-estimate-lemma-16c}. We begin to estimate \eqref{L2-interior-a-priori-estimate-lemma-12} with the first term on the right hand side, and note
    \begin{align}
        &\int_{B_{2}\left(0\right)}\eta_{\mathfrak{i}}G\ \gradperp_{y}\left(\Gamma+a_{*}+a\right)\cdot\nabla_{y}\left(\gamma\left(\Gamma+a_{*}\right)^{\gamma-1}_{+}G\right)dy\nonumber\\
        &=-\frac{1}{2}\int_{B_{2}\left(0\right)}\gamma\left(\Gamma+a_{*}\right)^{\gamma-1}_{+}G^{2}\ \gradperp_{y}\left(\Gamma+a_{*}+a\right)\cdot\nabla_{y}\left(\eta_{\mathfrak{i}}\right)dy\nonumber\\
        &+\frac{1}{2}\int_{B_{2}\left(0\right)}\eta_{\mathfrak{i}}\gamma\left(\gamma-1\right)\left(\Gamma+a_{*}\right)^{\gamma-2}_{+}G^{2}\ \gradperp_{y}\left(a\right)\cdot\nabla_{y}\left(\Gamma+a_{*}\right)dy,\label{L2-interior-a-priori-estimate-lemma-17}
    \end{align}
    where we use integration by parts. For the first term on the right hand side of \eqref{L2-interior-a-priori-estimate-lemma-17}, recalling the definition of $G$ in \eqref{L2-interior-a-priori-estimate-lemma-1}, we have
    \begin{align}
        &\left|\int_{B_{2}\left(0\right)}\gamma\left(\Gamma+a_{*}\right)^{\gamma-1}_{+}G^{2}\left(y,t\right)\ \gradperp_{y}\left(\Gamma+a_{*}+a\right)\cdot\nabla_{y}\left(\eta_{\mathfrak{i}}\left(y,t\right)\right)dy\right|\label{L2-interior-a-priori-estimate-lemma-18}\\
        &\leq C\int_{B_{2}\left(0\right)}\frac{\eta_{\mathfrak{i}}^{2}\phi^{2}}{\gamma\left(\Gamma+a_{*}\right)^{\gamma-1}_{+}}\left|\nabla_{y}\left(\eta_{\mathfrak{i}}\left(y,t\right)\right)\right|dy\nonumber\\
        &+C\left(\frac{\varepsilon}{t}\right)^{\frac{\gamma-1}{4}}\int_{B_{2}\left(0\right)}\left|\left(-\Delta_{y}\right)^{-1}\left(\eta_{\mathfrak{i}}\phi\right)\right|^{2}\left|\nabla_{y}\left(\eta_{\mathfrak{i}}\left(y,t\right)\right)\right|dy.\nonumber
    \end{align}
    For the first term on the right hand side, we note that the integrand is supported on $\supp{\nabla_{y}\eta_{\mathfrak{i}}}$, where \eqref{derivative-of-inner-cutoff-outer-cutoff-comparison-bound-1} gives that $\eta_{\mathfrak{i}}\leq C\eta_{\mathfrak{e}}$. For the second term, we can once again use the fact that the integrand is supported on $\supp{\nabla_{y}\eta_{\mathfrak{i}}}\subset\supp{\eta_{\mathfrak{e}}}$, whence applying \eqref{f-prime-L2-inner-weight-comparison-bound}, and \eqref{weighted-L2-exterior-cutoff-support} gives the necessary bound on  $\mathscr{W}_{1}\left(\Gamma+a_{*}\right)$ to provide the powers of $\varepsilon$ and $t$ appearing in front of the integral. Finally, we can bound the $L^{2}$ norm of $\left(-\Delta_{y}\right)^{-1}\left(\eta_{\mathfrak{i}}\phi\right)$ in terms of the $L^{2}$ norm of $\eta_{i}\phi$ and use the compact support to control any logarithmic terms. Taking all this together gives us
    \begin{align}
        &\left|\int_{B_{2}\left(0\right)}\gamma\left(\Gamma+a_{*}\right)^{\gamma-1}_{+}G^{2}\left(y,t\right)\ \gradperp_{y}\left(\Gamma+a_{*}+a\right)\cdot\nabla_{y}\left(\eta_{\mathfrak{i}}\left(y,t\right)\right)dy\right|\label{L2-interior-a-priori-estimate-lemma-18a}\\
        &\leq C\left(\alpha_{\mathfrak{e}}\left(t\right)^{2}+\left(\frac{\varepsilon}{t}\right)^{\frac{\gamma-1}{4}}\|\left(\gamma\left(\Gamma+a_{*}\right)^{\gamma-1}_{+}\right)^{-\frac{1}{2}}\eta_{\mathfrak{i}}\phi\|_{L^{2}\left(\mathbb{R}^{2}\right)}^{2}\right).\nonumber
    \end{align}
    For the second term on the right hand side of \eqref{L2-interior-a-priori-estimate-lemma-17}, using \eqref{linear-transport-operator-a-bounds}, \eqref{f-prime-L2-inner-weight-comparison-bound}, and \eqref{weighted-L2-interior-cutoff-support}, we have the bound
    \begin{align}
        \left|\eta_{\mathfrak{i}}\left(y,t\right)\gamma\left(\gamma-1\right)\left(\Gamma+a_{*}\right)^{\gamma-2}_{+} \gradperp_{y}\left(a\right)\right|&\leq \frac{C\varepsilon^{\frac{11}{4}}}{t^{\frac{5}{4}}}\frac{\eta_{\mathfrak{i}}\left(y,t\right)\gamma\left(\Gamma+a_{*}\right)^{\gamma-1}_{+}}{\Gamma+a_{*}}\label{L2-interior-a-priori-estimate-lemma-19}\\
        &\leq\frac{C\varepsilon^{\frac{5}{2}}}{t}\eta_{\mathfrak{i}}\left(y,t\right)\gamma\left(\Gamma+a_{*}\right)^{\gamma-1}_{+}.\nonumber
    \end{align}
    Then, analogously to \eqref{L2-interior-a-priori-estimate-lemma-18}--\eqref{L2-interior-a-priori-estimate-lemma-18a}, we have
    \begin{align}
        &\left|\int_{B_{2}\left(0\right)}\eta_{\mathfrak{i}}\left(y,t\right)\gamma\left(\gamma-1\right)\left(\Gamma+a_{*}\right)^{\gamma-2}_{+}G^{2}\left(y,t\right)\ \gradperp_{y}\left(a\right)\cdot\nabla_{y}\left(\Gamma+a_{*}\right)dy\right|\label{L2-interior-a-priori-estimate-lemma-20}\\
        &\leq\frac{C\varepsilon^{\frac{5}{2}}}{t}\left(\|\left(\gamma\left(\Gamma+a_{*}\right)^{\gamma-1}_{+}\right)^{-\frac{1}{2}}\eta_{\mathfrak{i}}\phi\|_{L^{2}\left(\mathbb{R}^{2}\right)}^{2}+\alpha_{\mathfrak{e}}\left(t\right)^{2}\right).\nonumber
    \end{align}
    Similar considerations give us, for $j=1,\dots,7$,
    \begin{align}\nonumber
        \sum_{j=1}^{7}\left|\mathcal{I}_{j}\right|\leq \frac{C\varepsilon^{\frac{5}{2}}}{t}\|\left(\gamma\left(\Gamma+a_{*}\right)^{\gamma-1}_{+}\right)^{-\frac{1}{2}}\eta_{\mathfrak{i}}\phi\|_{L^{2}\left(\mathbb{R}^{2}\right)}^{2}+C\alpha_{\mathfrak{e}}\left(t\right)^{2}.
    \end{align}
    Finally, we estimate the last term on the left hand side of \eqref{L2-interior-a-priori-estimate-lemma-2} given by
    \begin{align}\nonumber
        \int_{B_{2}\left(0\right)}\left(\frac{\eta_{\mathfrak{i}}^{2}\phi}{\gamma\left(\Gamma+a_{*}\right)^{\gamma-1}_{+}}-\eta_{\mathfrak{i}}\left(-\Delta_{y}\right)^{-1}\left(\eta_{\mathfrak{i}}\phi\right)\right)E\ dy.
    \end{align}
    We again use \eqref{f-prime-L2-inner-weight-comparison-bound} as well as Young's inequality to obtain that
    \begin{align}
        &\left|\int_{B_{2}\left(0\right)}\left(\frac{\eta_{\mathfrak{i}}^{2}\phi}{\gamma\left(\Gamma+a_{*}\right)^{\gamma-1}_{+}}-\eta_{\mathfrak{i}}\left(-\Delta_{y}\right)^{-1}\left(\eta_{\mathfrak{i}}\phi\right)\right)E\ dy\right|\label{L2-interior-a-priori-estimate-lemma-30}\\
        &\leq \frac{C\varepsilon^{2}}{\left|\log{\varepsilon}\right|t}\|\left(\gamma\left(\Gamma+a_{*}\right)^{\gamma-1}_{+}\right)^{-\frac{1}{2}}\eta_{\mathfrak{i}}\phi\|_{L^{2}\left(\mathbb{R}^{2}\right)}^{2}+\frac{Ct\left|\log{\varepsilon}\right|}{\varepsilon^{2}}\left|\log{\varepsilon}\right|^{2}\left\|E\right\|_{\mathscr{Y}_{\infty}}^{2},\nonumber
    \end{align}
    Then, using the identities \eqref{L2-interior-a-priori-estimate-lemma-3}, \eqref{L2-interior-a-priori-estimate-lemma-7}, \eqref{L2-interior-a-priori-estimate-lemma-12}--\eqref{L2-interior-a-priori-estimate-lemma-16}, as well as bounds \eqref{L2-interior-a-priori-estimate-lemma-6}, \eqref{L2-interior-a-priori-estimate-lemma-10}, \eqref{L2-interior-a-priori-estimate-lemma-18}, and \eqref{L2-interior-a-priori-estimate-lemma-20}--\eqref{L2-interior-a-priori-estimate-lemma-30}, we have
    \begin{align}
        &\frac{\varepsilon^{2}}{2}\frac{d}{dt}\left(\int_{B_{2}\left(0\right)}\left(\frac{\eta_{\mathfrak{i}}^{2}\phi^{2}}{\gamma\left(\Gamma+a_{*}\right)^{\gamma-1}_{+}}-\left(\eta_{\mathfrak{i}}\phi\right)\left(-\Delta_{y}\right)^{-1}\left(\eta_{\mathfrak{i}}\phi\right)\right)dy\right)\nonumber\\
        &\geq-\frac{C\varepsilon^{2}}{t\left|\log{\varepsilon}\right|}\|\left(\gamma\left(\Gamma+a_{*}\right)^{\gamma-1}_{+}\right)^{-\frac{1}{2}}\eta_{\mathfrak{i}}\phi\|_{L^{2}\left(\mathbb{R}^{2}\right)}^{2}-\frac{Ct\left|\log{\varepsilon}\right|^{3}}{\varepsilon^{2}}\left\|E\right\|_{\mathscr{Y}_{\infty}}^{2}-C\alpha_{\mathfrak{e}}\left(t\right)^{2},
        \nonumber
    \end{align}
    Then integrating on $\left[t,T\right]$, and noting that $\phi\left(T\right)=0$ from \eqref{linear-transport-initial-data}, we have
    \begin{align}
        &\int_{B_{2}\left(0\right)}\left(\frac{\eta_{\mathfrak{i}}^{2}\phi^{2}}{\gamma\left(\Gamma+a_{*}\right)^{\gamma-1}_{+}}-\left(\eta_{\mathfrak{i}}\phi\right)\left(-\Delta_{y}\right)^{-1}\left(\eta_{\mathfrak{i}}\phi\right)\right)dy\leq\int_{t}^{T}C\varepsilon^{-2}\alpha_{\mathfrak{e}}\left(\tau\right)^{2}d\tau\label{L2-interior-a-priori-estimate-lemma-32}\\
        &+\int_{t}^{T}\frac{C}{\tau\left|\log{\varepsilon}\right|}\|\left(\gamma\left(\Gamma+a_{*}\right)^{\gamma-1}_{+}\right)^{-\frac{1}{2}}\eta_{\mathfrak{i}}\phi\|_{L^{2}\left(\mathbb{R}^{2}\right)}^{2} d\tau+\int_{t}^{T}\frac{C\tau\left|\log{\varepsilon}\right|^{3}}{\varepsilon^{4}}\left\|E\right\|_{\mathscr{Y}_{\infty}}^{2}d\tau.\nonumber
    \end{align}
    Using \eqref{f-prime-L2-inner-weight-comparison-bound}, we have
    \begin{align}
        \left|\int_{B_{2}\left(0\right)}\eta_{\mathfrak{i}}^{2}\phi^{2}\left(\frac{1}{\gamma\left(\Gamma+a_{*}\right)^{\gamma-1}_{+}}-\frac{1}{\gamma\Gamma^{\gamma-1}_{+}}\right)dy\right|\leq C\left(\frac{\varepsilon}{t}\right)^{\frac{3}{2}}\int_{B_{2}\left(0\right)}\frac{\eta_{\mathfrak{i}}^{2}\phi^{2}}{\gamma\Gamma^{\gamma-1}_{+}}dy.\label{L2-interior-a-priori-estimate-lemma-34}
    \end{align}
    Then applying Lemma \ref{quadratic-form-estimate-lemma} and \eqref{L2-interior-a-priori-estimate-lemma-34} to \eqref{L2-interior-a-priori-estimate-lemma-32}, we obtain
    \begin{align}
        \int_{B_{2}\left(0\right)}\frac{\eta_{\mathfrak{i}}^{2}\phi^{2}}{\gamma\Gamma^{\gamma-1}_{+}}dy+\mathcal{R}_{1}&\leq\int_{t}^{T}\frac{C}{\tau\left|\log{\varepsilon}\right|}\|\left(\gamma\left(\Gamma+a_{*}\right)^{\gamma-1}_{+}\right)^{-\frac{1}{2}}\eta_{\mathfrak{i}}\phi\|_{L^{2}\left(\mathbb{R}^{2}\right)}^{2} d\tau\label{L2-interior-a-priori-estimate-lemma-35}\\
        &+\int_{t}^{T}\frac{C\tau\left|\log{\varepsilon}\right|^{3}}{\varepsilon^{4}}\left\|E\right\|_{\mathscr{Y}_{\infty}}^{2}ds+\int_{t}^{T}C\varepsilon^{-2}\alpha_{\mathfrak{e}}\left(\tau\right)^{2}d\tau.\nonumber
    \end{align}
    where we recall from Lemma \ref{quadratic-form-estimate-lemma} that
    \begin{align}\nonumber
        \mathcal{R}_{1}=\bigO\left(\left\|\eta_{\mathfrak{e}}\bar{\phi}\right\|_{L^{2}\left(\mathbb{R}^{2}\right)}^{2}+\left\|\eta_{\mathfrak{e}}\bar{\phi}\right\|_{L^{2}\left(\mathbb{R}^{2}\right)}\left\|\eta_{\mathfrak{i}}\bar{\phi}\right\|_{L^{2}\left(\mathbb{R}^{2}\right)}\right).
    \end{align}
    Now define $\alpha_{\mathfrak{i}}$ by
    \begin{align}\nonumber
        \alpha_{\mathfrak{i}}\left(t\right)^{2}=\int_{B_{2}\left(0\right)}\frac{\eta_{\mathfrak{i}}^{2}\phi^{2}}{\gamma\Gamma^{\gamma-1}_{+}}dy.
    \end{align}
    Using \eqref{L2-interior-a-priori-estimate-lemma-34}, \eqref{L2-interior-a-priori-estimate-lemma-35}, , \eqref{f-prime-L2-inner-weight-comparison-bound}, Lemma \ref{L2-exterior-a-priori-estimate-lemma}, \eqref{weighted-Lr-norm-bounds-Lr-norm}, and recalling the definition \eqref{alpha-exterior-definition}, we have
    \begin{align}
        \alpha_{\mathfrak{i}}\left(t\right)^{2}&\leq \int_{t}^{T}\frac{C\tau\left|\log{\varepsilon}\right|^{3}}{\varepsilon^{4}}\left\|E\right\|_{\mathscr{Y}_{\infty}}^{2}d\tau+\frac{C\varepsilon^{\frac{\gamma-11}{4}}}{t^{\frac{\gamma-3}{4}}}\left(\int_{t}^{T}\left\|E\right\|_{\mathscr{Y}_{\infty}}d\tau\right)^{2}\label{L2-interior-a-priori-estimate-lemma-39}\\
        &+\int_{t}^{T}\frac{C\varepsilon^{\frac{\gamma-19}{4}}}{\tau^{\frac{\gamma-3}{4}}}\left(\int_{\tau}^{T}\left\|E\right\|_{\mathscr{Y}_{\infty}}d\tau'\right)^{2}d\tau+\int_{t}^{T}\frac{C}{\tau\left|\log{\varepsilon}\right|}\alpha_{\mathfrak{i}}\left(\tau\right)^{2}d\tau\nonumber\\
        &+\frac{C\varepsilon^{\frac{\gamma-11}{4}}}{t^{\frac{\gamma-3}{4}}}\left(\int_{t}^{T}\alpha_{\mathfrak{i}}\left(\tau\right)d\tau\right)^{2}+\int_{t}^{T}\frac{C\varepsilon^{\frac{\gamma-19}{4}}}{\tau^{\frac{\gamma-3}{4}}}\left(\int_{\tau}^{T}\alpha_{\mathfrak{i}}\left(\tau'\right)d\tau'\right)^{2}d\tau\nonumber\\
        &+C\alpha_{\mathfrak{i}}\left(t\right)\frac{\varepsilon^{\frac{\gamma-11}{4}}}{t^{\frac{\gamma-3}{4}}}\left(\int_{t}^{T}\left(\left\|E\right\|_{\mathscr{Y}_{\infty}}+\alpha_{\mathfrak{i}}\right)d\tau'\right).\nonumber
    \end{align}
    Now, using the bound \eqref{L2-interior-a-priori-estimate-lemma-error-bound-statement}, and the fact that the left hand side of \eqref{L2-interior-a-priori-estimate-lemma-39} is $0$ at $t=T$, as long as $\gamma\geq19$, we have that \eqref{L2-interior-a-priori-estimate-lemma-39} implies \eqref{L2-interior-a-priori-estimate-lemma-statement} for all $T_{0}>0$ large enough, but still independent of $\varepsilon>0$, via a standard continuation argument.
\end{proof}
Using Lemmas \ref{L2-exterior-a-priori-estimate-lemma} and \ref{L2-interior-a-priori-estimate-lemma}, we can get a direct estimate on solutions to the $L^{2}$ norm of solutions to \eqref{linear-transport-operator-system} satisfying assumptions \eqref{linear-transport-operator-phi-mass-conditions}--\eqref{linear-transport-operator-error-solution-supports}.
\begin{corollary}\label{L2-a-priori-estimate-corollary}
    Suppose the assumptions of Lemmas \ref{L2-exterior-a-priori-estimate-lemma} and \ref{L2-interior-a-priori-estimate-lemma} hold. Then for all $\varepsilon>0$ small enough, and all $T_{0}>0$ large enough, we have for all $t\in\left[T_{0},T\right]$,
    \begin{align} \nonumber
        \left\|\phi\right\|_{L^{2}\left(\mathbb{R}^{2}\right)}\leq\frac{C\varepsilon^{3-\sigma_{0}}\left|\log{\varepsilon}\right|^{\frac{3}{2}}}{t^{\frac{3}{2}}},
    \end{align}
    for some absolute constant $C>0$ and $\sigma_{0}>0$ defined in \eqref{L2-interior-a-priori-estimate-lemma-error-bound-statement}.
\end{corollary}
We can also use \ref{poisson-equation-zero-mass-estimates-lemma} alongside Lemmas \ref{L2-exterior-a-priori-estimate-lemma} and \ref{L2-interior-a-priori-estimate-lemma} to gain $L^{r}$ estimates on $\phi$, some $2<r<\infty$, $L^{\infty}$ estimates on $\phi$ and $\psi$, and H{\"o}lder estimates on $\nabla\psi$, for $\phi$ solving \eqref{linear-transport-operator-system} and satisfying assumptions \eqref{linear-transport-operator-phi-mass-conditions}--\eqref{linear-transport-operator-error-solution-supports}. We recall that here $\psi=\left(-\Delta_{y}\right)^{-1}\left(\phi\right)$.
\begin{corollary}\label{L-infinity-a-priori-estimate-corollary}
    Suppose the assumptions of Lemmas \ref{L2-exterior-a-priori-estimate-lemma} and \ref{L2-interior-a-priori-estimate-lemma} hold. Then for all $\varepsilon>0$ small enough, and all $T_{0}>0$ large enough, we have for all $t\in\left[T_{0},T\right]$,
    \begin{align}
        \left|\psi\right|+\left|\nabla\psi\right|+\left[\nabla\psi\right]_{\beta}\left(y\right)&\leq \frac{C\varepsilon^{3-\sigma_{1}}}{t^{\frac{3}{2}-\sigma_{1}}},\label{L-infinity-a-priori-estimate-corollary-statement-2}\\
        t\varepsilon^{-2}\left\|\phi\right\|_{L^{r_{*}}\left(\mathbb{R}^{2}\right)}+\left\|\phi\right\|_{L^{\infty}\left(\mathbb{R}^{2}\right)}&\leq \frac{C\varepsilon^{1-\sigma_{1}}}{t^{\frac{1}{2}-\sigma_{1}}}\label{L-infinity-a-priori-estimate-corollary-statement-4},
    \end{align}
    for some $2<r_{*}<\infty$, some absolute constant $C>0$, some small $\sigma_{1}>0$.
\end{corollary}
\begin{proof}
    The proof follows by representing the solution $\phi$ to \eqref{linear-transport-operator-system} using the characteristics defined in \eqref{characteristics-ode-system}, and then applying Lemma \ref{phi-interpolation-estimate-lemma}.
\end{proof}
\subsection{A Priori Estimates on a Projected Linear Problem}\label{projected-transport-problem-a-priori-estimates-section}
In view of \eqref{tilde-phi-star-R-initial-value-problem}, we want to show a priori estimates on the coefficients of projection in the following linear problem
\begin{subequations}\label{projected-linear-transport-operator-system}
\begin{align}
&\varepsilon^{2}\phi_{t}+\gradperp_{y}\left(\Gamma+a_{*}+a\right)\cdot\nabla_{y}\left(\phi-\gamma\left(\Gamma+a_{*}\right)^{\gamma-1}_{+}\psi\right)+E=c_{j1}\left(t\right)\dell_{1}U(y)\nonumber\\
&+c_{j2}\left(t\right)\dell_{2}U(y)\ \ \text{in}\ \mathbb{R}^{2}\times \left[T_{0},T\right],\label{projected-linear-transport-equation}\\
&\phi(\blank,T)=0\ \ \text{in}\ \mathbb{R}^{2},\label{projected-linear-transport-initial-data}\\
&\phi_{*j}\to0,\ \left|y\right|\to\infty,\ \forall t\in\left[T_{0},T\right],\label{projected-linear-transport-boundary-condition}
\end{align}
\end{subequations}
with the same assumptions on $\phi,a,a_{*}$, and $E$ as in \eqref{linear-transport-operator-phi-mass-conditions}--\eqref{linear-transport-operator-error-solution-supports}.

\medskip
To obtain a priori estimates for the coefficients of projection in \eqref{projected-linear-transport-operator-system}, we also need an extra assumption on $a_{*}$ that we justify here.

\medskip
We note from \eqref{homotopic-operator-inner-error-E-star-R-definition-10} that the highest order terms from $\mathfrak{a}_{j}$ in terms of both $\varepsilon$ and $t$ come from the mode $2$ terms $\mathcal{E}^{\left(ji\right)}_{2}$, $i\neq j$, and $\psi_{j1}^{\left(2\right)}$, both of order $\varepsilon^{2}t^{-1}$. However, one can see the cancellation of these terms that occurs when one calculates $\Delta_{y}\left(\Gamma+\mathfrak{a}_{j}\right)+\left(\Gamma+\mathfrak{a}_{j}\right)^{\gamma}_{+}$ directly, and so since $a_{*}$ represents the part of $\mathfrak{a}_{j}$ that does not depend on any unknowns, we can assume
\begin{align}
    \left\|\Delta_{y}\left(\Gamma+a_{*}\right)+\left(\Gamma+a_{*}\right)^{\gamma}_{+}\right\|_{C^{1}}\leq\frac{\varepsilon^{2+\nu}}{t^{\frac{3}{2}}},\label{a-star-extra-assumption-statement-1}    
\end{align}
for some $\nu>3/4$.

\medskip
Now we state and prove a priori estimates on $\left(c_{j1},c_{j2}\right)$.
\begin{lemma}\label{projected-linear-transport-problem-coefficients-of-projection-a-priori-estimates-lemma}
    Suppose $\phi$ solves \eqref{projected-linear-transport-operator-system} with assumptions \eqref{linear-transport-operator-phi-mass-conditions}--\eqref{linear-transport-operator-error-solution-supports}, as well as the size assumption \eqref{L2-interior-a-priori-estimate-lemma-error-bound-statement} on $E$ from Lemma \ref{L2-interior-a-priori-estimate-lemma}, and the extra assumptions on $a_{*}$, \eqref{a-star-extra-assumption-statement-1}. Then for all $\varepsilon>0$ small enough, and all $T_{0}>0$ large enough, we have that for all $t\in\left[T_{0},T\right]$ that $\left(c_{j0},c_{j1},c_{j2}\right)$ satisfies equations of the form
    \begin{align}
        -Mc_{ji}&=\int_{\mathbb{R}^{2}}y_{i}E\ dy+\bigO{\left(\frac{\varepsilon^{2+\nu}}{t^{\frac{3}{2}}}\right)}\left\|\phi\right\|_{L^{2}\left(\mathbb{R}^{2}\right)},\label{projected-linear-transport-problem-coefficients-of-projection-a-priori-estimates-lemma-error-form-1}
    \end{align}
    where $M$ is the absolute constant defined in \eqref{vortex-mass-definition}, for $i=1,2$, and $\nu>3/4$.
\end{lemma}
\begin{proof}
    We test \eqref{projected-linear-transport-equation} with $y_{1}$, and integrate on $\mathbb{R}^{2}$. The first term on the left hand side will again disappear due to \eqref{linear-transport-operator-phi-mass-conditions}. The first and third terms on the right hand side will disappear because multiplying by $y_{1}$ will make the integrands odd in $y_{1}$. We are left with
    \begin{align}
        &\int_{\mathbb{R}^{2}}y_{1}E\ dy+\int_{\mathbb{R}^{2}}y_{1}\gradperp_{y}\left(\Gamma+a_{*}+a\right)\cdot\nabla_{y}\left(\phi-\gamma\left(\Gamma+a_{*}\right)^{\gamma-1}_{+}\psi\right)dy\label{projected-linear-transport-problem-coefficients-of-projection-a-priori-estimates-lemma-3}\\
        &=c_{j1}\left(\int_{\mathbb{R}^{2}}y_{1}\dell_{1}U(y)dy\right)=-Mc_{j1},\nonumber
    \end{align}
    where the last equality is due to integration by parts.

    \medskip
    For the second term on the left hand side of \eqref{projected-linear-transport-problem-coefficients-of-projection-a-priori-estimates-lemma-3}, we use integration by parts to obtain
    \begin{align}
        &\int_{\mathbb{R}^{2}}y_{1}\gradperp_{y}\left(\Gamma+a_{*}+a\right)\cdot\nabla_{y}\left(\phi-\gamma\left(\Gamma+a_{*}\right)^{\gamma-1}_{+}\psi\right)dy\label{projected-linear-transport-problem-coefficients-of-projection-a-priori-estimates-lemma-5}\\
        &=\int_{\mathbb{R}^{2}}\dell_{y_{2}}\left(\Gamma+a_{*}\right)\left(\Delta_{y}\psi+\gamma\left(\Gamma+a_{*}\right)^{\gamma-1}_{+}\psi\right)dy+\bigO{\left(\frac{\varepsilon^{2+\nu}}{t^{\frac{3}{2}}}\right)}\left\|\phi\right\|_{L^{2}\left(\mathbb{R}^{2}\right)}.\nonumber
    \end{align}
    Here we have once again used the compact support of the integrand to see that all boundary terms are $0$. Then we used that fact that $-\Delta_{y}\psi=\phi$ to obtain the first term on the right hand side of \eqref{projected-linear-transport-problem-coefficients-of-projection-a-priori-estimates-lemma-5}. Finally, we used assumption \eqref{linear-transport-operator-a-bounds-0}, and Lemma \ref{poisson-equation-zero-mass-estimates-lemma} on $a$ to obtain the last term on the right hand side of \eqref{projected-linear-transport-problem-coefficients-of-projection-a-priori-estimates-lemma-5}.

    \medskip
    Concentrating on the first term on the right hand side of \eqref{projected-linear-transport-problem-coefficients-of-projection-a-priori-estimates-lemma-5}, and once again using integration by parts, we obtain that it is equal to
    \begin{align}
        \int_{\mathbb{R}^{2}}\psi\dell_{y_{2}}\left(\Delta_{y}\left(\Gamma+a_{*}\right)+\left(\Gamma+a_{*}\right)^{\gamma}_{+}\right)dy=\bigO{\left(\frac{\varepsilon^{2+\nu}}{t^{\frac{3}{2}}}\right)}\left\|\phi\right\|_{L^{2}\left(\mathbb{R}^{2}\right)},\label{projected-linear-transport-problem-coefficients-of-projection-a-priori-estimates-lemma-6}
    \end{align}
    where we have used Lemma \ref{poisson-equation-zero-mass-estimates-lemma} and \eqref{a-star-extra-assumption-statement-1} to obtain the right hand side of \eqref{projected-linear-transport-problem-coefficients-of-projection-a-priori-estimates-lemma-6}.

    \medskip
    Finally, we test \eqref{projected-linear-transport-equation} against $y_{2}$ and integrate. Once again the first term on the left hand side disappears due to assumption \eqref{linear-transport-operator-phi-mass-conditions}. The first and second terms on the right hand side disappears because the integrands become odd in $y_{2}$. We obtain
    \begin{align}
        &\int_{\mathbb{R}^{2}}y_{2}E\ dy+\int_{\mathbb{R}^{2}}y_{2}\gradperp_{y}\left(\Gamma+a_{*}+a\right)\cdot\nabla_{y}\left(\phi-\gamma\left(\Gamma+a_{*}\right)^{\gamma-1}_{+}\psi\right)dy\label{projected-linear-transport-problem-coefficients-of-projection-a-priori-estimates-lemma-7}\\
        &=c_{j2}\left(\int_{\mathbb{R}^{2}}y_{2}\dell_{2}U(y)dy\right)=-Mc_{j2},\nonumber
    \end{align}
    once again using integration by parts for the last equality.

    \medskip
    Analogously to \eqref{projected-linear-transport-problem-coefficients-of-projection-a-priori-estimates-lemma-5}--\eqref{projected-linear-transport-problem-coefficients-of-projection-a-priori-estimates-lemma-6}, we have for the second term on the left hand side of \eqref{projected-linear-transport-problem-coefficients-of-projection-a-priori-estimates-lemma-7} that
    \begin{align}
        \int_{\mathbb{R}^{2}}y_{2}\gradperp_{y}\left(\Gamma+a_{*}+a\right)\cdot\nabla_{y}\left(\phi-\gamma\left(\Gamma+a_{*}\right)^{\gamma-1}_{+}\psi\right)dy=\bigO{\left(\frac{\varepsilon^{2+\nu}}{t^{\frac{3}{2}}}\right)}\left\|\phi\right\|_{L^{2}\left(\mathbb{R}^{2}\right)}.\label{projected-linear-transport-problem-coefficients-of-projection-a-priori-estimates-lemma-10}
    \end{align}
    Noting that $\left(c_{j1},c_{j2}\right)$ satisfy the linear system given by \eqref{projected-linear-transport-problem-coefficients-of-projection-a-priori-estimates-lemma-3} and \eqref{projected-linear-transport-problem-coefficients-of-projection-a-priori-estimates-lemma-7}, and we have estimates on the coefficients and errors for the linear system given by \eqref{projected-linear-transport-problem-coefficients-of-projection-a-priori-estimates-lemma-5}, \eqref{projected-linear-transport-problem-coefficients-of-projection-a-priori-estimates-lemma-6}, and \eqref{projected-linear-transport-problem-coefficients-of-projection-a-priori-estimates-lemma-10} respectively, we obtain a linear system for $\left(c_{j1},c_{j2}\right)$ of the form \eqref{projected-linear-transport-problem-coefficients-of-projection-a-priori-estimates-lemma-error-form-1}, as required.
\end{proof}
\subsection{A Priori Estimates on the Poisson Equation}
We now obtain a priori estimates on solutions $\psi^{out}_{*}$ to \eqref{psi-out-star-boundary-value-problem-equation} given by
\begin{align}
    -\Delta_{x}\psi_{*}^{out}=\lambda\sum_{j=1}^{3}\kappa_{j}\left(\psi_{*j}\Delta_{x}\eta_{j}+2\nabla_{x}\eta_{j}\cdot\nabla_{x}\psi_{*j}\right)+\lambda E_{2}^{out}\left(\psi^{in},\psi^{out};\xi\right)\ \ \text{in}\ \mathbb{R}^{2}\times \left[T_{0},T\right].\label{psi-out-star-equation-2}
\end{align}
To find good a priori estimates on $\psi_{*}^{out}$ satisfying the above equation, we need estimates on the source term on the right hand side. First note that \eqref{first-approximation-outer-error-E2out-final-bound} gives $\lambda E_{2}^{out}\left(\psi^{in},\psi^{out};\xi\right)=\bigO{\left(\lambda\varepsilon^{6}t^{-4}\right)}$.

\medskip
Next, we proceed analogously to \eqref{psi-out-equation}--\eqref{psi-out-grad-y-bound} and note that for $\left|y\right|\geq C\varepsilon^{-1}t^{\frac{1}{2}}$ and for each $j$ we have
\begin{align*}
        \psi_{*j}\left(y\right)=-\frac{1}{4\pi}\int_{B_{2}\left(0\right)}\log\left(1-\frac{2z\cdot y}{\left|y\right|^{2}}+\frac{\left|z\right|^{2}}{\left|y\right|^{2}}\right)\phi_{*j}\left(z\right)dz-\frac{\log{\left|y\right|}}{2\pi}\int_{B_{2}\left(0\right)}\phi_{*j}\left(z\right)dz.
\end{align*}
The second term on the right hand side above immediately disappears due to the $0$ mass condition from \eqref{linear-transport-operator-phi-mass-conditions}. Next, expanding the logarithm in the first integral, we have
\begin{align*}
    \psi_{*j}\left(y\right)=-\frac{y}{2\pi\left|y\right|^{2}}\cdot\int_{B_{2}\left(0\right)}z\phi_{*j}\left(z\right)dz+\int_{B_{2}\left(0\right)}\bigO{\left(\left|y\right|^{-2}\right)}\phi_{*j}\left(z\right)dz.
\end{align*}
This time, due to the centre of mass conditions in \eqref{linear-transport-operator-phi-mass-conditions}, we have that the first term on the right hand side is $0$, which means that for $\left|y\right|\geq C\varepsilon^{-1}t^{\frac{1}{2}}$, we have
\begin{align*}
    \left|\psi_{*j}\left(y\right)\right|\leq \frac{C\varepsilon^{2}}{t}\left\|\phi_{*j}\right\|_{L^{2}\left(\mathbb{R}^{2}\right)},\quad \frac{1}{\varepsilon}\left|\nabla_{y}\psi_{*j}\left(y\right)\right|\leq \frac{C\varepsilon^{2}}{t^{\frac{3}{2}}}\left\|\phi_{*j}\right\|_{L^{2}\left(\mathbb{R}^{2}\right)}.
\end{align*}
Thus, after once again noting that $\Delta_{x}\eta_{j}=\bigO{\left(t^{-1}\right)}$ and $\nabla_{j}\eta_{j}=\bigO{\left(t^{-\frac{1}{2}}\right)}$, the right hand side of \eqref{psi-out-star-equation-2} has estimate
\begin{align}
    &\left|\lambda\sum_{j=1}^{3}\kappa_{j}\left(\psi_{*j}\Delta_{x}\eta_{j}+2\nabla_{x}\eta_{j}\cdot\nabla_{x}\psi_{*j}\right)+\lambda E_{2}^{out}\left(\psi^{in},\psi^{out};\xi\right)\right|\nonumber\\
    &\leq C\lambda\left(\frac{\varepsilon^{2}}{t^{2}}\sum_{j=1}^{3}\left\|\phi_{*j}\right\|_{L^{2}\left(\mathbb{R}^{2}\right)}+\frac{\varepsilon^{6}}{t^{4}}\right).\label{psi-out-star-equation-3}
\end{align}
Then \eqref{psi-out-star-equation-3} gives
\begin{align}
    \left\|\left(\log{\left(\left|\cdot\right|+2\right)}\right)^{-1}\psi_{*}^{out}\right\|_{L^{\infty}\left(\mathbb{R}^{2}\right)}+\left\|\nabla\psi_{*}^{out}\right\|_{L^{\infty}\left(\mathbb{R}^{2}\right)}\leq C\lambda\left(\frac{\varepsilon^{2}}{t}\sum_{j=1}^{3}\left\|\phi_{*j}\right\|_{L^{2}\left(\mathbb{R}^{2}\right)}+\frac{\varepsilon^{6}}{t^{3}}\right).\label{psi-out-star-equation-4}
\end{align}
Finally, for $\left|y\right|\leq 1$,
\begin{align}
        \psi^{out}\left(\varepsilon y+\xi_{j}\right)&=-\frac{1}{4\pi}\int_{\left|z-\xi_{j}\right|\geq K\left(1+\frac{t}{\tau}\right)^{\frac{1}{2}}}\log{\left(1-\frac{2\varepsilon\left(z-\xi_{j}\right).y}{\left|z-\xi_{j}\right|^{2}}+\frac{\varepsilon^{2}\left|y\right|^{2}}{\left|z-\xi_{j}\right|^{2}}\right)}\mathscr{S}_{*}(z)\ dz\nonumber\\
        &-\frac{1}{2\pi}\int_{\left|z-\xi_{j}\right|\geq K\left(1+\frac{t}{\tau}\right)^{\frac{1}{2}}}\log{\left|z-\xi_{j}\right|}\ \mathscr{S}_{*}(z)\ dz,\label{psi-out-star-equation-5}
\end{align}
where $\mathscr{S}_{*}$ is the source term on the right hand side of \eqref{psi-out-star-equation-2}, has bound given by \eqref{psi-out-star-equation-3}, and has support on three annuli of area $\sim t$. Using these facts and \eqref{psi-out-star-equation-5}, we obtain
\begin{align*}
    \left|\gradperp_{y}\psi^{out}_{*}\left(\varepsilon y+\xi_{j}\right)\cdot\nabla_{y}U(y)\right|&=\left|\gamma\Gamma^{\gamma-1}_{+}(y)\frac{\Gamma'(y)}{\left|y\right|}\dell_{\theta}\psi^{out}_{*}\left(\varepsilon y+\xi_{j}\right)\right|\nonumber\\
    &\leq C\lambda\left(\frac{\varepsilon^{3}}{t^{\frac{3}{2}}}\sum_{j=1}^{3}\left\|\phi_{*j}\right\|_{L^{2}\left(\mathbb{R}^{2}\right)}+\frac{\varepsilon^{7}}{t^{\frac{7}{2}}}\right)
\end{align*}
\subsection{A Priori Estimates on Solutions to the Homotopic Operators}
We now obtain a priori estimates to solutions $\phi_{*j}$ solving \eqref{tilde-phi-star-R-initial-value-problem} given by
\begin{align*}
&\mathscr{E}_{j,\lambda}\left(\phi_{*j},\psi^{out}_{*},\Tilde{\xi}\right)=c_{j1}\left(t\right)\dell_{1}U(y)+c_{j2}\left(t\right)\dell_{2}U(y)\ \ \text{in}\ \mathbb{R}^{2}\times \left[T_{0},T\right],\\
&\phi_{*j}(\blank,T)=0\ \ \text{in}\ \mathbb{R}^{2},\\
&\phi_{*j}\to0,\ \left|y\right|\to\infty,\ \forall t\in\left[T_{0},T\right],
\end{align*}
with $\left(\psi^{out}_{*},\Tilde{\xi}\right)$ given, and sufficiently small, and $\phi_{*j}$ satisfying the orthogonality conditions \eqref{tilde-phi-star-R-0-mass-condition}, which correspond to satisfying assumptions \eqref{linear-transport-operator-phi-mass-conditions} for solutions to our linearized operator.

\medskip
Recall from \eqref{homotopic-operator-inner-error-E-star-R-definition-10} that
\begin{align*}
    \mathfrak{a}_{j}=\sum_{i\neq j}\kappa_{i}\Gamma\left(y+\frac{\xi_{j}(t)-\xi_{i}(t)}{\varepsilon}\right)-\varepsilon \dot{\xi}\cdot y^{\perp}+\kappa_{j}\psi_{j}+\kappa_{j}\psi_{*j}+\psi^{out}+\psi^{out}_{*}
\end{align*}
Using Lemma \ref{remainder-modal-calculation-lemma} and Theorem \ref{first-approximation-construction-theorem}, we can define
\begin{align}
    a_{*}(y,t)=-\sum_{i\neq j}\frac{m_{i}}{4\pi}\mathcal{E}^{\left(ji\right)}_{2}(\xi_{*})+\kappa_{j}\psi_{j},\label{astar-definition}
\end{align}
and note that $a_{*}$ does not depend on any unknowns. Note that $a_{*}$ satisfies \eqref{linear-transport-operator-a-star-bounds} and \eqref{a-star-extra-assumption-statement-1}. Accordingly we can also write
\begin{align*}
    a(y,t)=\mathfrak{a}_{j}-a_{*}.
\end{align*}
Then we can write \eqref{homotopic-operator-inner-error-main-definition} as
\begin{align}
    &\mathscr{E}_{j,\lambda}\left(\phi_{*j},\psi^{out}_{*},\Tilde{\xi}\right)=\varepsilon^{2}\dell_{t}\phi_{*j}\left(y,t\right)+\lambda\tilde{\mathscr{E}}_{j}\left(\psi_{*j},\psi^{out}_{*},\Tilde{\xi}\right)+\gradperp_{y}\psi^{out}_{*}\cdot\nabla_{y}U\label{homotopic-operators-a-priori-estimates-1}\\
    &+\kappa_{j}\gradperp_{y}\left(\Gamma+\lambda\mathfrak{a}_{j}\right)\cdot\nabla_{y}\left(\phi_{*j}-\gamma\left(\Gamma+\lambda a_{*}\right)^{\gamma-1}_{+}\psi_{*j}\right)+\mathfrak{R}_{1}-\varepsilon\nabla_{y}\left(U+\phi_{j}\right)\cdot\left(\dot{\tilde{\xi}}_{j}+\mathcal{N}_{j}\left(\xi_{*}\right)\left[\tilde{\xi}\right]\right)\nonumber.
\end{align}
We can also compare the first two terms of \eqref{homotopic-operators-a-priori-estimates-1} to the first three terms on the right hand side of \eqref{homotopic-operator-inner-error-main-definition} and obtain that
\begin{align}
    &\mathfrak{R}_{1}=\kappa_{j}\gradperp_{y}\left(\Gamma+\lambda\mathfrak{a}_{j}\right)\cdot\nabla_{y}\left(\left(\Gamma+\lambda a_{*}\right)^{\gamma-1}_{+}\psi_{*j}\right)\label{homotopic-operators-a-priori-estimates-3a}\\
    &-\kappa_{j}\gradperp_{y}\left(\Gamma+\lambda\mathfrak{a}_{j}\right)\cdot\nabla_{y}\left(\left(\gamma\Gamma^{\gamma-1}_{+}+\gamma\left(\gamma-1\right)\Gamma^{\gamma-2}_{+}\left(\frac{\lambda\phi_{j1}}{\gamma\Gamma^{\gamma-1}_{+}}\right)\right)\psi_{*j}\right)\nonumber\\
    &+\kappa_{j}\lambda\gradperp_{y}\left(\mathfrak{a}_{j}-\frac{\phi_{j1}}{\gamma\Gamma^{\gamma-1}_{+}}\right)\cdot\nabla_{y}\left(\gamma\Gamma^{\gamma-1}_{+}\psi_{*j}\right)\nonumber\\
    &+\kappa_{j}\lambda\gradperp_{y}\left(\mathfrak{a}_{j}\right)\cdot\nabla_{y}\left(\gamma\left(\gamma-1\right)\Gamma^{\gamma-2}_{+}\left(\frac{\phi_{j1}}{\gamma\Gamma^{\gamma-1}_{+}}\right)\psi_{*j}\right)+\lambda\gradperp_{y}\psi_{*j}\cdot\nabla_{y}\phi_{j2}.\nonumber
\end{align}
Recalling that we can represent $\phi_{j1}$ using \eqref{phi-j1-l-formula}, upon Taylor expanding the term multiplying $\psi_{*j}$ under the gradient on the first term of the right hand side of \eqref{homotopic-operators-a-priori-estimates-3a}, we get some cancellation that gives us $\mathfrak{R}_{1}$ is either quadratic in unknowns or linear in unknowns with coefficients sufficiently small.
If we define the error $\mathfrak{E}_{\lambda}$ as being 
\begin{align}
    \mathfrak{E}_{\lambda}=\lambda\tilde{\mathscr{E}}_{j}\left(\psi_{*j},\psi^{out}_{*},\Tilde{\xi}\right)+\mathfrak{R}_{1}-\varepsilon\nabla_{y}\left(U+\phi_{j}\right)\cdot\left(\dot{\tilde{\xi}}_{j}+\mathcal{N}_{j}\left(\xi_{*}\right)\left[\tilde{\xi}\right]\right)+\gradperp_{y}\psi^{out}_{*}\cdot\nabla_{y}U,\label{homotopic-operators-a-priori-estimates-6}
\end{align}
then $\mathfrak{E}_{\lambda}$ satisfies the assumption on its support \eqref{linear-transport-operator-error-solution-supports}, as well as the bound \eqref{L2-interior-a-priori-estimate-lemma-error-bound-statement} for $\left(\psi_{*j},\psi^{out}_{*},\Tilde{\xi}\right)$ sufficiently small. Then from the information on the support of $\mathfrak{E}_{\lambda}$ defined in \eqref{homotopic-operators-a-priori-estimates-6}, we also obtain that the solution $\phi_{*j}$ satisfies \eqref{linear-transport-operator-error-solution-supports}.

\medskip
Thus, with the assumptions on $\left(\psi_{*j},\psi^{out}_{*},\Tilde{\xi}\right)$ we have made, \eqref{tilde-phi-star-R-initial-value-problem-equation} can be written in the form 
\begin{align}
    &\varepsilon^{2}\dell_{t}\phi_{*j}\left(y,t\right)++\kappa_{j}\gradperp_{y}\left(\Gamma+\lambda\mathfrak{a}_{j}\right)\cdot\nabla_{y}\left(\phi_{*j}-\gamma\left(\Gamma+\lambda a_{*}\right)^{\gamma-1}_{+}\psi_{*j}\right)\nonumber\\
    &+\mathfrak{E}_{\lambda}=c_{j0}\left(t\right)U(y)+c_{j1}\left(t\right)\dell_{1}U(y)+c_{j2}\left(t\right)\dell_{2}U(y)\nonumber
\end{align}
with all the assumptions of Lemma \ref{projected-linear-transport-problem-coefficients-of-projection-a-priori-estimates-lemma}, and thus we have the following estimates
\begin{align*}
    \|\phi_{*j}\|_{L^{2}\left(\mathbb{R}^{2}\right)}\leq \frac{C\varepsilon^{3-\sigma_{3}}}{t^{\frac{3}{2}}},\quad \|\phi_{*j}\|_{L^{\infty}\left(\mathbb{R}^{2}\right)}\leq \frac{C\varepsilon^{1-\sigma_{3}}}{t^{\frac{1}{2}-\sigma_{3}}},
\end{align*}
for all $t\in\left[T_{0},T\right]$ and some small $\sigma_{3}>0$.

\subsection{Fixed Point Formulation and Solution on $\left[T_{0},T\right]$}\label{fixed-point-formulation-section}
We reformulate the system \eqref{psi-star-r-green-function-representation-upper-half-plane}, \eqref{tilde-phi-star-R-initial-value-problem}, \eqref{cRj-initial-value-problem}, \eqref{psi-out-star-boundary-value-problem-equation} as a fixed point problem of the form \eqref{full-solution-construction-fixed-point-formulation} for a well chosen Banach space $\mathfrak{X}$. Recall that $\phi_{*L}$ is an odd reflection of $\phi_{*j}$ in the $x_{1}$ variable, and thus it is enough to solve the above problem.

\medskip
Fix a small number $\beta\in\left(0,1/2\right)$. For $\phi_{*j}\left(y,t\right)$ defined on $\mathbb{R}^{2}\times\left[T_{0},T\right]$, we consider the set of functions with support on $B_{2}\left(0\right)$ with finite $\left\|\cdot\right\|_{i}$ norm, where 
\begin{align*}
    \left\|\phi\right\|_{i}\coloneqq \sup_{\left[T_{0},T\right]}\left(t^{\frac{3}{2}}\left\|\phi\right\|_{L^{2}\left(B_{2}\left(0\right)\right)}\right)+\sup_{\left[T_{0},T\right]}\left(\varepsilon^{2}t^{\frac{1}{2}-\beta}\left\|\phi\right\|_{L^{\infty}\left(B_{2}\left(0\right)\right)}\right).
\end{align*}
Next we define the $\left\|\cdot\right\|_{o}$ norm for $\psi_{*}^{out}\left(x,t\right)$ defined on $\mathbb{R}^{2}\times\left[T_{0},T\right]$:
\begin{align*}
    \left\|\psi\right\|_{o}\coloneqq \sup_{\left[T_{0},T\right]}\left(t^{\frac{5}{2}}\|\left(\log{\left(\left|\cdot\right|+2\right)}\right)^{-1}\psi\|_{L^{\infty}\left(\mathbb{R}^{2}\right)}\right)+\sup_{\left[T_{0},T\right]}\left(t^{\frac{5}{2}}\left\|\nabla\psi\right\|_{L^{\infty}\left(\mathbb{R}^{2}\right)}\right).
\end{align*}
For $\psi_{*j}$ related to $\phi_{*j}$ via \eqref{psi-star-r-green-function-representation-upper-half-plane}, we define $\mathfrak{X}$ as the Banach space of all points $\mathfrak{P}=\left(\phi_{*j},\psi_{*}^{out},\tilde{\xi}\right)$ such that $\nabla_{y}\psi_{*j}$, $\nabla_{y}\psi_{*}^{out}$, $\dot{\tilde{\xi}}$ all exist and are continuous, with $\tilde{\xi}=\left(\tilde{\xi}_{1},\tilde{\xi}_{2},\tilde{\xi}_{3}\right)$, satisfying
\begin{align*}
    \left\|\mathfrak{P}\right\|_{\mathfrak{X}}\coloneqq\sum_{j=1}^{3}\|\phi_{*j}\|_{i}+\left\|\psi_{*}^{out}\right\|_{o}+\sum_{j=1}^{3}\left(\|t^{\frac{3}{2}}\tilde{\xi}_{j}\|_{\left[T_{0},T\right]}+\|t^{\frac{5}{2}}\dot{\tilde{\xi}}_{j}\|_{\left[T_{0},T\right]}\right)<\infty,
\end{align*}
with $\phi_{*j}$ also satisfying $\supp{\phi_{*j}}\subset B_{2}\left(0\right)$ and the orthogonality conditions \eqref{tilde-phi-star-R-0-mass-condition}--\eqref{tilde-phi-star-R-0-mass-condition}.

\medskip
Then in $\mathfrak{X}$, we define the open set $\mathscr{O}$ as the ``deformed ball" around $0$ in $\mathscr{X}$, that is all $\mathfrak{P}$ satisfying, for $j=1,2,3$,
\begin{align}
    \|\phi_{*j}\|_{i}<\varepsilon^{3-3\beta}, \quad \left\|\psi_{*}^{out}\right\|_{o}<\varepsilon^{5-3\beta},\quad \|t^{\frac{3}{2}}\tilde{\xi}_{j}\|_{\left[T_{0},T\right]}+\|t^{\frac{5}{2}}\dot{\tilde{\xi}}_{j}\|_{\left[T_{0},T\right]}<\varepsilon^{4-3\beta}.\label{deformed-open-ball-definition}
\end{align}
\begin{remark}
    We note that by the representation formula, $\phi_{*j}$ solving \eqref{tilde-phi-star-R-initial-value-problem} has support depending only on $U\left(y\right)$ and $\left(\Gamma+a_{*}\right)^{\gamma}_{+}$ for $a_{*}$ defined in \eqref{astar-definition}. It is then simple using \eqref{deformed-open-ball-definition} and the characteristics defined in \eqref{characteristics-ode-system} to show its support is uniformly contained in $B_{2}\left(0\right)$ for any $\lambda\in\left[0,1\right]$.
\end{remark}
Then the problem of finding a solution on $\left[T_{0},T\right]$ is reduced to restating the system \eqref{psi-star-r-green-function-representation-upper-half-plane}, \eqref{tilde-phi-star-R-initial-value-problem}, \eqref{cRj-initial-value-problem}, \eqref{psi-out-star-boundary-value-problem-equation} as a fixed point problem by finding an operator $\mathcal{T}$ such that solving the system is equivalent to solving $\mathcal{T}\left(\mathfrak{P},\lambda\right)=\mathfrak{P}$ for $\mathfrak{P}\in\mathscr{O}$, and for all $\lambda\in\left[0,1\right]$. This can be done by inverting the transport operator for \eqref{tilde-phi-star-R-initial-value-problem}, the Laplacian for \eqref{psi-out-star-boundary-value-problem-equation}, and by using Lemma \ref{projected-linear-transport-problem-coefficients-of-projection-a-priori-estimates-lemma} to see that \eqref{cRj-initial-value-problem} is satisfied if an ODE system for $\tilde{\xi}_{j}$, obtained by testing \eqref{tilde-phi-star-R-initial-value-problem} by $y_{1}$, and $y_{2}$ in turn, is satisfied. The first order differential operator acting on $\left(\tilde{\xi}_{1},\tilde{\xi}_{2},\tilde{\xi}_{3}\right)$ can be inverted using the analysis on the linearized operator around $\left(\xi_{*1},\xi_{*2},\xi_{*3}\right)$ in Section \ref{point-vortex-trajectory-section}. This is very similar to the set up constructed in \cite{DDMW2020}, and we refer to this work for more details.

\medskip
Then calling the inverted transport operator, inverted Laplacian, and inverted first order differential operator $\mathcal{T}^{in}_{\lambda}, \mathcal{T}^{out}_{\lambda}$, and $\mathcal{T}^{ode}_{\lambda}$ respectively, if we define the new operators, for $\phi_{*}=\left(\phi_{*1},\phi_{*2}.\phi_{*3}\right)$, and $\tilde{\xi}=\left(\tilde{\xi}_{1},\tilde{\xi}_{2},\tilde{\xi}_{3}\right)$,
\begin{align*}
    \tilde{\mathcal{T}}^{out}_{\lambda}\left(\phi_{*},\psi_{*}^{out},\tilde{\xi}\right)&=\mathcal{T}^{out}_{\lambda}\left(\mathcal{T}^{in}_{\lambda}\left(\phi_{*},\psi_{*}^{out},\tilde{\xi}\right),\psi_{*}^{out},\tilde{\xi}\right)\\
    \tilde{\mathcal{T}}^{ode}_{\lambda}\left(\phi_{*},\psi_{*}^{out},\tilde{\xi}\right)&=\mathcal{T}^{ode}_{\lambda}\left(\mathcal{T}^{in}_{\lambda}\left(\phi_{*},\psi_{*}^{out},\tilde{\xi}\right),\psi_{*}^{out},\tilde{\xi}\right),
\end{align*}
we see that the system consisting of \eqref{psi-star-r-green-function-representation-upper-half-plane}, \eqref{tilde-phi-star-R-initial-value-problem}, \eqref{cRj-initial-value-problem}, \eqref{psi-out-star-boundary-value-problem-equation} on the deformed open ball $\mathscr{O}$ defined in \eqref{deformed-open-ball-definition} is equivalent to the fixed point problem
\begin{align}
    \mathfrak{P}=\tilde{\mathcal{T}}_{\lambda}\left(\mathfrak{P}\right),\quad \mathfrak{P}\in\mathscr{O},\label{fixed-point-problem-final}
\end{align}
where
\begin{align}
    \tilde{\mathcal{T}}_{\lambda}\left(\mathfrak{P}\right)=\left(\mathcal{T}^{in}_{\lambda}\left(\mathfrak{P}\right),\tilde{\mathcal{T}}^{out}_{\lambda}\left(\mathfrak{P}\right),\tilde{\mathcal{T}}^{ode}_{\lambda}\left(\mathfrak{P}\right)\right).\label{fixed-point-operator}
\end{align}
\begin{lemma}
    The operator $\tilde{\mathcal{T}}\colon\mathscr{O}\times\left[0,1\right]\rightarrow\mathfrak{X}$ defined by $\tilde{\mathcal{T}}\left(\cdot,\lambda\right)\coloneqq\tilde{\mathcal{T}}_{\lambda}$ in \eqref{fixed-point-operator} is compact, and for all $\lambda\in\left[0,1\right]$, is a degree $1$ operator.
\end{lemma}
\begin{proof}
    Since we are trying to construct a solution to \eqref{fixed-point-problem-final} on a finite time interval $\left[T_{0},T\right]$, the proof is once again analogous to the proof given in \cite{DDMW2020}. So we have compactness of the operator as well as the desired information on the degree for all $\lambda\in\left[0,1\right]$. 
\end{proof}
The degree of the operator at $\lambda=1$ exactly means we have a unique solution in $\mathscr{O}$ to \eqref{psi-star-r-green-function-representation-upper-half-plane}, \eqref{tilde-phi-star-R-initial-value-problem}, \eqref{cRj-initial-value-problem}, \eqref{psi-out-star-boundary-value-problem-equation} at $\lambda=1$, and therefore to \eqref{2d-euler-vorticity-stream} with the desired estimates in $\varepsilon$ and $t$ on $\left[T_{0},T\right]$, for all $T$ large enough, as desired.

\section{Proof of Theorem \ref{teo1}}\label{conclusion}
We have a sequence of solutions to \eqref{2d-euler-vorticity-stream} that is identified with the sequence of solutions with estimates coming from \eqref{deformed-open-ball-definition} on $\left[T_{0},T\right]$, $\mathfrak{P}^{T}=\left(\phi_{*}^{T},\psi_{*}^{out,T},\tilde{\xi}^{T}\right)$ to \eqref{psi-star-r-green-function-representation-upper-half-plane}, \eqref{tilde-phi-star-R-initial-value-problem}, \eqref{cRj-initial-value-problem}, \eqref{psi-out-star-boundary-value-problem-equation} at $\lambda=1$ that we constructed in Section \ref{fixed-point-formulation-section}. These solutions satisfy $\mathfrak{P}^{T}\left(T\right)=0$.

\medskip
To obtain the solution on $\left[T_{0},\infty\right)$ with the properties claimed in Theorem \ref{teo1}, we want to apply Arzelà–Ascoli. This is standard with the estimates in $t$ that we have on $\left[T_{0},T\right]$ for any $T$ large enough. See \cite{DDMP2VP2023} for a very similar argument with more details. In the end, we get a subsequence $\left(\phi_{*}^{T_{n}},\psi_{*}^{out,T_{n}},\tilde{\xi}^{T_{n}}\right)$ that tends to some limit $\left(\phi_{*},\psi_{*}^{out},\tilde{\xi}\right)$ locally uniformly in both $t$ and $x$ for each fixed $\varepsilon>0$ small enough and $T_{0}>0$ large enough that satisfies bounds
\begin{align*}
    &\|\phi_{*j}\|_{L^{\infty}\left(\mathbb{R}^{2}\right)}\leq \frac{C\varepsilon^{1-\sigma}}{t^{\frac{1}{2}-\sigma}},\quad \|\phi_{*j}\|_{L^{2}\left(\mathbb{R}^{2}\right)}\leq \frac{C\varepsilon^{3-\sigma}}{t^{\frac{3}{2}}},\quad \|t^{\frac{3}{2}}\tilde{\xi}\|_{\left[T_{0},\infty\right)}\leq \varepsilon^{4-\sigma},\quad \|t^{\frac{5}{2}}\dot{\tilde{\xi}}\|_{\left[T_{0},\infty\right)}\leq \varepsilon^{4-\sigma},
\end{align*}
some small $\sigma>0$. The bounds for $\psi_{*j}$ and $\psi_{*}^{out}$ can be inferred from the bounds above. This limit then gives us a solution to $\eqref{2d-euler-vorticity-stream}$ with the properties stated in Theorem \ref{teo1}, as desired.

\appendix
    \section{Linearisation Around the Power Nonlinearity Vortex}\label{appendix-section}
    Here we record some useful results about the linearised operator around $\Gamma$ defined in \eqref{power-semilinear-problem-R2}. 

    \begin{lemma}\label{dancer-yan-non-degeneracy-lemma}
        Suppose $\varphi\in L^{\infty}$ is a solution to
        \begin{align}
            \Delta\varphi+\gamma \Gamma^{\gamma-1}_{+}\varphi=\kappa \gamma \Gamma^{\gamma-1}_{+}\varphi\nonumber
        \end{align}
        on $\mathbb{R}^{2}$, with $\kappa\in[0,1)$ a constant. If $\kappa=0$, then $\varphi$ lives in the span of $\dell_{1}\Gamma$ and $\dell_{2}\Gamma$. If $\kappa>0$, then $\varphi\equiv0$.
    \end{lemma}
    The proof of this result (with minor modifications) can be found in \cite{danceryan}.

    \medskip
    Next, we look at the linear problem given by
    \begin{align}
        \Delta_{y}\psi+\gamma \Gamma^{\gamma-1}_{+}\psi=\gamma \Gamma^{\gamma-1}_{+}g,\label{vortex-linearised-equation}
    \end{align}
    where $g$ is a $C^{1}$ function. We assume that $g$ has no mode $0$ terms, and move to complex coordinates $y=re^{i\theta}$, so that
    \begin{align}
        \psi\left(r,\theta\right)&=\sum_{k\in\mathbb{Z}}\uppsi_{k}\left(r\right)e^{ik\theta},\label{vortex-linearised-equation-solution-fourier-expansion}\\
        g\left(r,\theta\right)&=\sum_{k\in\mathbb{Z}}\mathfrak{g}_{k}\left(r\right)e^{ik\theta}.\label{vortex-linearised-equation-error-fourier-expansion}
    \end{align}
    No mode $0$ term for $g$ means that $\mathfrak{g}_{0}\equiv0$, and we accordingly impose that $\uppsi_{0}\equiv0$. Then \eqref{vortex-linearised-equation} splits into infinitely many ODEs given by
    \begin{align}
        \mathscr{L}_{k}\left[\uppsi_{k}\right]\coloneqq\dell_{r}^{2}\uppsi_{k}+\frac{1}{r}\dell_{r}\uppsi_{k}-\frac{k^{2}}{r^{2}}\uppsi_{k}+\gamma\Gamma^{\gamma-1}_{+}\uppsi_{k}=\gamma\Gamma^{\gamma-1}_{+}\mathfrak{g}_{k}.\label{modal-odes}
    \end{align}
    For $k\neq0$, one can find a positive solution $\zeta_{k}$ to the equation $\mathscr{L}_{k}\left[\zeta_{k}\right]=0$, that behaves like
    \begin{align}
        \zeta_{k}=r^{\left|k\right|}\left(1+\bigO{\left(r^{2}\right)}\right),\ r\to0,\label{zeta-k-origin-behaviour}
    \end{align}
    For $k=\pm1$, we have that $\zeta_{k}=-\Gamma'\left(r\right)$, which by \eqref{Gamma-piecewise-definition}, decays like $r^{-1}$ as $r\to\infty$. For $\left|k\right|\geq2$, we instead have
    \begin{align}
        \zeta_{k}=r^{\left|k\right|}\left(1+\bigO{\left(r^{-2\left|k\right|}\right)}\right),\ r\to\infty,\label{zeta-k-infinity-behaviour}
    \end{align}
    Then, in the case that $\mathfrak{g}_{k}\sim r^{\left|k\right|}$ as $r\to0$, we have the following lemma on how $\uppsi_{k}$ behaves at the origin and as $r\to\infty$.
    \begin{lemma}\label{vortex-linearised-equation-fourier-coefficients-behaviour-lemma}
        Assume we have $\psi$ and $g$ as in \eqref{vortex-linearised-equation}--\eqref{vortex-linearised-equation-error-fourier-expansion}, with no mode $0$ terms. Also assume that $\mathfrak{g}_{k}\sim r^{\left|k\right|}$ as $r\to0$ for $\left|k\right|\geq 1$. Then, for $\left|k\right|\geq 2$, we have 
        \begin{align}
            &\uppsi_{k}\in L^{\infty}\left(\mathbb{R}^{2}\right),\quad \uppsi_{k}=r^{\left|k\right|}\left(1+o\left(1\right)\right),\ r\to0,\nonumber\\
            &\uppsi_{k}=\twopartdef{\bigO{\left(r^{2}\left|k\right|^{-2}\right),}}{r\to0,\ \left|k\right|\to\infty,}{\bigO{\left(r^{-\left|k\right|}\left|k\right|^{-2}\right)},}{r\to\infty,\ \left|k\right|\to\infty.}\label{vortex-linearised-equation-fourier-coefficients-behaviour-lemma-statement-1}
        \end{align}
        For $\left|k\right|=1$, we have
        \begin{align}
            \frac{1}{r+1}\uppsi_{k}\in L^{\infty}\left(\mathbb{R}^{2}\right),\quad \uppsi_{k}=\twopartdef{\bigO{\left(r^{3}\right)},}{r\to0,}{\bigO{\left(r\right)},}{r\to\infty.}\label{vortex-linearised-equation-fourier-coefficients-behaviour-lemma-statement-2}
        \end{align}
    \end{lemma}
    \begin{proof}
        Using variation of parameters, we can write $\uppsi_{k}$ in terms of $\zeta_{k}$ and $\mathfrak{g}_{k}$ explicitly. Taking into account the growth or decay of each $\zeta_{k}$ at either $0$ or $\infty$, we have
        \begin{align}
            \uppsi_{k}=\twopartdef{\zeta_{k}\int_{0}^{r}\frac{1}{s\zeta_{k}\left(s\right)^{2}}\int_{0}^{s}\zeta_{k}\left(\tau\right)\mathfrak{g}_{k}\left(\tau\right)\gamma\Gamma^{\gamma-1}_{+}\left(\tau\right)\tau\ d\tau\ ds,}{\left|k\right|=1,}{\zeta_{k}\int_{r}^{\infty}\frac{1}{s\zeta_{k}\left(s\right)^{2}}\int_{0}^{s}\zeta_{k}\left(\tau\right)\mathfrak{g}_{k}\left(\tau\right)\gamma\Gamma^{\gamma-1}_{+}\left(\tau\right)\tau\ d\tau\ ds,}{\left|k\right|\geq2.}\label{variation-of-parameters}
        \end{align}
        The properties \eqref{vortex-linearised-equation-fourier-coefficients-behaviour-lemma-statement-1}--\eqref{vortex-linearised-equation-fourier-coefficients-behaviour-lemma-statement-2} can then be inferred from \eqref{variation-of-parameters} along with the additional fact that
        \begin{align*}
            \int_{0}^{s}\zeta_{k}\left(\tau\right)\mathfrak{g}_{k}\left(\tau\right)\gamma\Gamma^{\gamma-1}_{+}\left(\tau\right)\tau\ d\tau=\int_{0}^{1}\zeta_{k}\left(\tau\right)\mathfrak{g}_{k}\left(\tau\right)\gamma\Gamma^{\gamma-1}_{+}\left(\tau\right)\tau\ d\tau.
        \end{align*}
        for all $s\geq1$, due to the fact that $\supp{\Gamma^{\gamma-1}_{+}}\subset B_{1}\left(0\right)$.
    \end{proof}
    Using the above results, we define $\varrho_{k}\left(r\right)$ such that
    \begin{align}
        \mathfrak{g}_{k}\left(r\right)=r^{\left|k\right|}.\label{vortex-linearised-equation-rk-error-solution}
    \end{align}
    
	\bigskip\noindent
	{\bf Acknowledgements:}
	J.~D\'avila has been supported  by  a Royal Society  Wolfson Fellowship, UK, Grant
RSWF/FT/191007. M.~del Pino has been supported by the Royal Society Research Professorship grant RP-R1-180114 and by  the  ERC/UKRI Horizon Europe grant  ASYMEVOL, EP/Z000394/1. M. Musso has been supported by EPSRC research Grant EP/T008458/1. 
    S. Parmeshwar has been supported by EPSRC research Grants EP/T008458/1 and EP/V000586/1, and by the CY Initiative of Excellence (Grant
``Investissements d'Avenir'' ANR-16-IDEX-0008).

\bibliography{main} 
\bibliographystyle{siam}
\end{document}